%% file: linest_preprint.tex
\newcount\Comments  
\documentclass[letterpaper,10pt]{scrartcl}

\usepackage[utf8]{inputenc}
\usepackage[english]{babel}
\usepackage[automark]{scrpage2}
\usepackage{amsmath,amssymb,upref}
\usepackage{amsfonts}
\usepackage{amsthm,array,makecell}
\usepackage{color}
\usepackage{graphicx}
\usepackage{sidecap}
\usepackage{multirow}
\usepackage{booktabs}
\usepackage{fullpage}

\usepackage[font=small]{caption}
\usepackage{subcaption}

\usepackage{longtable}
\usepackage{rotating}

\usepackage{algorithm}
\usepackage{algcompatible}
\extrafloats{100}
\usepackage{lineno}
\usepackage[hidelinks]{hyperref}

\newcommand{\R}{\mathbb{R}}

\newcommand{\bfx}{\boldsymbol x}

\newcommand{\bff}{\boldsymbol f}
\newcommand{\bfg}{\boldsymbol g}

\newcommand{\Gcal}{\mathcal{G}}

\newcommand{\Tcal}{\mathcal{T}}

\newcommand{\Vcal}{\mathcal{V}}

\newcommand{\bfb}{\boldsymbol b}
\newcommand{\bfc}{\boldsymbol c}

\newcommand{\bfo}{\boldsymbol o}

\newcommand{\bfr}{\boldsymbol r}
\newcommand{\bfu}{\boldsymbol u}

\newcommand{\bfv}{\boldsymbol v}
\newcommand{\bfw}{\boldsymbol w}
\newcommand{\bfz}{\boldsymbol z}

\newcommand{\PDEdomain}{\Omega}
\newcommand{\ROMOutput}{\tilde{y}}
\newcommand{\StErrBnd}{\Delta^{\bfw}}
\newcommand{\OutErrBnd}{\Delta^{y}}

\newcommand{\bfA}{\boldsymbol A}
\newcommand{\bfB}{\boldsymbol B}
\newcommand{\bfC}{\boldsymbol C}
\newcommand{\bfD}{\boldsymbol D}

\newcommand{\bfF}{\boldsymbol F}
\newcommand{\bfG}{\boldsymbol G}

\newcommand{\bfI}{\boldsymbol I}

\newcommand{\bfK}{\boldsymbol K}

\newcommand{\bfM}{\boldsymbol M}

\newcommand{\bfR}{\boldsymbol R}

\newcommand{\bfV}{\boldsymbol V}

\newcommand{\bfQ}{\boldsymbol Q}

\newcommand{\bfW}{\boldsymbol W}

\newcommand{\bfDtrain}{\bfD^{\text{train}}}
\newcommand{\bfGtrain}{\bfG^{\text{train}}}
\newcommand{\bfgtrain}{\bfg^{\text{train}}}
\newcommand{\bfGtest}{\bfG^{\text{test}}}
\newcommand{\bfgtest}{\bfg^{\text{test}}}
\newcommand{\bfGbasis}{\bfG^{\text{basis}}}
\newcommand{\bfGnorm}{\bfG^{\text{norm}}}
\newcommand{\bfgbasis}{\bfg^{\text{basis}}}
\newcommand{\bfWbasis}{\bfW^{\text{basis}}}
\newcommand{\bfwbasis}{\bfw^{\text{basis}}}

\newcommand{\bfwtrain}{\bfw^{\text{train}}}
\newcommand{\rbfWtrain}{\rbfW^{\text{train}}}
\newcommand{\rbfwtrain}{\rbfw^{\text{train}}}
\newcommand{\bfWtest}{\bfW^{\text{test}}}
\newcommand{\bfwtest}{\bfw^{\text{test}}}
\newcommand{\tbfWtest}{\tbfW^{\text{test}}}
\newcommand{\tbfwtest}{\tbfw^{\text{test}}}

\newcommand{\bbfWtest}{\breve{\bfW}^{\text{test}}}
\newcommand{\hatbfWtest}{\hat{\bfW}^{\text{test}}}
\newcommand{\tbfrtest}{\tilde{\bfr}^{\text{test}}}
\newcommand{\hatbfrtest}{\hat{\bfr}^{\text{test}}}
\newcommand{\rbfRtrain}{\rbfR^{\text{train}}}
\newcommand{\rbfrtrain}{\rbfr^{\text{train}}}
\newcommand{\bfrtest}{\bfr^{\text{test}}}
\newcommand{\bfPsitrain}{\bfPsi^{\text{train}}}

\newcommand{\ubasis}{u^{\text{basis}}}
\newcommand{\utest}{u^{\text{test}}}
\newcommand{\bfutest}{\bfu^{\text{test}}}
\newcommand{\bfubasis}{\bfu^{\text{basis}}}
\newcommand{\bfutrain}{\bfu^{\text{train}}}
\newcommand{\ErrPerf}{e}
\newcommand{\ErrPerfBnd}{\Delta}

\DeclareMathOperator{\vech}{vech}
\DeclareMathOperator{\vect}{vec}
\DeclareMathOperator{\diag}{diag}

\newcommand{\ProbLB}{P^{\text{LB}}}

\newcommand{\nh}{N}
\newcommand{\nr}{n}

\newcommand{\bftheta}{\boldsymbol{\theta}}
\newcommand{\bfPsi}{\boldsymbol{\Psi}}

\newcommand{\rbfR}{\bar{\bfR}}
\newcommand{\rbfr}{\bar{\bfr}}

\newcommand{\hbfM}{\hat{\bfM}}

\newcommand{\hbfA}{\hat{\bfA}}
\newcommand{\hbfB}{\hat{\bfB}}

\newcommand{\tbfA}{\tilde{\bfA}}
\newcommand{\tbfB}{\tilde{\bfB}}

\newcommand{\rbfw}{\bar{\bfw}}
\newcommand{\rbfW}{\bar{\bfW}}

\newcommand{\hbfo}{\hat{\bfo}}

\newcommand{\tbfw}{\tilde{\bfw}}

\newcommand{\tbfW}{\tilde{\bfW}}

\newcommand{\bfZ}{\boldsymbol{Z}}

\newcommand{\bfTheta}{\boldsymbol \Theta}

\newcommand{\kibitz}[2]{\ifnum\Comments=1\textcolor{#1}{#2}\fi}

\newenvironment{keywords}%
   {\begin{trivlist}\item[]{\bfseries\sffamily Keywords:}\ }
   {\end{trivlist}}

\ihead[]{}
\ohead[]{}
\ifoot[]{}
\ofoot[]{}

\newtheorem{theorem}{Theorem}

\newtheorem{definition}[theorem]{Definition}

\newtheorem{lemma}[theorem]{Lemma}

\newtheorem{proposition}[theorem]{Proposition}

\newtheorem{remark}[theorem]{Remark}

\numberwithin{equation}{section}
\renewcommand{\theequation}{\arabic{section}.\arabic{equation}}

\title{Probabilistic error estimation for non-intrusive reduced models learned from data of systems governed by linear parabolic partial differential equations}

\author{Wayne Isaac Tan Uy and Benjamin Peherstorfer\thanks{\{wayne.uy,pehersto\}@cims.nyu.edu, Courant Institute of Mathematical Sciences, New York University, New York, NY 10012}}

\begin{document}

\maketitle

\begin{abstract}
This work derives a residual-based \emph{a posteriori} error estimator for reduced models learned with non-intrusive model reduction from data of high-dimensional systems governed by linear parabolic partial differential equations with control inputs.
It is shown that quantities that are necessary for the error estimator can be either
obtained exactly as the solutions of least-squares problems in a non-intrusive way from data such as initial conditions, control inputs, and high-dimensional solution trajectories or bounded in a probabilistic sense.
The computational procedure follows an offline/online decomposition. In the offline (training) phase, the high-dimensional system is judiciously solved in a black-box fashion to generate data and to set up the error estimator. In the online phase, the estimator is used to bound the error of the reduced-model predictions for new initial conditions and new control inputs without recourse to the high-dimensional system.
Numerical results demonstrate the workflow of the proposed approach from data to reduced models to certified predictions.
\end{abstract}


\section{Introduction}
Model reduction constructs reduced models that rapidly approximate solutions of differential equations by solving in problem-dependent, low-dimensional subspaces of classical, high-dimensional (e.g., finite-element) solution spaces \cite{RozzaPateraSurvey,paper:BennerGW2015,Quarteroni2011,book:HesthavenRS2016,doi:10.1080/00207170410001713448}. Traditional model reduction methods typically are intrusive in the sense that full knowledge about the underlying governing equations and their discretizations are required to derive reduced models. In contrast, this work considers non-intrusive model reduction that aims to learn reduced models from data with only little knowledge about the governing equations and their discretizations. However, constructing reduced models is only one aspect of model reduction. Another aspect is deriving \emph{a posteriori} error estimators that bound the error of reduced-model predictions with respect to the high-dimensional solutions that are obtained numerically with, e.g., finite-element methods  \cite{paper:PrudhommeRVMMPT2001,COCV_2002__8__1007_0,veroy_posteriori_2003,doi:10.1002/fld.867,paper:GreplP2005,paper:HaasdonkO2011,HaasdonkError}. This work builds on \emph{a posteriori} error estimators \cite{paper:GreplP2005,paper:HaasdonkO2011} from intrusive model reduction to establish error estimation for reduced models that are learned with non-intrusive methods. The key contribution is to show that all quantities required for deriving the error estimator can be either obtained in a non-intrusive way via least-squares regression from initial conditions, control inputs, and solution trajectories or bounded in a probabilistic sense, if the system of interest is known to be governed by a linear parabolic partial differential equation (PDE) with control inputs.
The key requirement to make the estimator practical is that the high-dimensional system is queryable in the sense that during a training (offline) phase one has access to a black box that one can feed with initial conditions and inputs and that returns the corresponding numerical approximations of the high-dimensional solution trajectories. If one considers learning reduced models from data as a machine learning task, then the proposed error estimator can be considered as pre-asymptotic computable generalization bound \cite{FoundationsML} of the learned models because the proposed estimator provides an upper bound on the error of the reduced model for initial conditions and inputs that have not been seen during learning (training) the reduced model and the error-estimator quantities. The bound is pre-asymptotic with respect to the number of data points and the dimension of the reduced model.

We review literature on non-intrusive and data-driven model reduction. First, the systems and control community has developed methods for identifying dynamical systems from frequency-response or impulse-response data, e.g., the Loewner approach by Antoulas and collaborators \cite{paper:AntoulasGI2016,paper:GoseaA2018,paper:IonitaA2014}, vector fitting \cite{772353,paper:DrmacGB2015}, and eigensystem realization \cite{doi:10.2514/3.20031,doi:10.1137/17M1137632}. In contrast, our approach will learn from time-domain data; not necessarily impulse-response data.
Second, dynamic mode decomposition~\cite{SchmidDMD,FLM:7843190,FLM:6837872,Tu2014391,NathanBook} has been shown to successfully derive linear dynamical systems that best fit data in the $L_2$ norm. However, the authors are unaware of error estimators for models derived with dynamic mode decomposition. Third, there is operator inference \cite{paper:PeherstorferW2016} that coincides with dynamic mode decomposition in case of linear systems but is also applicable to data from systems with nonlinear terms; see also the work on lift \& learn for general nonlinear systems \cite{QIAN2020132401} and the work on dynamic reduced models \cite{pehersto15dynamic}. The error estimators proposed in the following will build on operator inference for non-intrusive model reduction because, together with a particular data-sampling scheme \cite{paper:Peherstorfer2019}, operator inference exactly recovers the reduced models that are obtained via traditional intrusive model reduction. Thus, the learned models are the traditional reduced models with well-studied properties known from intrusive model reduction.

We now review literature on error estimators developed for intrusive model reduction. First, the reduced-basis community has developed error estimators for elliptic PDEs \cite{paper:PrudhommeRVMMPT2001} and parabolic PDEs \cite{paper:GreplP2005} with affine parameter dependence, time-dependent viscous Burgers' equation \cite{paper:NguyenRP2009,paper:JanonNP2013}, and linear evolution equations \cite{haasdonk_ohlberger_2008,paper:HaasdonkO2011}, among others. For systems that are nonlinear and/or have non-affine parameter dependence, error bounds have been established for reduced models with empirical interpolation in, e.g.,  \cite{paper:EftangGP2010,paper:HaasdonkOR2008,paper:ChaturantabutS2012,paper:WirtzSH2014}. These error estimators typically depend on the dual norm of the reduced-model residual and on other quantities of the underlying PDE discretizations such as coercivity and inf-sup stability constants \cite{paper:HuynhRSP2007} that require knowledge about the weak form of the governing equations that are unavailable in the setting of non-intrusive model reduction where one has access to data alone.
The work \cite{paper:SmetanaZP2019} proposes a probabilistic error bound involving randomized residuals which overcomes the need to compute constants in the error estimators; however, the reduced models are constructed with traditional intrusive model reduction and, in particular, residuals are computed in an intrusive way which conflicts with non-intrusive model reduction.
In the systems and control community, the discrepancy between the high-dimensional solutions of systems of ordinary differential equations and reduced-model solutions is bounded in terms of the transfer functions, see, e.g., \cite{paper:FengAB2017,doi:10.1137/140998603}.

This manuscript is organized as follows: Section~\ref{sec:Prelim} outlines preliminaries on spatial and temporal discretization of linear parabolic PDEs and intrusive model reduction. Section~\ref{sec:CertifiedROM} describes the proposed error estimator for reduced models learned with operator inference from data. First, least-squares problems are derived to infer residual-norm operators from data. Second, constants required for error estimation are bounded in a probabilistic sense. These two novel components are combined together with an intrusive error estimator \cite{paper:HaasdonkO2011} into a computational procedure that realizes the full workflow from data to reduced models to certification of reduced-model predictions, under certain conditions that are made precise. Numerical results are presented in Section~\ref{sec:NumEx} and conclusions are drawn in Section~\ref{sec:Concl}.

\section{Preliminaries} \label{sec:Prelim}

Section~\ref{subsec:ParabolicPDE} reviews linear parabolic PDEs with spatial and time discretization discussed in Sections~\ref{subsec:spatialDiscr} and~\ref{subsec:timeDiscr}, respectively. The continuous-time problem is transformed into a discrete linear time-invariant system. Intrusive model reduction is then recalled in Section~\ref{subsec:IntMR}.  Section~\ref{subsec:ProbForm} outlines the problem formulation.

\subsection{Linear parabolic PDEs with time-independent coefficients} \label{subsec:ParabolicPDE}

Let $\PDEdomain \subset \mathbb{R}^m$ be a bounded domain and let $\Tcal = (0, T)$ be a time interval with $T>0$ fixed. Consider the linear parabolic PDE on $(\bfx,t) \in \PDEdomain \times \Tcal$ given by
\begin{align} \label{eq:linearParabolicPDE}
\frac{\partial}{\partial t}w(\bfx,t) & =  \nabla \cdot (\bfb^T (\bfx) \nabla w(\bfx,t)) - \bfc(\bfx) \cdot \nabla w(\bfx,t) - d(\bfx) w(\bfx,t) + R(\bfx), \\
w(\bfx,t) & = 0 \text{\,\,\, for \,\,\,} \bfx \in \Gamma_D, \notag \\
\bfb^T(\bfx) \nabla w(\bfx,t) \cdot \mathbf{n} & = u_j(t) \text{\,\,\, for \,\,\,} \bfx \in \Gamma^j_N, \,\, j = 1,\dots,n_{\Gamma}, \notag \\
w(\bfx,0) & = h(\bfx), \notag
\end{align}
where $w: \Omega \times \Tcal \to \mathbb{R}$ is the solution,  $\bfb: \PDEdomain \rightarrow \R^{m \times m},  \bfc: \PDEdomain \rightarrow \R^m, d: \PDEdomain \rightarrow {\R}$ are time-independent coefficients, $R : \PDEdomain \to \mathbb{R}$ is the source term and the boundary $\partial \PDEdomain$ is decomposed into the $n_{\Gamma}$ disjoint segments $\cup_{j=1}^{n_{\Gamma}} \Gamma_N^j =\Gamma_N$ with Neumann conditions and the remaining portion $\Gamma_D$ with Dirichlet condition. The control inputs are $\{u_j(t)\}_{j=1}^{n_{\Gamma}}$ for $t \in \Tcal$. Define $[\cdot]_i$ as the $i$-th component of a vector and $[\cdot]_{ij}$ as the $(i,j)$-th component of a matrix. Let further $[\bfb]_{ij}, [\bfc]_i,d \in L^{\infty}(\PDEdomain)$ for $i,j \in \{1,\dots,m \}$, $R,h \in L^2(\PDEdomain)$, and $u_j \in L^2(\Tcal)$ for $j=1,\dots,n_{\Gamma}$ where $L^2,L^{\infty}$ correspond to the space of square-integrable and essentially bounded measurable functions, respectively. For~\eqref{eq:linearParabolicPDE} to be parabolic, it is required that for any $\boldsymbol{\xi} \in \R^m$ and $\boldsymbol{x} \in \PDEdomain$, there exists a constant $\theta > 0$ such that $\boldsymbol{\xi}^T \bfb(\bfx) \boldsymbol{\xi} \ge \theta \|\boldsymbol{\xi}\|_2^2$ \cite{book:Evans2010}.

\subsection{Spatial discretization} \label{subsec:spatialDiscr}

For the Sobolev space $H^1(\Omega)$, define $\Vcal = \{v \in H^1(\PDEdomain) \, \big \vert \, v|_{\Gamma_D} = 0 \}$ which is equipped with the norm $\|\cdot\|_{\Vcal}$. We seek $w \in \Vcal$ such that
\begin{align}\label{eq:WeakForm}
\int_\PDEdomain v \frac{\partial}{\partial t} w \,d\bfx = -a(w,v) + f(v) \,\,\,\, \forall v \in \Vcal\,
\end{align}
where
\[
    a(w,v)  = \int_\PDEdomain \nabla v \cdot (\bfb^T \nabla w) \, d\bfx - \int_\PDEdomain v [  \nabla w  \cdot \bfc  + wd ] \,d \bfx \quad
    \]
and
\[
    f(v)  = \int_\PDEdomain v R  \,d \bfx
 + \sum_{j=1}^{n_{\Gamma}} u_j(t) \int_{\Gamma^j_N} v \, d\Gamma_N,
\]
see \cite{book:Thomee2006,book:Evans2010,book:HesthavenRS2016} for details. In the following, we assume that the bilinear form $a$ in \eqref{eq:WeakForm} is coercive and continuous, i.e., $\exists \,\, \alpha > 0$ and $\gamma < \infty$ for which $a(v,v) \ge \alpha \|v\|^2_{\Vcal}$ and $a(w,v) \le \gamma \|w\|_{\Vcal} \|v\|_{\Vcal}$ for $v, w \in \Vcal$ and $f$ in \eqref{eq:WeakForm} is continuous. To discretize~\eqref{eq:WeakForm}, consider a finite-dimensional approximation space $\Vcal_{\nh}\subset \Vcal$ with basis $\{\varphi_i\}_{i=1}^N$  such that for $w \in \Vcal_{\nh}$, $w(\bfx,t) = \sum_{i=1}^N w_i(t) \varphi_i (\bfx)$. Setting $v = \varphi_i, i= 1,\dots,N$ in~\eqref{eq:WeakForm} results in
\begin{align} \label{eq:WeakFormMatrix}
\bfM \frac{d \bfw(t)}{dt} = \bfK \bfw(t) + \bfF \bfu(t)
\end{align}
where $\bfw(t) = [w_1(t),\dots,w_{\nh}(t)]^T \in \R^{\nh}$, $\bfM \in \R^{N \times N}$ such that $[\bfM]_{ij} = \int_\PDEdomain \varphi_j \varphi_i \,d\bfx$, $\bfK \in \R^{N \times N}$ such that $[\bfK]_{ij}= -a(\varphi_j,\varphi_i)$, $\bfu(t) = [1, u_1(t),\dots,u_{n_{\Gamma}}(t)]^T   \in \R^{p} $ with $p = n_{\Gamma} + 1$, while
$$\bfF =  \begin{bmatrix}
\int_{\PDEdomain} \varphi_1 R \,d \bfx & \int_{\Gamma^1_{N}}  \varphi_1 \,d \Gamma_N & \dots & \int_{\Gamma^{n_{\Gamma}}_N}  \varphi_1 \,d\Gamma_N \\
\vdots & \vdots & \ddots & \vdots \\
\int_{\PDEdomain} \varphi_{\nh} R \,d \bfx & \int_{\Gamma^1_{N} } \varphi_N \,d \Gamma_N & \dots & \int_{\Gamma^{n_{\Gamma}}_N} \varphi_N \,d\Gamma_N \end{bmatrix} \in \R^{N \times p}.$$
If the source term $R=0$, $p = n_{\Gamma}$ and the resulting $\bfu(t),\bfF$ are obtained by truncating the first component of $\bfu(t)$ and the first column of $\bfF$ defined above.

\subsection{Time discretization} \label{subsec:timeDiscr}

To temporally discretize the time-continuous system~\eqref{eq:WeakFormMatrix}, let $\{t_k\}_{k=0}^K \subset \Tcal$ be equally spaced points with $t_{k+1}-t_k = \delta t$
and denote by $\bfw_k,\bfu_k$ the discrete time approximations to $\bfw(t_k),\bfu(t_k)$.
 A one-step scheme can be expressed as \begin{align} \label{eq:OneStepDisc}
\frac{\bfw_{k+1} - \bfw_k}{\delta t} = \beta \bfM^{-1} (\bfK \bfw_{k+1} + \bfF \bfu_{k+1}) + (1-\beta) \bfM^{-1} (\bfK \bfw_k + \bfF \bfu_k), \hspace{2em} \beta \in [0,1]
\end{align}
in which we recover the forward Euler, backward Euler, and Crank-Nicolson method with $\beta = 0$, $\beta = 1$, and $\beta = \frac{1}{2}$, respectively. We rewrite \eqref{eq:OneStepDisc} as
\begin{align} \label{eq:LTIsystem}
\bfw_{k+1} = \bfA \bfw_k + \bfB \bfg_{k+1}
\end{align}
with
\begin{align*}
    \bfA & =  (\bfI_N - \beta \delta t \bfM^{-1} \bfK)^{-1} (\bfI_N + (1-\beta) \delta t \bfM^{-1} \bfK), \\
    \bfB & =  (\bfI_N - \beta \delta t \bfM^{-1} \bfK)^{-1} \delta t \bfM^{-1} \bfF, \\
    \bfg_{k+1} & =  \beta \bfu_{k+1} + (1-\beta) \bfu_k,
\end{align*}
and the $N \times N$ identity matrix $\bfI_N$. Note that $\bfg_{k+1}=\bfu_k$ for $\beta = 0$ while $\bfg_{k+1} = \bfu_{k+1}$ for $\beta = 1$. We refer to $\bfW = [\bfw_0, \dots, \bfw_K]$ as a trajectory. We further define $\Gcal$ as the set of input trajectories $\bfG = [\bfg_1,\dots,\bfg_K] \in \mathbb{R}^{p \times K}$ of arbitrary but finite length $K$ so that $\sum_{k=1}^K [\bfg_k]_i^2 < \infty$ for $i=1,\dots,p$, i.e. each component of the discrete-time input has finite norm on the time interval $\Tcal$. Since $u_j \in L^2(\Tcal)$,  we only consider input trajectories $\bfG \in \Gcal$.

\subsection{Traditional (intrusive) model reduction} \label{subsec:IntMR}

Model reduction seeks an approximate solution to~\eqref{eq:LTIsystem} which lies in a low-dimensional subspace $\Vcal_{\nr}$ spanned by the columns of $\bfV_{\nr} = [\bfv_1,\dots,\bfv_{\nr}]\in \R^{\nh \times \nr}$ with $\nr \ll \nh$. Various approaches exist for constructing the low-dimensional subspace, see, e.g.,~\cite{RozzaPateraSurvey,paper:BennerGW2015,Quarteroni2011,book:HesthavenRS2016,doi:10.1080/00207170410001713448}. In the following, we use the proper orthogonal decomposition (POD) to construct $\bfV_{\nr}$. Let $[\bfw_0,\dots,\bfw_K]$  be the snapshot matrix whose columns are  the states $\bfw_k$. The basis $\bfV_n$ for $\Vcal_{\nr}$  is derived from the left singular vectors of the snapshot matrix corresponding to the $\nr$ largest singular values. Via Galerkin projection, the low-dimensional (reduced) system can then be derived as \begin{align} \label{eq:ROMsystem}
\tbfw_{k+1} &= \tbfA \tbfw_k + \tbfB \bfg_{k+1}
\end{align}
where
\begin{equation}
\tbfA = \bfV_{\nr} ^T \bfA \bfV_{\nr} \in \R^{\nr \times \nr}\,,\qquad \tbfB = \bfV_{\nr}^T \bfB \in \R^{\nr \times p}\,.
\label{eq:IntProj}
\end{equation}
The low-dimensional solution $\tbfw_k$ approximates the solution $\bfw_k$ to~\eqref{eq:LTIsystem} through $\bfV_{\nr} \tbfw_k$. We refer to $\tbfW = [\tbfw_0, \dots, \tbfw_{K-1}]$ as a reduced trajectory.

\subsection{Non-intrusive model reduction and problem formulation} \label{subsec:ProbForm}

Deriving reduced model \eqref{eq:ROMsystem} by forming the matrix-matrix products \eqref{eq:IntProj} of the basis matrix $\bfV_n$ and the operators $\bfA$ and $\bfB$ of the high-dimensional system is intrusive in the sense that $\bfA$ and $\bfB$ are required either in assembled form or implicitly through a routine that provides the action of $\bfA$ and $\bfB$ to a vector. In the following, we are interested in the situation where $\bfA$ and $\bfB$ are unavailable. Rather, we can simulate the high-dimensional system \eqref{eq:LTIsystem} at initial conditions and control inputs to generate state trajectories. Building on non-intrusive model reduction, we learn the reduced operators $\tbfA$ and $\tbfB$ from state trajectories without having $\bfA$ and $\bfB$ available. A major component of intrusive model reduction, besides constructing reduced models, is deriving error estimators that rigorously upper bound the approximation error of the reduced models with respect to the high-dimensional solutions \cite{paper:PrudhommeRVMMPT2001,COCV_2002__8__1007_0,veroy_posteriori_2003,doi:10.1002/fld.867,paper:GreplP2005,paper:HaasdonkO2011,HaasdonkError}. However, such error estimators typically depend on quantities such as norms of $\bfA$ and residuals that are unavailable in non-intrusive model reduction. Thus, error estimators developed for intrusive model reduction typically cannot be directly applied when reduced models are learned with non-intrusive model reduction methods.

\section{Certifying reduced models learned from data} \label{sec:CertifiedROM}

Our goal is two-fold: (i) learning the reduced operators \eqref{eq:IntProj} from state trajectories of the high-dimensional system and (ii) learning quantities to establish \emph{a posteriori} error estimators to rigorously bound the error $\|\bfw_k - \bfV_n\tbfw_k\|_2$ in the $2$-norm $\|\cdot\|_2$ of the reduced solution $\tbfw_k$ with respect to the high-dimensional solution $\bfw_k$ at time step $k$ for different initial conditions and different inputs than what was used during (i). The reduced operators and the quantities for the error estimators are learned under the setting that the high-dimensional operators in~\eqref{eq:LTIsystem} are unavailable in assembled and implicit form. We build on a non-intrusive approach for model reduction based on operator inference \cite{paper:PeherstorferW2016} and re-projection \cite{paper:Peherstorfer2019} and on an error estimator for linear evolution equations \cite{paper:HaasdonkO2011}. We show that the required quantities for the error estimator can be recovered from residual trajectories corresponding to training control inputs in a non-intrusive way similar to learning the reduced operators with operator inference and re-projection. These quantities then allow bounding the state error for other inputs and initial conditions.

Section~\ref{subsec:OpInfReproj} reviews operator inference with re-projection introduced in \cite{paper:Peherstorfer2019} and provides novel results on conditions which permit recovery of the reduced system operators. Section~\ref{subsec:HObound} discusses an error estimator from intrusive model reduction as presented in  \cite{paper:HaasdonkO2011}. To carry over the error estimator \cite{paper:HaasdonkO2011} to the non-intrusive model reduction case, an optimization problem is formulated in Section~\ref{subsec:StateErrBnd} whose unique solution leads to the required quantities for error estimation under certain conditions.  Sections~\ref{subsec:StateErrBndComput} and~\ref{subsec:Anorm} address prediction of the state \emph{a posteriori} error for other control inputs. The former utilizes a deterministic bound for the state error. In contrast, the latter offers a probabilistic error estimator whose reliability,  the probability of failure of the error estimator, can be controlled by the number of samples.  A summary of the proposed approach comprised of an offline (training) and online (prediction) phase is then given in Section~\ref{subsec:OffOn}.

\subsection{Recovering reduced models from data with operator inference and re-projection} \label{subsec:OpInfReproj}

Let $\bfV_n$ be the basis matrix with $n$ columns. Building on \cite{paper:PeherstorferW2016}, the work \cite{paper:Peherstorfer2019} introduces a re-projection scheme to generate the reduced trajectory $\tilde{\bfW} = [\tilde{\bfw}_0, \dots, \tilde{\bfw}_{K-1}]$  that would be obtained with the reduced model \eqref{eq:ROMsystem} as if it were available by querying the high-dimensional  system \eqref{eq:LTIsystem} alone with input trajectory $\bfG = [\bfg_1,\dots,\bfg_{K}]$. We define a queryable system as follows. \begin{definition} \label{defn:Queryable}
A system \eqref{eq:LTIsystem} is queryable if the trajectory $[\bfw_0, \dots, \bfw_K]$ with $K \geq 1$ can be computed for any initial condition $\bfw_0 \in \Vcal_{\nr}$ and any input trajectory $\bfG = [\bfg_1, \dots, \bfg_K] \in \Gcal$.
\end{definition}
For example, system \eqref{eq:LTIsystem} can be black-box and queryable in the sense that the operators $\bfA$ and $\bfB$ are unavailable but $\bfw_0$ and $\bfG$ can be provided to a black box to produce $\bfW$. In contrast, if there is a high-dimensional system for which a trajectory $\bfW$ for an input trajectory $\bfG$ is given, without being able to choose $\bfG$ and initial condition, then such a system is not queryable.

For a queryable system, the re-projection scheme alternates between time-stepping the high-dimensional system \eqref{eq:LTIsystem} and projecting the state onto the space $\Vcal_{\nr}$ spanned by the columns of $\bfV_n$. Let $\bfw_0 \in \Vcal_n$  be the initial condition and define $\rbfw_0 = \bfV_n^T\bfw_0$. The re-projection scheme takes a single time step with the high-dimensional system \eqref{eq:LTIsystem} with initial condition $\bfV_n\rbfw_0$ and control input $\bfg_1$ to obtain $\bfw_1$. The state $\bfw_1$ is projected to obtain $\rbfw_1 = \bfV_n^T\bfw_1$, and the process is repeated by taking a single time step with the high-dimensional system \eqref{eq:LTIsystem} with initial condition $\bfV_n\rbfw_1$ and control input $\bfg_2$.  It is shown in \cite{paper:Peherstorfer2019} that the re-projected trajectory $\rbfW = [\rbfw_0, \dots, \rbfw_{K-1}]$ is the reduced trajectory $\tbfW =  [\tbfw_0, \dots, \tbfw_{K-1}]$ in our case of a linear system \eqref{eq:LTIsystem}. Furthermore, the least-squares problem
\begin{align}\label{eq:OpInf}
\min_{\hbfA, \hbfB} \sum_{k=0}^{K-1} \left \Vert  \hbfA \rbfw_k  + \hbfB \bfg_{k+1} - \rbfw_{k+1}\right \Vert_2^2
\end{align}
has as the unique solution the reduced operators $\tilde{\bfA}$ and $\tilde{\bfB}$ if the data matrix
\begin{equation}
\bfPsi = \begin{bmatrix}\rbfW^T & \bfG^T \end{bmatrix} \in \R^{K \times (n+p)}
\label{eq:DataMatrix}
\end{equation}
has full rank and $K \geq n + p$; see Corollary~3.2 in \cite{paper:Peherstorfer2019} for more details.

The following proposition generalizes the least-squares problem \eqref{eq:OpInf} to trajectories from multiple initial conditions and shows that in this case there always exist initial conditions and input trajectories such that the unique solution of the corresponding least-squares problem is given by the reduced operators $\tbfA$ and $\tbfB$.

\begin{proposition}
There exist $\nr + p$ input trajectories $\bfG^{(1)}, \dots, \bfG^{(\nr + p)} \in \Gcal$, each of finite length $K_{\ell} \in \mathbb{N}$ for $\ell = 1, \dots, \nr + p$, and $\nr + p$ initial conditions $\bfw_0^{(1)}, \dots, \bfw_0^{(\nr + p)} \in \Vcal_{\nr}$ such that the generalized data matrix
\[
\boldsymbol\Phi=\begin{bmatrix} \rbfW^{(1)} & \dots & \rbfW^{(\nr+p)} \\ \bfG^{(1)} &  \dots & \bfG^{(\nr + p)}  \end{bmatrix}^T \in \R^{(\sum_{\ell = 1}^{\nr + p} K_{\ell}) \times (n+p)}
\]
with re-projected trajectories $\rbfW^{(\ell)} = [\rbfw^{(\ell)}_0,\dots,\rbfw^{(\ell)}_{K_{\ell}-1}] \in \R^{n \times K_{\ell}}$
has full rank, thereby guaranteeing the recovery of the reduced operators $\tbfA, \tbfB$ via least-squares regression.
\label{prop:InputsExist}
\end{proposition}
\begin{proof}
The generalized data matrix $\boldsymbol\Phi$ is induced by the  least squares problem \begin{align}\label{eq:OpInfLSMoreInput}
\min_{\hbfA, \hbfB} \sum_{\ell=1}^{\nr + p} \sum_{k=0}^{K_{\ell}-1} \left \Vert  \hbfA \rbfw^{(\ell)}_k  + \hbfB \bfg^{(\ell)}_{k+1} - \rbfw^{(\ell)}_{k+1}\right \Vert_2^2 \end{align}
which is an extension of the least squares problem~\eqref{eq:OpInf}
for the case when there are $\ell = 1, \dots, \nr + p$ initial conditions $\bfw_0^{(1)}, \dots, \bfw_0^{(\nr + p)}$ and input trajectories $\bfG^{(\ell)} = [\bfg^{(\ell)}_1,\dots,\bfg^{(\ell)}_{K_{\ell}}] \in \R^{p \times K_{\ell}}$. If $\boldsymbol\Phi$ is full rank, $\hbfA,\hbfB$ in~\eqref{eq:OpInfLSMoreInput} recover the reduced operators $\tbfA,\tbfB$ as discussed in \cite{paper:Peherstorfer2019}.

We now derive specific initial conditions and control inputs that lead to a full-rank $\boldsymbol\Phi$. First, we select $\nr$ linearly independent initial conditions $\bfw_0^{(1)}, \dots, \bfw_0^{(\nr)} \in \Vcal_{\nr}$, which exist because $\Vcal_{\nr}$ has $\nr$ dimensions. Correspondingly, $\rbfw_0^{(1)}, \dots, \rbfw_0^{(\nr)} \in \mathbb{R}^{\nr}$ are linearly independent as well. To see this, note that $\rbfw_0^{(i)} = \bfV_{\nr}^T\bfw_0^{(i)}$ holds for $i = 1, \dots, \nr$ and thus $\bfV_{\nr}[\rbfw_0^{(1)}, \dots, \rbfw_0^{(\nr)}] = [\bfw_0^{(1)}, \dots, \bfw_0^{(\nr)}]$ because $\bfw_0^{(1)}, \dots, \bfw_0^{(\nr)} \in \Vcal_{\nr}$. Because $\bfV_{\nr}$ has orthonormal columns, the rank of $[\rbfw_0^{(1)}, \dots, \rbfw_0^{(\nr)}] = \bfV_{\nr}^T[\bfw_0^{(1)}, \dots, \bfw_0^{(\nr)}]$ is equal to the rank of $\bfV_{\nr}(\bfV_{\nr}^T[\bfw_0^{(1)}, \dots, \bfw_0^{(\nr)}]) = [\bfw_0^{(1)}, \dots, \bfw_0^{(\nr)}]$, which is $\nr$. Set $\bfg_1^{(i)} = \boldsymbol 0_{p \times 1}$ for $i = 1, \dots, \nr$ where $\boldsymbol{0}_{m \times n}$ represents an $m \times n$ matrix of zeros. Second, set $\bfw_0^{(\nr + 1)} = \dots = \bfw_0^{(\nr + p)} = \boldsymbol 0_{\nh \times 1} \in \Vcal_{\nr}$ and select $p$ linearly independent control inputs $\bfg_1^{(\nr + 1)}, \dots, \bfg_1^{(\nr + p)} \in \Gcal$, which exist because $\mathbb{R}^p \subset \Gcal$   per definition; see Section~\ref{subsec:timeDiscr}. Taking these $\nr + p$ initial conditions and input signals and time-stepping with re-projection the high-dimensional system for a finite number of times steps leads to a generalized data matrix $\boldsymbol \Phi$ that contains at least the following rows
\[
\begin{bmatrix}
\rbfw_0^{(1)}\\
\bfg_1^{(1)}
\end{bmatrix}^T\,, \dots\,, \begin{bmatrix}
\rbfw_0^{(\nr + p)}\\
\bfg_1^{(\nr + p)}
\end{bmatrix}^T\,.
\]
The matrix $\boldsymbol \Phi$ therefore contains $\nr + p$ linearly independent rows and thus has full rank. Note that $K_{\ell} \geq 1$ for $\ell = 1, \dots, \nr + p$.
\end{proof}

\begin{remark}
Proposition~\ref{prop:InputsExist} considers trajectories from multiple initial conditions to show that initial conditions and input trajectories exist to recover the reduced model via operator inference and re-projection. To ease exposition, we build on the formulation with a single initial condition \eqref{eq:OpInf} in the following and in all our numerical results. However, the following results immediately generalize to the formulation with multiple initial conditions used in Proposition~\ref{prop:InputsExist}.
\end{remark}

\subsection{Error estimation for linear reduced models in intrusive model reduction} \label{subsec:HObound}

We now elaborate on an \emph{a posteriori} estimator for the state error in intrusive model reduction by following the presentation  by Haasdonk and Ohlberger \cite{paper:HaasdonkO2011}; note, however, that intrusive error estimation for reduced models of parabolic PDEs has been studied by Grepl and Patera in \cite{paper:GreplP2005} as well and the following non-intrusive approach may extend to their error estimators too. For $k \in \mathbb{N}$,  define the state error at time $k$ as $\bfw_k - \bfV_n \tbfw_k$  and the residual $\bfr_k$ as \begin{align} \label{eq:LTIResidual}
\bfr_{k+1} = \bfA \bfV_{\nr} \tbfw_k + \bfB \bfg_{k+1} - \bfV_{\nr} \tbfw_{k+1}.
\end{align}
The
state error is
\begin{align} \label{eq:ErrState}
    \bfw_k - \bfV_{\nr} \tbfw_k = \bfA^k (\bfw_0 - \bfV_{\nr} \tbfw_{0}) + \sum_{l=0}^{k-1} \bfA^{k-l-1} \bfr_{l+1}.
\end{align}
Define
\begin{align} \label{eq:aposterioribnd}
    \StErrBnd_k(c_0,\dots,c_k; \bfw_0,\bfG) = c_0 \|\bfw_0 - \bfV_{\nr} \tbfw_0\|_2 + \sum_{l=0}^{k-1} c_{l+1}\|\bfr_{l+1}\|_2
\end{align}
which relies on the initial condition $\bfw_0$, input trajectory $\bfG \in \Gcal$, and constants $c_0,\dots,c_k \in \mathbb{R}$. The norm of \eqref{eq:ErrState} is then bounded by
\begin{align} \label{eq:ErrStateBnd}
    \|\bfw_k - \bfV_n \tbfw_k\|_2 \leq \StErrBnd_k(\|\bfA^k\|_2,\dots,\|\bfA^0\|_2; \bfw_0,\bfG) =  \|\bfA^k\|_2 \|\bfw_0 - \bfV_n \tbfw_0\|_2 + \sum_{l=0}^{k-1} \|\bfA^{k-l-1}\|_2\|\bfr_{l+1}\|_2.
\end{align}
If $\max_{0 \le l \le k} \|\bfA^l\|_2 \leq C$
for a constant $C \in \mathbb{R}$, then the following holds
$$\|\bfw_k - \bfV_n \tbfw_k\|_2  \le \StErrBnd_k(\underbrace{C,\dots,C}_{k+1}; \bfw_0,\bfG).$$
The error $\bfw_0 - \bfV_0 \tbfw_0$ of the initial condition is the projection error $\bfw_0 - \bfV_n\bfV^T_n\bfw_0$ and can be computed if $\bfV_n$ and the initial condition $\bfw_0$ are known.

\subsection{Recovering the residual operators from residual trajectories} \label{subsec:StateErrBnd}

The residual norm $\|\bfr_k\|_2$ at time step $k$ is a critical component for the error estimator in \cite{paper:HaasdonkO2011}; directly computing $\|\bfr_k\|_2$ using formula \eqref{eq:LTIResidual} would require either the high-dimensional system operators $\bfA$ and $\bfB$ or querying the system \eqref{eq:LTIsystem} at each $\tbfw_k$. Following \cite{paper:HaasdonkO2011}, the squared residual norm is expanded as
\begin{equation}
\|\bfr_k\|^2_2 = \tbfw_k^T \bfM_1 \tbfw_k + \bfg_{k+1}^T \bfM_2 \bfg_{k+1} + 2\bfg_{k+1}^T\bfM_3 \tbfw_k
+ \tbfw_{k+1}^T \bfM_4 \tbfw_{k+1}
 - 2\tbfw_{k+1}^T\tbfA \tbfw_k -2\tbfw_{k+1}^T \tbfB \bfg_{k+1}
\label{eq:SqResNorm}
\end{equation}
with the matrices
\[
\bfM_1 = \bfV_n^T\bfA^T\bfA\bfV_n\,,\quad \bfM_2 = \bfB^T\bfB\,,\quad \bfM_3 = \bfB^T\bfA\bfV_n\,,
\]
and $\bfM_4 = \bfV_n^T\bfV_n$. Observe that after the reduced model has been obtained with operator inference and re-projection (Section~\ref{subsec:OpInfReproj}), the matrices $\tbfA, \tbfB$, and $\bfM_4$ can be readily computed without $\bfA$ and $\bfB$. Only matrices $\bfM_1, \bfM_2, \bfM_3$ are needed additionally to compute the squared residual norm with \eqref{eq:SqResNorm}.

Let $\rbfW = [\rbfw_0, \dots, \rbfw_{K-1}]$ be the re-projected trajectory using an input trajectory $\bfG$. Let further $\rbfR = [\rbfr_0, \dots, \rbfr_{K-1}]$ be the residual trajectory corresponding to the re-projected trajectory defined as
\[
\rbfr_k = \bfA\bfV_n\rbfw_k + \bfB\bfg_{k+1} - \bfV_n\rbfw_{k + 1}\,,
\]
following the residual expression in~\eqref{eq:LTIResidual}. The following proposition shows that $\bfM_1, \bfM_2, \bfM_3$ can be derived via a least-squares problem using $\rbfR,\rbfW,\bfG$.
\begin{proposition} \label{prop:ErrorLS}
Define the data matrix $\bfD \in \R^{K \times \frac{1}{2}(n+p)(n+p+1)}$ as
\begin{align} \label{eq:DataMatErrInf}
 \bfD =
 \begin{bmatrix}
 \vech(2\rbfw_0 \rbfw_0^T - \diag(\rbfw_0 \rbfw_0^T)) & \cdots & \vech(2\rbfw_{K-1} \rbfw_{K-1}^T - \diag(\rbfw_{K-1} \rbfw_{K-1}^T)) \\
 \vech(2\bfg_{1} \bfg_{1}^T - \diag(\bfg_{1} \bfg_{1}^T)) & \cdots & \vech(2\bfg_{K} \bfg_{K}^T - \diag(\bfg_{K} \bfg_{K}^T))\\
 2 \vect(\bfg_{1} \rbfw_0^T) & \cdots & 2 \vect(\bfg_{K} \rbfw_{K-1}^T)
\end{bmatrix}^T
\end{align}
where $\vect(\cdot)$ is the vectorization operator, $\vech(\cdot)$ is the half-vectorization operator of a symmetric matrix, and $\diag(\cdot)$ is a diagonal matrix preserving only the diagonal entries of its matrix argument. Let $\bff\in \R^k$ whose $(k+1)$-th entry is
$$[\bff]_{k+1} = \|\rbfr_k\|^2_2 - \rbfw_{k+1}^T \bfM_4 \rbfw_{k+1} + 2\rbfw_{k+1}^T\tbfA \rbfw_k +2\rbfw_{k+1}^T \tbfB \bfg_{k+1}$$ and consider the least squares problem
\begin{align} \label{eq:ErrorLeastSquares}
\min_{\substack{\hbfM_1 \in \mathbb{R}^{\nr \times \nr},\\ \hbfM_2 \in \mathbb{R}^{p \times p},\\ \hbfM_3 \in \mathbb{R}^{p \times \nr}}} \sum_{k=0}^{K-1} \biggl(\rbfw_k^T \hbfM_1 \rbfw_k &+  \bfg_{k+1}^T \hbfM_2 \bfg_{k+1} + 2 \bfg_{k+1}^T \hbfM_3 \rbfw_k  -[\bff]_{k+1}\biggr)^2.
\end{align}
If $K \ge (n+p)(n+p+1)/2$ and the data matrix $\bfD$ has full rank, the unique solution to~\eqref{eq:ErrorLeastSquares} is $\hbfM_1 = \bfM_1, \hbfM_2 = \bfM_2, \hbfM_3 = \bfM_3$ with objective value 0.
\end{proposition}
\begin{proof}
The least squares problem~\eqref{eq:ErrorLeastSquares} is equivalent to
\begin{align} \label{eq:ErrorLeastSquaresMatrix}
\min_{\hbfo} \|\bfD \hbfo -\bff\|_2^2
\end{align}
where
$$\hbfo = \begin{bmatrix}
\vech(\hbfM_1) \\ \vech(\hbfM_2) \\ \vect(\hbfM_3)
\end{bmatrix}
\in \R^{\frac{1}{2}(n+p)(n+p+1)}.$$
As the data matrix $\bfD$ is full rank with $K \ge (n+p)(n+p+1)/2$, it follows that~\eqref{eq:ErrorLeastSquaresMatrix} has a unique solution. This implies that~\eqref{eq:ErrorLeastSquares} also has a unique solution due to the equivalence between~\eqref{eq:ErrorLeastSquares} and~\eqref{eq:ErrorLeastSquaresMatrix}. From the residual norm expression~\eqref{eq:SqResNorm}, notice that $\hbfM_1 = \bfM_1, \hbfM_2 = \bfM_2, \hbfM_3 = \bfM_3$ yields an objective value of 0 for~\eqref{eq:ErrorLeastSquares}. Therefore, it is the unique minimizer for the least squares problem  \eqref{eq:ErrorLeastSquares}.
\end{proof}

\subsection{Error estimator based on the learned residual norm operators} \label{subsec:StateErrBndComput}

Consider a queryable system \eqref{eq:LTIsystem}. The residual trajectory of the re-projected state trajectory can be computed during the re-projection step. Let $\bfV_n$ be a basis matrix, $\bfwtrain_0 \in \Vcal_{\nh}$ an initial condition, and $\bfGtrain = [\bfgtrain_1, \dots, \bfgtrain_K] \in \Gcal$ an input trajectory.
Consider further the corresponding re-projected trajectory $\rbfWtrain = [\rbfwtrain_0,\dots,\rbfwtrain_{K-1}]$ and the corresponding residual trajectory $\rbfRtrain =[\rbfrtrain_0,\dots,\rbfrtrain_{K-1}]$. Denote by $$\bfPsitrain = \begin{bmatrix}(\rbfWtrain)^T & (\bfGtrain)^T \end{bmatrix} \in \R^{K \times (n+p)}$$ the data matrix for operator inference and $\bfDtrain$ the data matrix~\eqref{eq:DataMatErrInf} with $\rbfw_k = \rbfwtrain_k$ and $\bfg_k = \bfgtrain_k$. If $\bfPsitrain$ and $\bfDtrain$ have full rank with $K \ge(n+p)(n+p+1)/2$, the reduced model \eqref{eq:ROMsystem} can be recovered together with $\bfM_1, \bfM_2, \bfM_3$ defined in \eqref{eq:SqResNorm} following Section~\ref{subsec:OpInfReproj} and Proposition~\ref{prop:ErrorLS}.

Set $J>0$ as the number of time steps for prediction and let $\bfWtest = [\bfwtest_1,\dots,\bfwtest_J]$  be the state trajectory resulting from system~\eqref{eq:LTIsystem} subject to the initial state $\bfwtest_0$ and the input trajectory $\bfGtest = [\bfgtest_1, \dots, \bfgtest_J] \in \Gcal$. For the initial state $\tbfwtest_0 = \bfV^T_{\nr} \bfwtest_0$, denote by  $\tbfWtest=[\tbfwtest_1,\dots,\tbfwtest_J]$ the associated reduced state trajectory produced by the recovered reduced model derived from operator inference and re-projection. The  norm of the residual of the trajectory $\tbfWtest$ with respect to the high-dimensional model can be computed via \eqref{eq:SqResNorm} by invoking $\bfM_1, \bfM_2, \bfM_3$ learned as in Proposition~\ref{prop:ErrorLS}. Under certain conditions, the state error of $\tbfwtest_k$ can be bounded as follows.
\begin{proposition} \label{prop:ErrBndUnitNorm}
If $\|\bfrtest_k\|_2, k \in \mathbb{N},$ is the residual norm of $\tbfwtest_k$ calculated through~\eqref{eq:SqResNorm}, under the assumption that $\|\bfA\|_2 \le 1$, the state error of the learned reduced model can be bounded via
\begin{align} \label{eq:ErrBndStateUnitNorm}
    \|\bfwtest_k - \bfV_{\nr} \tbfwtest_k\|_2 \leq \StErrBnd_k(\underbrace{1,\dots,1}_{k+1}; \bfwtest_0,\bfGtest)= \|\bfwtest_0 - \bfV_{\nr} \tbfwtest_0\|_2 + \sum_{l = 0}^{k - 1} \|\bfrtest_{l + 1}\|_2.
\end{align}
\label{eq:PropConstant1}
\end{proposition}
\begin{proof}
Using the basis matrix $\bfV_{\nr}$, the input trajectory $\bfGtest$, the recovered reduced operators $\tbfA,\tbfB$ from Section~\ref{subsec:OpInfReproj}, and the recovered matrices $\bfM_1,\bfM_2,\bfM_3$ from Proposition~\ref{prop:ErrorLS}, the residual norm $\|\bfrtest_k\|_2$ can be computed for $k = 1,\dots,J$.

From~\eqref{eq:ErrState}, we deduce
\begin{align} \label{eq:ActualStateEB}
   \| \bfwtest_k - \bfV_{\nr} \tbfwtest_k\|_2 \le \|\bfA^k\|_2 \|\bfwtest_0 - \bfV_{\nr} \tbfwtest_{0}\|_2 + \sum_{l=0}^{k-1} \|\bfA^{k-l-1}\|_2 \|\bfrtest_{l+1}\|_2 \le \StErrBnd_k(\underbrace{1,\dots,1}_{k+1}; \bfwtest_0,\bfGtest) .
\end{align}
The second inequality in~\eqref{eq:ActualStateEB} holds as $\|\bfA^l\|_2 \le \|\bfA\|^l_2 \le 1$ for $0\le l \le k$.
\end{proof}

\begin{remark}
Proposition~\ref{eq:PropConstant1} shows that $\StErrBnd_k$ is a pre-asymptotic, computable upper bound on the generalization error of the learned reduced model with respect to control inputs.
\end{remark}

The condition stated in Proposition~\ref{prop:ErrBndUnitNorm} is met, for example, in the following situations. Let the bilinear form $a$ in \eqref{eq:WeakForm} be symmetric. If $\beta = 0$ in~\eqref{eq:LTIsystem} (forward Euler) and the basis functions $\varphi_i$ are, e.g., orthonormal such that $\bfM$ is a multiple of the identity matrix, then $\bfA$ is symmetric and there exists a sufficiently small time-step size $\delta t$ such that the spectral radius $\rho(\bfA) = \|\bfA\|_2 \le 1$. Alternatively, certain mass lumping techniques \cite{book:Thomee2006} may be applied
to attain an $\bfM$ with such structure. Finally, if $\beta=1$ in \eqref{eq:LTIsystem} (backward Euler), it can be shown that there exists $\delta t$ such that the maximum singular value of $(\bfI - \delta t \bfM^{-1} \bfK)^{-1}$ is at most 1, which relies on the symmetry of $\bfM$ and $\bfK$.

\subsection{Probabilistic \emph{a posteriori} error estimator for the state} \label{subsec:Anorm}

We discuss an approach to bound $\|\bfA^l\|_2$, $0 \le l \le J$, if the condition $\|\bfA\|_2 \le 1$ in Proposition~\ref{prop:ErrBndUnitNorm} is not met or if it is unknown if $\|\bfA\|_2 \le 1$ holds. We seek an upper bound for $\|\bfA^l\|_2$ with probabilistic guarantees in order to derive a probabilistic \emph{a posteriori} error estimator for the state in Section~\ref{subsec:ProbError}. The practical implementation of this error estimator is then discussed in Section~\ref{subsec:PracticalProbBnd}. In the following, denote by $N(\boldsymbol{\mu},\boldsymbol{\Sigma})$ the multivariate Gaussian distribution  with mean $\boldsymbol{\mu}$  and covariance matrix $\boldsymbol{\Sigma}$.

\subsubsection{Probabilistic upper bound for $\|\bfA^{l}\|_2$ and the state error} \label{subsec:ProbError}

\begin{lemma} \label{lem:ChiSqErrorBnd}
For $l \in \mathbb{N}$, let $\bfTheta^{(l)} = \bfA^l \bfZ_1$ where $\bfZ_1 \sim N(\boldsymbol{0}_{\nh\times 1}, \bfI_{\nh})$ so that $\bfTheta^{(l)}$ is  an $\nh$-dimensional Gaussian random vector with mean zero and covariance $\bfA^l (\bfA^l)^T $. Suppose that $\{\bfTheta^{(l)}_{i}\}_{i=1}^M$ are $M \in \mathbb{N}$ independent and identically distributed $\nh$-dimensional random vectors with the same law as  $\bfTheta^{(l)}$. Then, for $\gamma_l > 0$,
\begin{align} \label{eq:ChiSqProb}
P\left(\gamma_l \max_{i=1,\dots,M} \|\bfTheta^{(l)}_{i}\|^2_2 \ge \|\bfA^l\|_2^2 \right) \ge  1-\left[F_{\chi^2_1}\left(\frac{1}{\gamma_l} \right) \right]^M
\end{align}
where $F_{\chi^2_1}$ is the cumulative distribution function of the chi-squared distribution with 1 degree of freedom.
\end{lemma}

\begin{proof}
It suffices to show that
\begin{align} \label{eq:ChiSqProb1Sample}
    P\left(\gamma_l \|\bfTheta^{(l)}\|^2_2 \ge \|\bfA^l\|_2^2 \right) \ge  1-F_{\chi^2_1}\left(\frac{1}{\gamma_l} \right)
\end{align}
because using the fact
\begin{align*}
    P\left(\gamma_l \max_{i=1,\dots,M} \|\bfTheta^{(l)}_{i}\|^2_2 \le \|\bfA^l\|_2^2 \right)= \left[ P(\gamma_l \|\bfTheta^{(l)}_{i}\|^2_2 \le \|\bfA^l\|_2^2)\right]^M,
\end{align*}
we conclude that
\begin{align*}
    P\left(\gamma_l \max_{i=1,\dots,M} \|\bfTheta^{(l)}_{i}\|^2_2 \ge \|\bfA^l\|_2^2 \right) = 1 - \left[ P(\gamma_l \|\bfTheta^{(l)}_{i}\|^2_2 \le \|\bfA^l\|_2^2)\right]^M \ge 1-\left [F_{\chi^2_1}\left(\frac{1}{\gamma_l}\right)\right ]^M
\end{align*}
as desired.
The proof of~\eqref{eq:ChiSqProb1Sample} uses ideas similar to that in \cite{paper:Dixon1983}. Recall that $$\|\bfA^l\|_2 = \sqrt{\lambda_{max}((\bfA^l)^T\bfA^l)} $$ where $\lambda_{max}(\cdot)$ represents the largest eigenvalue of the matrix argument. Since $(\bfA^l)^T\bfA^l$ is real and symmetric, $(\bfA^l)^T\bfA^l = \bfQ \Lambda \bfQ^T$ where $\bfQ \in \R^{\nh \times \nh}$, $\bfQ^T\bfQ = \bfI_{\nh}$, and $\Lambda$ is a diagonal matrix whose entries $[\Lambda]_{ii} = \lambda_i$ satisfy $0 \le \lambda_1 \le \dots \le \lambda_{\nh} = \|\bfA^l\|^2_2$.  By setting $\bfZ_2 =\bfQ^T\bfZ_1$, we have $\bfZ_2 \sim N(\boldsymbol{0}_{\nh \times 1},\bfI_{\nh})$ and that
\begin{align*} 
\|\bfTheta^{(l)}\|_2^2 = (\bfTheta^{(l)})^T \bfTheta^{(l)} = \bfZ_1^T (\bfA^l)^T \bfA^l \bfZ_1 = \bfZ_1^T \bfQ \Lambda \bfQ^T   \bfZ_1 = \bfZ_2^T \Lambda \bfZ_2 \ge \|\bfA^l\|_2^2([\bfZ_2]_{\nh})^2
\end{align*}
where $[\bfZ_2]_{\nh}$ is the $\nh$-th component of $\bfZ_2$. Since $[\bfZ_2]_{\nh} \sim N(0,1)$, $([\bfZ_2]_{\nh})^2 \sim \chi^2_1$, i.e. it is a chi-squared random variable with 1 degree of freedom. It follows that for a constant $\gamma_l > 0$
with
\[
P\left(\gamma_l ([\bfZ_2]_{\nh})^2 \ge 1 \right) = 1-F_{\chi^2_1}\left(\frac{1}{\gamma_l} \right)
\]
we obtain
\begin{align*}
P(\gamma_l \|\bfTheta^{(l)}\|^2_2 \ge \|\bfA^l\|_2^2) \ge P\left(\gamma_l ([\bfZ_2]_{\nh})^2 \ge 1 \right) = 1-F_{\chi^2_1}\left(\frac{1}{\gamma_l} \right).
\end{align*}
\end{proof}

\begin{remark}
Results similar to~\eqref{eq:ChiSqProb} can be obtained for other distributions on $\bfTheta^{(l)}$ building on, e.g., \cite{paper:Dixon1983,paper:BujanovicK2020}.
\end{remark}

Using \eqref{eq:ChiSqProb}, we derive a probabilistic \emph{a posteriori} error estimator as the next result demonstrates.
\begin{proposition} \label{prop:GeneralStateErrBnd}
For  $l=1,\dots,J$, let $\bfTheta^{(l)} = \bfA^l \bfZ$ where $\bfZ \sim N(\boldsymbol{0}_{\nh\times 1}, \bfI_{\nh})$ so that $\bfTheta^{(l)}$ is  an $\nh$-dimensional Gaussian random vector with mean zero and covariance $\bfA^l (\bfA^l)^T $. Let $\{\bfTheta^{(l)}_{i}\}_{i=1}^M$ be independent and identically distributed $\nh$-dimensional random vectors with the same law as  $\bfTheta^{(l)}$ and define $$\Xi_l = \sqrt{\gamma_l \max_{i=1,\dots,M} \|\bfTheta^{(l)}_{i}\|^2_2},$$ for $\gamma_l > 0, l \ge 1$ with $\Xi_0=1$. For an initial state $\bfw_0 \in \Vcal_{\nh}$ and an input trajectory $\bfG$, the following holds
\begin{align} \label{eq:ProbErrBound}
    P\biggl(\bigcap_{k=1}^J \biggl\{\| \bfw_k - \bfV_{\nr} \tbfw_k\|_2   \le \StErrBnd_k(\Xi_k,\dots,\Xi_0;\bfw_0,\bfG)\biggr\} \biggr)
     \ge \max \left(0,  1 - \sum_{l=1}^J \left[F_{\chi^2_1}\left(\frac{1}{\gamma_l} \right) \right]^M \right).
\end{align}
\end{proposition}

\begin{proof}
Define the events $E_l=\{\Xi_l^2 \ge \|\bfA^l\|_2^2\}$ for $l=1,\dots,J$ and the event
\begin{align*}
E=\bigcap_{k=1}^J\biggl\{\| \bfw_k - \bfV_{\nr} \tbfw_k\|_2 \le \StErrBnd_k(\Xi_k,\dots,\Xi_0;\bfw_0,\bfG) \biggr\}\,.
\end{align*}
Recall from \eqref{eq:ErrStateBnd} that
\[
\| \bfw_k - \bfV_{\nr} \tbfw_k\|_2 \le \StErrBnd_k(\|\bfA^k\|_2,\dots,\|\bfA^0\|_2;\bfw_0,\bfG)\,,\qquad k = 1, \dots, J\,,
\]
holds, which means that we obtain $$ \displaystyle P(E) \ge P(\cap_{l=1}^J E_l).$$
Using the Fr\'echet inequality, we obtain
$$P(\cap_{l=1}^J E_l) \ge \max \left(0, \sum_{l=1}^J P(E_l) - (J-1) \right).$$
Therefore, using Lemma~\ref{lem:ChiSqErrorBnd},
$$P(E) \ge \max \left(0,\sum_{l=1}^J \left(1-\left[F_{\chi^2_1}\left(\frac{1}{\gamma_l} \right) \right]^M \right)- (J-1) \right) = \max \left(0,  1 - \sum_{l=1}^J \left[F_{\chi^2_1}\left(\frac{1}{\gamma_l} \right) \right]^M \right).$$

\end{proof}

\subsubsection{Sampling random vectors from queryable systems} \label{subsec:PracticalProbBnd}

We now discuss a practical implementation of the probabilistic error bound in Proposition~\ref{prop:GeneralStateErrBnd}. We resume the setup outlined in Section~\ref{subsec:StateErrBndComput}. Recall that the reduced model~\eqref{eq:ROMsystem} and $\bfM_1,\bfM_2,\bfM_3$ are recovered using the input trajectory $\bfGtrain$. Also, $\bfWtest$ and $\tbfWtest$ are the state and reduced state trajectories owing to the input trajectory $\bfGtest$ while $\|\bfrtest_k\|_2$ is the residual norm of $\tbfwtest_k$ calculated through~\eqref{eq:SqResNorm}.

To construct an upper bound for $\|\bfwtest_k - \bfV_{\nr}\tbfwtest_k\|_2$ according to Proposition~\ref{prop:GeneralStateErrBnd}, realizations of the random vectors $\bfTheta^{(l)} \sim N(\boldsymbol{0}_{\nh \times 1},\bfA^l (\bfA^l)^T), l = 1,\dots,k$ need to be simulated. Therefore, for fixed $M$, if $\{\bfz_i\}_{i=1}^M$ are realizations of $\bfZ$, realizations $\{\bftheta^{(l)}_{i}\}_{i=1}^M$ of $\bfTheta^{(l)}$ and hence a single realization $$\xi_l=\sqrt{\gamma_l \max_{i=1,\dots,M} \|\bftheta^{(l)}_{i}\|^2_2} $$ of $\Xi_l$ for $l=1,\dots,J$ can be simulated by querying~\eqref{eq:LTIsystem} for $J$ time steps with control input $\bfg_k = \boldsymbol{0}_{p \times 1}$ for all $k$ and with the realizations $\{\bfz_i\}_{i=1}^M$ serving as $M$ initial states, i.e. $\bfw_0 = \bfz_i$ for $i=1,\dots,M$. This produces $M$ trajectories of $\bfw_k = \bfA^k \bfw_0$. Note that $\xi_0=1$.

For specified $\gamma_l > 0$ which controls the confidence level (failure probability) of the probabilistic error estimator in~\eqref{eq:ProbErrBound}, an error estimate for $\|\bfwtest_k-\bfV_{\nr}\tbfwtest_k\|_2$ for $k=1,\dots,J$ is provided by
\begin{align}\label{eq:RealizedStateBnd}
      \StErrBnd_k(\xi_k,\dots,\xi_0;\bfwtest_0,\bfGtest) =
    \xi_k \, \|\bfwtest_0 - \bfV_{\nr} \tbfwtest_{0}\|_2 + \sum_{l=0}^{k-1} \xi_{k-l-1} \, \|\bfrtest_{l+1}\|_2
\end{align}
which we refer to as learned error estimate.

\begin{remark} Bounds on an output, a quantity of interest which is obtained via a linear functional of the state $\bfw_k$, can also be formulated if the norm of the output operator is available. Let the output at time $k$ be expressed as
\begin{align*}
    y_k = \bfC \bfw_k
\end{align*}
for which it is assumed that $\|\bfC\|_2$ is known. The output for the low-dimensional system $\ROMOutput_k$ is therefore
\begin{align*}
    \ROMOutput_k = \bfC \bfV_{\nr} \tbfw_k.
\end{align*}
Following \cite{paper:HaasdonkO2011}, since $$\|y_k - \ROMOutput_k\|_2 = \|\bfC (\bfw_k - \bfV_n \tbfw_k)\|_2 \le \|\bfC\|_2 \| \bfw_k - \bfV_n \tbfw_k\|_2,$$ and
$\StErrBnd_k(\xi_k,\dots,\xi_0;\bfw_0,\bfG)$ is an error estimate for $\|\bfw_k - \bfV_{\nr} \tbfw_k\|_2$,
we deduce
that $\ROMOutput_k -\OutErrBnd_k$ and $\ROMOutput_k + \OutErrBnd_k$ are lower and upper bound estimates for $y_k$
where
\begin{align} \label{eq:outputBounds}
    \OutErrBnd_k = \|\bfC\|_2 \StErrBnd_k(\xi_k,\dots,\xi_0;\bfw_0,\bfG).
\end{align}
\label{rm:Output}
\end{remark}

\subsection{Computational procedure for offline and online phase} \label{subsec:OffOn}

The proposed offline-online computational procedure for non-intrusive model reduction of certified reduced models is summarized in Algorithm~\ref{alg:OffOn}. It builds on the reprojection scheme in Algorithm~\ref{alg:reprojection} introduced in~\cite{paper:Peherstorfer2019} which is modified to include computation of the residual trajectory. The offline phase serves as a training stage to determine the unknown quantities while the online phase utilizes these for certified predictions.

The inputs to Algorithm~\ref{alg:OffOn} include the number of time steps  $K$ (training), $J$ (prediction), initial condition $\bfw_0$, the snapshot matrix $\bfWbasis = [
\bfw_0, \bfwbasis_1,\dots, \bfwbasis_K]$ owing to the input trajectory $\bfGbasis=[\bfgbasis_1,\dots,\bfgbasis_K] \in \Gcal$, the basis dimension $\nr$, the input trajectories $\bfGtrain,\bfGtest$ for training and prediction, $M \ge 1$, $\{\gamma_l\}_{l=1}^J$, and the input trajectory $\bfGnorm = [\boldsymbol{0}_{p \times 1},\dots,\boldsymbol{0}_{p \times 1}] \in \R^{p \times J}$ for finding an upper bound for $\|\bfA^l\|_2$, $1 \le l \le J$, and the computational model~\eqref{eq:LTIsystem} that can be queried.

The offline stage constitutes operator inference with reprojection (Section~\ref{subsec:OpInfReproj}) and estimation of state error upper bounds (Sections~\ref{subsec:StateErrBnd},~\ref{subsec:Anorm}) with the input trajectory $\bfGtrain \in \Gcal$. It is composed of three parts: inferring the reduced system, inferring the residual-norm operator, and finding an upper bound for the norm of $\bfA$ in the error estimator. The offline phase proceeds by building the low-dimensional basis $\bfV_{\nr}$ from trajectories of the state contained in $\bfWbasis$. The re-projetion algorithm is then invoked to generate the re-projected states $\rbfwtrain_k$ and its residual $\rbfrtrain_k$ corresponding to the control input $\bfGtrain$. Using data on $\rbfwtrain_k$ and $\bfGtrain$, the least squares problem~\eqref{eq:OpInf} is formulated in order to recover the reduced system \eqref{eq:ROMsystem} in a non-intrusive manner. The second part of the offline stage utilizes the inferred reduced system and data on the residual $\rbfrtrain$ to set up the least squares problem~\eqref{eq:ErrorLeastSquares}.  Solving
\eqref{eq:ErrorLeastSquares} yields the operators $\bfM_1,\bfM_2,\bfM_3$, which enable the computation of the residual norm \eqref{eq:SqResNorm} at any time for a specified control input. Finally, upper bounds for the operator norms $\|\bfA^l\|_2$ in the \emph{a posteriori} error expression \eqref{eq:aposterioribnd} are sought by querying the system~\eqref{eq:LTIsystem} at initial conditions consisting of $M$ realizations of $\bfZ \sim N(\boldsymbol{0}_{\nh \times 1}, \bfI_{\nh})$. The $M$ trajectories corresponding to each initial condition are employed in the definition of $\xi_l$ which is a realization of the probabilistic bound $\Xi_l$, i.e. $\|\bfA^l\|_2 \le \Xi_l$. Notice that Algorithms~\ref{alg:reprojection} and~\ref{alg:OffOn} do not rely on knowledge of $\bfA,\bfB$ in~\eqref{eq:LTIsystem} and $\bfM_1,\bfM_2,\bfM_3$ in~\eqref{eq:SqResNorm}. Furthermore, it is unnecessary  to use the same input trajectory $\bfGtrain$ for solving the least squares problems~\eqref{eq:OpInf} and~\eqref{eq:ErrorLeastSquares}.

In the online stage, the deduced quantities in the offline stage are invoked to compute the low-dimensional solution~\eqref{eq:ROMsystem}, the norm of its residual \eqref{eq:SqResNorm}, and consequently an upper bound for the state error~\eqref{eq:ErrBndStateUnitNorm} or~\eqref{eq:RealizedStateBnd} for an input trajectory $\bfGtest \in \Gcal$.

Algorithm~\ref{alg:OffOn} serves as the reference for the numerical examples undertaken in Section~\ref{sec:NumEx}.

\begin{algorithm}[t]
\caption{Data sampling with re-projection}
\begin{algorithmic}[1]
  \STATE Set $\rbfw_0 = \bfV_{\nr}^T\bfw_0$ 
  \FOR{$k = 0, \dots, K-1$}
        \STATE Query \eqref{eq:LTIsystem} for a single time step to obtain $\bfw_{tmp} = \bfA \bfV_{\nr} \rbfw_k + \bfB \bfg_{k+1}$
        \STATE Set $\rbfw_{k+1} = \bfV_{\nr}^T \bfw_{tmp}$
        \STATE Compute the residual $\rbfr_k = \bfw_{tmp} - \bfV_n\rbfw_{k + 1}$
  \ENDFOR
  \STATE Return $[\rbfw_0, \rbfw_1, \dots, \rbfw_{K}]$ and $[\rbfr_0, \rbfr_1, \dots, \rbfr_{K-1}]$
\end{algorithmic}
\label{alg:reprojection}
\end{algorithm}

\begin{algorithm}[t]
\caption{Data-driven model reduction}
\begin{algorithmic}[1]
\STATEx {\bfseries Offline phase}
\STATE Construct a low-dimensional basis $\bfV_{\nr}$ from the snapshot matrix $\bfWbasis$
\STATE Generate $\{\rbfwtrain_k\}_{k=0}^{K}$ via re-projection and its residual $\{\rbfrtrain_k\}_{k=0}^{K-1}$ (Algorithm~\ref{alg:reprojection}) using $\bfGtrain$
\STATE Perform operator inference by solving~\eqref{eq:OpInf} to deduce $\tbfA,\tbfB$
\STATE Infer $\bfM_1,\bfM_2,\bfM_3$ from \eqref{eq:ErrorLeastSquares} for the computation of~\eqref{eq:SqResNorm}
\STATE Simulate $M$ realizations $\{\bfz_i\}_{i=1}^M$ of $\bfZ \sim N(\boldsymbol{0}_{\nh \times 1},\bfI_{\nh})$
\STATE Produce $M$ realizations $\{\bftheta^{(l)}_{i}\}_{i=1}^M$ of $\bfTheta^{(l)}$ for $l=1,\dots,J$ by querying~\eqref{eq:LTIsystem} for $J$ time steps with $\bfw_0 = \bfz_i$, $i=1,\dots,M$ and input $\bfGnorm$
\STATE Compute $\xi_l = \sqrt{\gamma_l \max_{i=1,\dots,M} \|\bftheta^{(l)}_{i}\|_2^2}$ for $l=1,\dots,J$
\STATEx
\STATEx {\bfseries Online phase}
\STATE Calculate the low-dimensional solution $\{\tbfwtest_k\}_{k=1}^J$ to~\eqref{eq:ROMsystem} using the inferred $\tbfA,\tbfB$ and input $\bfGtest$
\STATE Evaluate $\|\bfrtest_k\|_2^2$ for $k=1,\dots,J$ in~\eqref{eq:SqResNorm} utilizing the deduced $\bfM_1,\bfM_2,\bfM_3$
\STATE Estimate the \emph{a posteriori} error for the state via \eqref{eq:ErrBndStateUnitNorm} or~\eqref{eq:RealizedStateBnd}  for $k=1,\dots,J$
\end{algorithmic}
\label{alg:OffOn}
\end{algorithm}

\section{Numerical results} \label{sec:NumEx}
The numerical examples in this section illustrate the following points: 1) the quantities for error estimators are learned from data up to numerical errors, 2) the learned low-dimensional system and the residual norm for the \emph{a posteriori} error estimators are exact reconstructions of those resulting from intrusive model reduction, 3) the learned quantities can be used to predict the low-dimensional solution and provide error estimates for specified control inputs, and 4) error estimators for the output, i.e. quantity of interest, can be deduced if the output operator is linear in the state and its norm  is available.

\subsection{Error quantities}
We compute the following quantities to assess the predictive capabilities of reduced models learned from data for test input trajectories $\bfGtest$ and test initial conditions $\bfwtest_0$.

Error of the reduced solution:
    \begin{align} \label{eq:Err_IntNonIntVSFOM}
       \ErrPerf^{(1)} = \frac{\|\bfWtest - \bfV_{\nr}\bbfWtest\|_F}{\|\bbfWtest\|_F}
    \end{align}
where $\bbfWtest$ refers to the trajectory of the reduced system inferred via intrusive model reduction ($\hatbfWtest$) or operator inference ($\tbfWtest$) and $\|\cdot\|_F$ is the Frobenius norm.

Time-averaged residual norm:
      \begin{align} \label{eq:Err_IntVSNonIntUnitNorm}
       \ErrPerf^{(2)} = \frac{1}{J}\sum_{k=0}^{J-1} \|\bfrtest_{k+1}\|_2
    \end{align}
where the residual norm $\|\bfrtest_{k+1}\|_2$ is computed through the two approaches for model reduction we compare: intrusive ($\|\hatbfrtest_{k+1}\|_2$) vs operator inference ($\|\tbfrtest_{k+1}\|_2$).

Relative average state error over time and its corresponding \emph{a posteriori} error estimates tabulated in Table~\ref{table:RelAveErr}.
\begin{table}

\begin{tabular*}{\textwidth}{
  @{\extracolsep{\fill}}
  l
  >{$\displaystyle}c<{\vphantom{\sum_{1}{N}}$} >{\refstepcounter{equation}(\theequation)}r
  @{}
}
\toprule
Errors and error estimators & \multicolumn{1}{c}{Definition} & \multicolumn{1}{c}{} \\
\midrule
\makecell{error of reduced solution \\via operator inference} &
  \ErrPerf^{(3)} = \frac{\sum_{k=0}^{J-1} \|\bfwtest_k - \bfV_{\nr}\tbfwtest_k \|_2}{J \sum_{k=0}^{J-1} \|\bfwtest_k\|_2} & \label{eq:Err_AveRelErrActualErr}
\\[1cm]
\makecell{intrusive model reduction \\ upper bound for the state error} &
  \ErrPerfBnd^{(1)} = \frac{\sum_{k=0}^{J-1} \StErrBnd_k(\|\bfA^k\|_2,\dots,\|\bfA^0\|_2;\bfwtest_0,\bfGtest)}{J \sum_{k=0}^{J-1} \|\bfwtest_k\|_2} & \label{eq:Err_AveRelErrActualBnd}
\\[1cm]
\makecell{realization of probabilistic \\ upper bound for the state error} &
  \ErrPerfBnd^{(2)} = \frac{\sum_{k=0}^{J-1} \StErrBnd_k(\xi_k,\dots,\xi_0;\bfw_0,\bfGtest)}{J \sum_{k=0}^{J-1} \|\bfwtest_k\|_2} & \label{eq:Err_AveRelErrProbBnd}
\\[1cm]
\makecell{learned deterministic \\ upper bound for the state error} &
  \ErrPerfBnd^{(3)} = \frac{\sum_{k=0}^{J-1} \StErrBnd_k(1,\dots,1;\bfwtest_0,\bfGtest)}{J \sum_{k=0}^{J-1} \|\bfwtest_k\|_2} & \label{eq:Err_AveRelErrDetBnd}
\\
\bottomrule
\end{tabular*}
\caption{Relative average state error over time and its corresponding error estimates obtained from intrusive model reduction and operator inference.}
\label{table:RelAveErr}
\end{table}

In \eqref{eq:Err_AveRelErrProbBnd}, $\xi_k$, $k=0,\dots,J$, are realizations of the random variables $\Xi_k$ defined in Proposition~\ref{prop:GeneralStateErrBnd}. In our experiments, we set $\gamma_l = \gamma$ for $l=1,\dots,J$ so that the probability lower bound in~\eqref{eq:ProbErrBound} becomes $$\ProbLB(\gamma,M,J) = \max \left(0,  1 - J \left[F_{\chi^2_1}\left(\frac{1}{\gamma} \right) \right]^M \right).$$

Relative state error at a particular time point $k$ and its corresponding \emph{a posteriori} error estimates tabulated in Table~\ref{table:ErrAtTime}.

\begin{table}

\begin{tabular*}{\textwidth}{
  @{\extracolsep{\fill}}
  l
  >{$\displaystyle}c<{\vphantom{\sum_{1}{N}}$} >{\refstepcounter{equation}(\theequation)}r
  @{}
}
\toprule
Errors and error estimators & \multicolumn{1}{c}{Definition} & \multicolumn{1}{c}{} \\
\midrule
\makecell{error of reduced solution \\via operator inference} &
  \ErrPerf^{(4)}=\frac{ \|\bfwtest_k - \bfV_{\nr}\tbfwtest_k \|_2}{\|\bfwtest_k\|_2} & \label{eq:Err_TimeActualErr}
\\[1cm]
\makecell{intrusive model reduction \\ upper bound for the state error} &
  \ErrPerfBnd^{(4)} = \frac{\StErrBnd_k(\|\bfA^k\|_2,\dots,\|\bfA^0\|_2;\bfwtest_0,\bfGtest)}{\|\bfwtest_k\|_2} & \label{eq:Err_TimeActualBnd}
\\[1cm]
\makecell{realization of probabilistic \\ upper bound for the state error} &
 \ErrPerfBnd^{(5)} = \frac{\StErrBnd_k(\xi_k,\dots,\xi_0;\bfw_0,\bfGtest)}{\|\bfwtest_k\|_2} & \label{eq:Err_TimeProbBnd}
\\[1cm]
\makecell{learned deterministic \\ upper bound for the state error} &
  \ErrPerfBnd^{(6)} = \frac{\StErrBnd_k(1,\dots,1;\bfwtest_0,\bfGtest)}{\|\bfwtest_k\|_2} & \label{eq:Err_TimeDetBnd}
\\
\bottomrule
\end{tabular*}
\caption{Relative state error at a particular time point $k$ and its corresponding error estimates obtained from intrusive model reduction and operator inference.}
\label{table:ErrAtTime}
\end{table}

\subsection{Heat transfer} \label{subsec:HeatTransf}

The setup for non-intrusive model reduction applied to this example is first described which is followed by the numerical results.

\subsubsection{Setup}

For $\PDEdomain = (0,1), \Tcal = (0,T), T=5,$  consider the heat equation on $(x,t) \in \PDEdomain \times \Tcal$ given by
\begin{align*}
    \frac{\partial}{\partial t} w(x,t) & = \mu \, \frac{\partial^2}{\partial x^2} w(x,t), \\
    w(0,t) & = 0, \\
    \frac{\partial}{\partial x}w(1,t) & = u(t) \\
    w(x,0) & = 0.
\end{align*}
To discretize the PDE, $\Omega$ is subdivided into $N = 133$ intervals with  width $\Delta x = 1/\nh. $ Let $\{\varphi_i\}_{i=1}^{\nh}$ be linear hat basis functions with $\varphi_i (j \Delta x) = \delta_{ij}$ where $\delta_{ij}$ is the Kronecker delta function. We obtain the continuous-time system $$\bfM \frac{d\bfw(t)}{dt} = \bfK \bfw(t) + \mu \begin{bmatrix}
0 \\ \vdots \\ 0 \\ u(t)
\end{bmatrix}$$
where $ [\bfK]_{ij} = -\mu \int_0^1 \frac{\partial \varphi_i}{\partial x} \frac{\partial \varphi_j}{\partial x} \,dx$ for $i,j=1,\dots,\nh$. In our simulation, we set $\mu = 0.1$ for the diffusivity parameter and temporally discretized the continuous system using backward Euler with $\delta t = 0.01$ being the time step size.

The basis $\bfV_{\nr}$ was constructed from the snapshot matrix of $K = 500$ time steps driven by the control input $\ubasis(t) = e^t \sin(20\pi t/T)$. The objective functions~\eqref{eq:OpInf} and~\eqref{eq:ErrorLeastSquaresMatrix} were optimized using the input trajectory $\bfGtrain = [0,z_1,\dots,z_K]$ where $z_i$ is a realization of $Z_i$, $i=1,\dots,K$,  which are independent and identically distributed $N(0,1)$ random variables.

\subsubsection{Results}

We now assess the performance of the learned reduced model and quantities required for \emph{a posteriori} error estimation. In the online stage, the control input $\utest(t) = e^t \sin(12\pi t/T)$ was discretized using 500 time steps. The quantities listed in Section~\ref{sec:NumEx} are calculated up to $\nr = 8$ basis vectors.

Figure~\ref{fig:sheat_residual_FOMerror} demonstrates that the reduced system and quantities required for error estimation can be recovered up to numerical errors. In particular, Figure~\ref{fig:sheat_FOMerr} plots~\eqref{eq:Err_IntNonIntVSFOM} for the reduced solution resulting from intrusive model reduction compared to that from operator inference. It demonstrates the theory established in earlier work \cite{paper:Peherstorfer2019} on the recovery of the reduced operators in the system \eqref{eq:ROMsystem}. Due to this, the error of the reduced trajectory from either approach is almost identical. The quantity \eqref{eq:Err_IntVSNonIntUnitNorm} involving the residual norm for both approaches of model reduction is presented in Figure~\ref{fig:sheat_residual}. The plot shows that both methods are in close agreement. If the conditions in Proposition~\ref{prop:ErrorLS} are met, the matrices $\bfM_1,\bfM_2,\bfM_3$ in \eqref{eq:SqResNorm} and hence the residual norm itself can also be recovered.

\begin{figure}
  \begin{subfigure}[b]{0.45\textwidth}
    \begin{center}
{{\LARGE\resizebox{1\columnwidth}{!}{\input{sheatBwd_IntNonIntVsFOM}}}}
\end{center}
    \caption{state error}
    \label{fig:sheat_FOMerr}
  \end{subfigure}
  \begin{subfigure}[b]{0.45\textwidth}
    \begin{center}
{{\LARGE\resizebox{1\columnwidth}{!}{\input{sheatBwd_IntVsNonIntResidual}}}}
\end{center}
    \caption{residual}
    \label{fig:sheat_residual}
  \end{subfigure}
  \caption{Heat equation: The results in plots (a)-(b) indicate that the reduced system and the quantities required for error estimation under operator inference, i.e. residual norm operators, are equal to their intrusive counterparts up to numerical errors.} \label{fig:sheat_residual_FOMerror}
\end{figure}
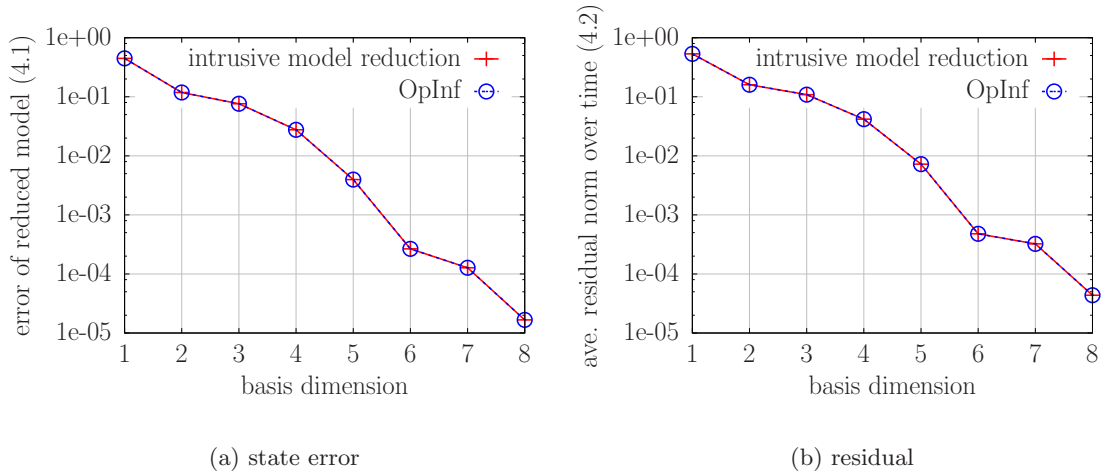

In this simulation, our knowledge of $\bfM,\bfK$ informed the choice of $\delta t$ so that $\|\bfA\|_2 \le 1$ and thus, the deterministic error estimator \eqref{eq:ErrBndStateUnitNorm} is applicable. If this is not the case, the probabilistic error estimator introduced in Section~\ref{subsec:Anorm} can be utilized instead. Figure~\ref{fig:sheat_errAtTimes_relAveErr} shows the deterministic and probabilistic \emph{a posteriori} error estimates. For the probabilistic error estimate, we chose $\gamma=1, M = 25, J = 500$ so that  $\ProbLB(\gamma,M,J) \approx 0.9641$. Only one realization of each of the random variables $\Xi_l$, $l=1,\dots,J$ was generated for this example. Figures~\ref{fig:sheat_errAtTimes0001} and~\ref{fig:sheat_errAtTimes01} display the learned reduced model error
\eqref{eq:Err_TimeActualErr} and the intrusive~\eqref{eq:Err_TimeActualBnd}, probabilistic~\eqref{eq:Err_TimeProbBnd}, and deterministic~\eqref{eq:Err_TimeDetBnd} error estimates at $t = 1$ and $t = 5$, respectively. These plots depict the intrusive model reduction error estimate~\eqref{eq:ErrStateBnd} for the state error. We notice that the intrusive and deterministic (non-intrusive) error estimates are almost identical. In addition, the plots convey that the learned error estimate \eqref{eq:RealizedStateBnd} under operator inference is roughly of the same order of magnitude as the error estimate provided by the intrusive approach. The calculated quantities for the time-averaged learned reduced model error~\eqref{eq:Err_AveRelErrActualErr} and its corresponding intrusive \eqref{eq:Err_AveRelErrActualBnd}, probabilistic ~\eqref{eq:Err_AveRelErrProbBnd}, and deterministic~\eqref{eq:Err_AveRelErrDetBnd} error estimates are likewise shown in Figure~\ref{fig:sheat_relAveErr}. The plot reveals that the behavior of  the time-averaged relative state error is similar to that of the relative state error at various time instances.

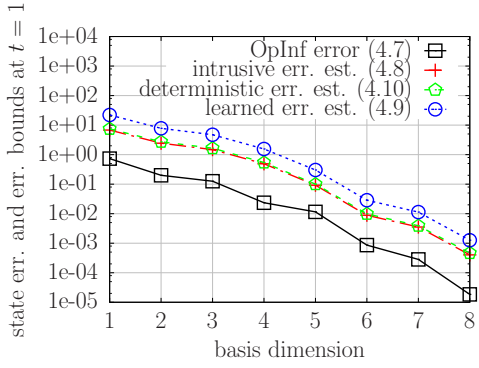
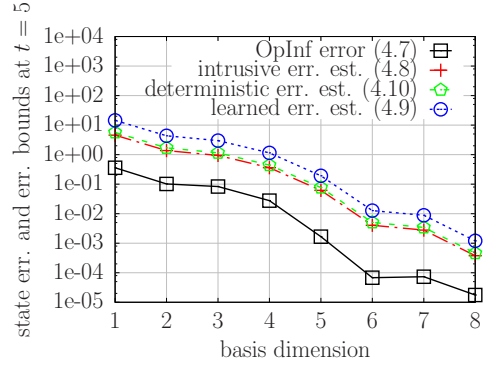
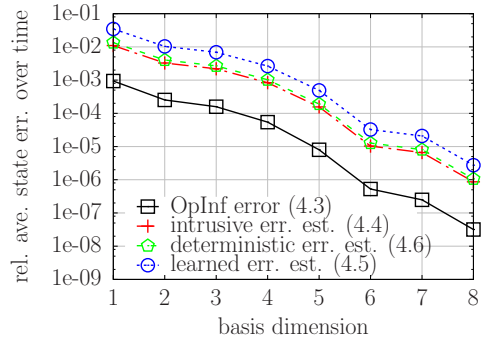
\begin{figure}[ht!]
  \begin{subfigure}[b]{0.45\textwidth}
    \begin{center}
{{\LARGE\resizebox{0.9\columnwidth}{!}{\input{sheatBwd_ErrAtTimes1}}}}
\end{center}
    \caption{error and error estimates at $t=1$}
    \label{fig:sheat_errAtTimes0001}
  \end{subfigure}
  \begin{subfigure}[b]{0.45\textwidth}
    \begin{center}
{{\LARGE\resizebox{0.9\columnwidth}{!}{\input{sheatBwd_ErrAtTimes5}}}}
\end{center}
    \caption{error and error estimates at $t=5$}
    \label{fig:sheat_errAtTimes01}
  \end{subfigure}

\begin{center}
  \begin{subfigure}[b]{0.45\textwidth}
{{\LARGE\resizebox{0.9\columnwidth}{!}{\input{sheatBwd_RelAveErr}}}}
    \caption{average error and error estimates over $t \in [0,5]$}
    \label{fig:sheat_relAveErr}
  \end{subfigure}
\end{center}
\caption{Heat equation: The plots illustrate that the deterministic (non-intrusive) error estimator \eqref{eq:ErrBndStateUnitNorm} and the learned (probabilistic) error estimator derived in Proposition~\ref{prop:GeneralStateErrBnd} bound the error of the reduced solution in this example. The intrusive and deterministic error estimates are close. In addition, the learned error estimator indicates an error of the same order of magnitude as the intrusive error estimator. The parameters used for the learned error estimator were chosen as $\gamma=1,M=25,J=500$ so that the the learned estimator gives an upper with probability $\ProbLB \approx 0.9641$.} \label{fig:sheat_errAtTimes_relAveErr}
\end{figure}

In practice, the learned error estimator may depend on the realizations of the random variables $\bfTheta^{(l)}$ simulated. In all simulations described above, we performed calculations using only a single realization of $\Xi_l$. We therefore generate multiple realizations of $\Xi_l$ and study the variability in the resulting learned error estimate associated with various sets of realizations of  $\bfTheta^{(l)}$. Figure~\ref{fig:sheat_MC} compiles the mean (solid) of 100 realizations of the learned error estimator \eqref{eq:Err_TimeProbBnd} for $t=1$ and $t=5$ and \eqref{eq:Err_AveRelErrProbBnd} in Figures~\ref{fig:sheat_MCerrAtTimes0001},~\ref{fig:sheat_MCerrAtTimes0001}, and~\ref{fig:sheat_MCrelAveErr} respectively. In each panel, the vertical bars symbolize the minimum and maximum among the simulated realizations while the error estimate from the intrusive approach is also shown. We observe from the minimum and maximum values that there is low variability in the learned error estimates generated.

\begin{figure}[ht!]
  \begin{subfigure}[b]{0.45\textwidth}
    \begin{center}
{{\huge\resizebox{0.9\columnwidth}{!}{\input{sheatBwd_MCSamples01}}}}
\end{center}
    \caption{error estimates at $t=1$}
    \label{fig:sheat_MCerrAtTimes0001}
  \end{subfigure}
  \begin{subfigure}[b]{0.45\textwidth}
    \begin{center}
{{\huge\resizebox{0.9\columnwidth}{!}{\input{sheatBwd_MCSamples05}}}}
\end{center}
    \caption{error estimates at $t=5$}
    \label{fig:sheat_MCerrAtTimes01}
  \end{subfigure}

\begin{center}
  \begin{subfigure}[b]{0.45\textwidth}
{{\huge\resizebox{0.9\columnwidth}{!}{\input{sheatBwd_MCSamplesRelErr}}}}
    \caption{average error estimates over $t=5$}
    \label{fig:sheat_MCrelAveErr}
  \end{subfigure}
\end{center}
\caption{Heat equation: The mean, minimum, and maximum of the quantities \eqref{eq:Err_TimeProbBnd} and \eqref{eq:Err_AveRelErrProbBnd} computed from 100 samples of the learned error estimator are shown. It is observed that there is low variation among the samples of the learned error estimator.} \label{fig:sheat_MC}
\end{figure}
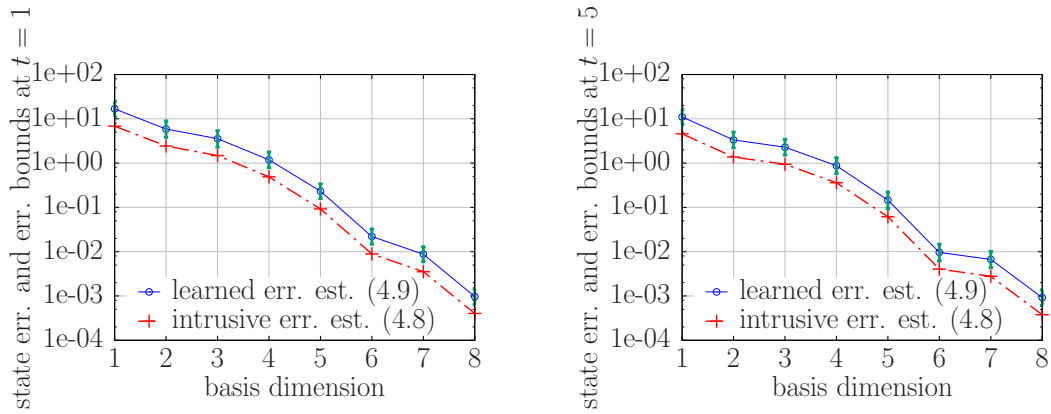
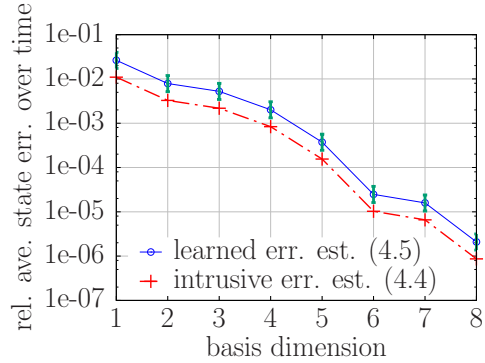

\subsection{Convection-diffusion in a pipe} \label{subsec:ConvDiff2D}

The setup for this problem is first described followed by the numerical results for two types of control inputs and bounds on the output error.

\subsubsection{Setup}

We now consider a parabolic PDE over a 2-D spatial domain according to the convection-diffusion equation. Let $\Tcal = (0,0.5)$ and $\PDEdomain = (0,1) \times (0,0.25)$. For $(x_1,x_2,t) \in \PDEdomain \times \Tcal$, the PDE examined is
\begin{align} \label{eq:ConvDiffPDE}
\hspace{-1in}   \frac{\partial}{\partial t} w(x_1,x_2,t) &= \nabla \cdot (\mu \nabla w(x_1,x_2,t)) - (1,1) \cdot \nabla w(x_1,x_2,t), \\
\hspace{-1in}   w(x_1,x_2,t) & = 0 \text{\,\, for \,\,} (x_1,x_2) \in \partial \Omega \backslash \cup_{i=1}^5 E_i, \notag \\
\hspace{-1in}  \nabla w(x_1,x_2,t) \cdot \mathbf{n} & = u_i(t) \text{\,\, for \,\,} (x_1,x_2) \in E_i, i = 1,\dots, 5, \notag \\
\hspace{-1in}  w(x_1,x_2,0) & = 0. \notag
\end{align}
where the domain $\Omega$ and the segments $E_i$, $i=1,\dots,5$ with Neumann conditions are visualized in Figure~\ref{fig:convdiff2Ddomain}.

\begin{figure}[ht!]
    \centering
    \includegraphics[width=0.6\columnwidth]{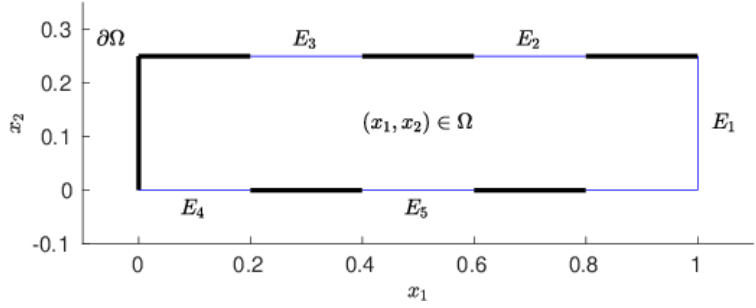}
\caption{Domain $\Omega$ for the convection-diffusion PDE in Section~\ref{subsec:ConvDiff2D} with segments of the boundary with Neumann conditions indicated by thin solid lines.}
\label{fig:convdiff2Ddomain}
\end{figure}

The finite element discretization is constructed using square elements with width $\Delta x_1 = \Delta x_2= 1/75$ and associated linear hat basis functions $\{\varphi_i(x_1,x_2)\}_{i=1}^N$ where $N = 1121$. The continuous-time system resulting from this PDE is
$$\bfM \frac{d \bfw(t)}{dt} = \bfK \bfw(t) + \bfF \bfu(t)$$
where $\bfM$ is the mass matrix as before, $[\bfK]_{ij} = -\mu\int_{\Omega} \nabla \varphi_j \cdot \nabla \varphi_i \,d \bfx - \int_{\Omega} ((1,1) \cdot \nabla \varphi_j) \varphi_i \,d \bfx$ for $i,j=1,\dots,N$ and $[\bfF]_{ij} = \mu \int_{E_j} \varphi_i \,d\bfx$ for $i=1,\dots,N, j=1,\dots,5$. This was then discretized using forward Euler with the time step size $\delta t =10^{-5}$.

Two variants of this problem are investigated in Sections~\ref{subsubsec:ConvDiffEasy} and~\ref{subsubsec:ConvDiffHard} in which we implemented different pairs of control signals $(\bfubasis(t),\bfutest(t))$ in each variation. The same control input $\bfutrain(t)$ is used to solve the optimization problems~\eqref{eq:OpInf} and~\eqref{eq:DataMatErrInf} for both variants which is discretized to obtain $\bfGtrain$. The trajectory $\bfGtrain$ was simulated as follows: for the time points $\{t_k\}_{k=0}^K$, $K = 5\times 10^4$, $[\bfg_k]_j$ is a realization of $Z_k^{(j)} \sim N(0,\sin^2 (j \pi t_k))$ such that $Z_k^{(j)},Z_l^{(j)}$ are independent for $k,l = 1,\dots,K$, $k \neq l$.

\subsubsection{Results for exponentially growing sinusoidal control input} \label{subsubsec:ConvDiffEasy}

The diffusivity parameter in this example is set to $\mu = 0.5$. The basis $\bfV_{\nr}$ is constructed using the control input
$\ubasis_j(t) = \sin(2jt), j =1,\dots,5$ while the control input $\utest_j(t) = e^t\sin(1.75jt), j = 1,\dots,5$ is used for prediction in the online stage. Both of these control inputs are discretized in time using $K=5\times 10^4$ (basis) and $J = 5\times 10^4$ (prediction) intervals of equal width. To visualize trajectories of the high-dimensional system resulting from the control input $\bfutest(t)$, Fig~\ref{fig:convDiffEasySurf} illustrates $w(x_1,x_2,t)$ for $t = 0.1,0.5$.

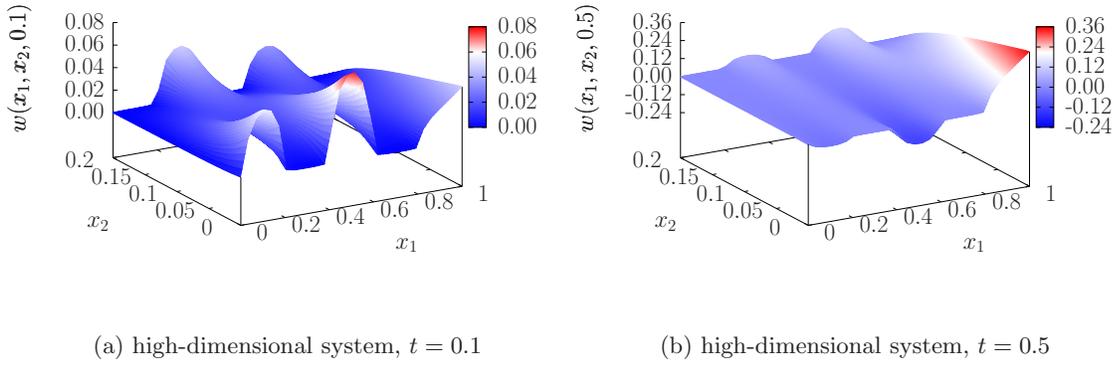
\begin{figure}
  \begin{subfigure}[b]{0.45\textwidth}
    \begin{center}
{{\LARGE\resizebox{0.9\columnwidth}{!}{\input{convDiffAnormEasySurf_t01}}}}
\end{center}
    \caption{high-dimensional system, $t=0.1$}
  \end{subfigure}
  \begin{subfigure}[b]{0.45\textwidth}
    \begin{center}
{{\LARGE\resizebox{0.9\columnwidth}{!}{\input{convDiffAnormEasySurf_t05}}}}
\end{center}
    \caption{high-dimensional system, $t=0.5$}
  \end{subfigure}
  \caption{Convection-diffusion equation: Numerical approximation of the solution to~\eqref{eq:ConvDiffPDE} at times $t = 0.1$ and $t=0.5$ for $\mu=0.5$ and control input $\bfGtest$.}
  \label{fig:convDiffEasySurf}
\end{figure}

We now examine the accuracy of the inferred reduced model and its state error estimate under operator inference by computing the errors listed above. The quantity~\eqref{eq:Err_IntVSNonIntUnitNorm} corresponding to the intrusive  and operator inference approach as a function of the basis dimension $\nr$ is contrasted in Figure~\ref{fig:convDiffeasy_residual}. The plot demonstrates the recovery of the residual norm \eqref{eq:SqResNorm} in the latter method. The reduced system operators for both methods are also almost identical.

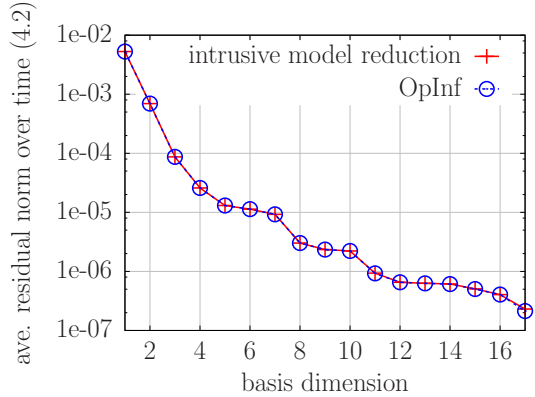
\begin{figure}
    \begin{center}
{{\LARGE\resizebox{0.45\columnwidth}{!}{\input{convDiffAnormEasy_IntVsNonIntResidual}}}}
    \end{center}
    \caption{Convection-diffusion equation (Section~\ref{subsubsec:ConvDiffEasy}): The graph indicates that the residual norm \eqref{eq:SqResNorm} needed for the \emph{a posteriori} estimate \eqref{eq:aposterioribnd} can be recovered under operator inference.}
    \label{fig:convDiffeasy_residual}
\end{figure}

We then investigate the effect of the parameters $\gamma$ and $M$ in the learned error estimator \eqref{eq:RealizedStateBnd} in Figures~\ref{fig:convDiffeasy_varGamma} and \ref{fig:convDiffeasy_varM}. Figures~\ref{fig:convDiffeasy_varGamma01} and~\ref{fig:convDiffeasy_varGamma05} depict the learned reduced model error~\eqref{eq:Err_TimeActualErr} and the intrusive \eqref{eq:Err_TimeActualBnd} and learned (probabilistic) \eqref{eq:Err_TimeProbBnd} error estimates at times $t = 0.1$ and $t=0.5$. Each panel presents 3 realizations of the probabilistic error estimator~\eqref{eq:ProbErrBound} using $M=10$ and $\gamma = 7,20,50$ with their respective lower bound probabilities of $\ProbLB \approx 0.7543, 0.9985, 0.9999$. The same set of realizations of $\Xi_l$ for $l=1,\dots,J$ were utilized for the values of $\gamma$ considered.
For fixed $M$, the learned error estimates become more conservative with respect to the intrusive error estimate in favor of increased confidence in the estimate; cf.~the definition of $\Xi_l$ in Proposition~\ref{prop:GeneralStateErrBnd}. 

\begin{figure}
  \begin{subfigure}[b]{0.45\textwidth}
    \begin{center}
{{\LARGE\resizebox{1\columnwidth}{!}{\input{convDiffAnormEasy_ErrAtTime01VaryGamma}}}}
\end{center}
    \caption{error and error estimates at $t=0.1$ }
    \label{fig:convDiffeasy_varGamma01}
  \end{subfigure}
  \begin{subfigure}[b]{0.45\textwidth}
    \begin{center}
{{\LARGE\resizebox{1\columnwidth}{!}{\input{convDiffAnormEasy_ErrAtTime05VaryGamma}}}}
\end{center}
    \caption{error and error estimates at $t=0.5$ }
    \label{fig:convDiffeasy_varGamma05}
  \end{subfigure}
  \caption{Convection-diffusion equation (Section~\ref{subsubsec:ConvDiffEasy}). Increasing $\gamma$ leads to a more conservative learned error estimate~\eqref{eq:RealizedStateBnd} for fixed $M,J$ with respect to the intrusive error estimate. This simultaneously corresponds to a larger lower bound probability $\ProbLB$. The parameters used were
  $M=10$ and $\gamma = 7,20,50$ for which $\ProbLB \approx 0.75,0.99,0.99$.}
  \label{fig:convDiffeasy_varGamma}
\end{figure}
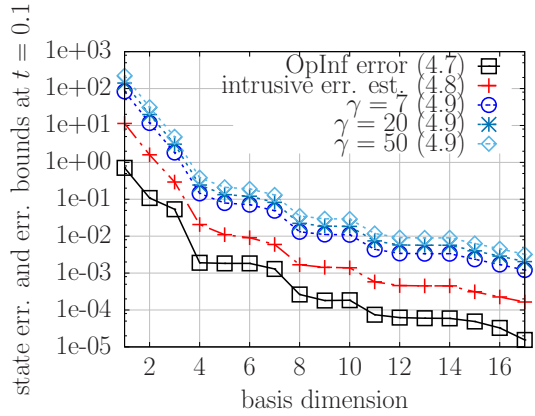
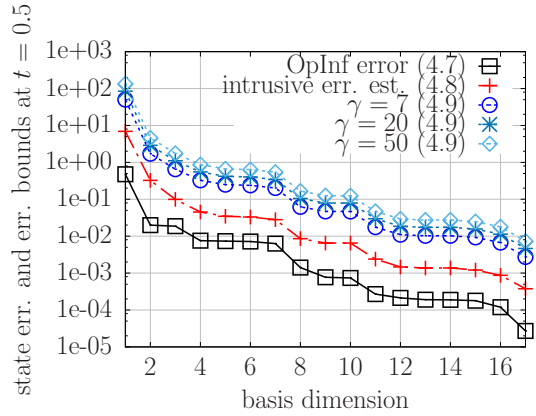

Figure~\ref{fig:convDiffeasy_varM} plots the same quantities shown in Figure~\ref{fig:convDiffeasy_varGamma} but for the parameters $\gamma = 1$ and $M = 35,100,500$, i.e. $M$ is varied while $\gamma$ is fixed. The lower bound probability values for each $M$ are $\ProbLB \approx 0.9212, 0.9999, 1$. The sets of the $M=35,100,500$ realizations of $\bfTheta^{(l)}$ for $l=1,\dots,J$ are nested. 
For this example, increasing $M$ did lead only to slight changes in the learned error estimate. The influence of $M$ is more difficult to gauge numerically since the maximum of $\{\bftheta_i^{(l)}\}_{i=1}^M$ may not differ substantially as a function of $M$.
The results indicate that in this example, for a fixed value for $\ProbLB$, it is more favorable to choose a larger value of $M$ and a smaller value of $\gamma$ to obtain a tighter learned error estimate that is close to the error estimate from intrusive model reduction with a high confidence in the estimate.

\begin{figure}
  \begin{subfigure}[b]{0.45\textwidth}
    \begin{center}
{{\LARGE\resizebox{1\columnwidth}{!}{\input{convDiffAnormEasy_ErrAtTime01VaryM}}}}
\end{center}
    \caption{error and error estimates at $t=0.1$}
    \label{fig:convDiffeasy_varM01}
  \end{subfigure}
  \begin{subfigure}[b]{0.45\textwidth}
    \begin{center}
{{\LARGE\resizebox{1\columnwidth}{!}{\input{convDiffAnormEasy_ErrAtTime05VaryM}}}}
\end{center}
    \caption{error and error estimates at $t=0.5$}
    \label{fig:convDiffeasy_varM05}
  \end{subfigure}
  \caption{Convection-diffusion equation (Section~\ref{subsubsec:ConvDiffEasy}). Increasing $M$ only slightly changes the learned error estimate \eqref{eq:RealizedStateBnd} for fixed $\gamma,J$ in this example. The parameters used were
  $\gamma=1$ and $M = 35,100,500$ for which $\ProbLB \approx 0.92,0.99,1$.}
  \label{fig:convDiffeasy_varM}
\end{figure}
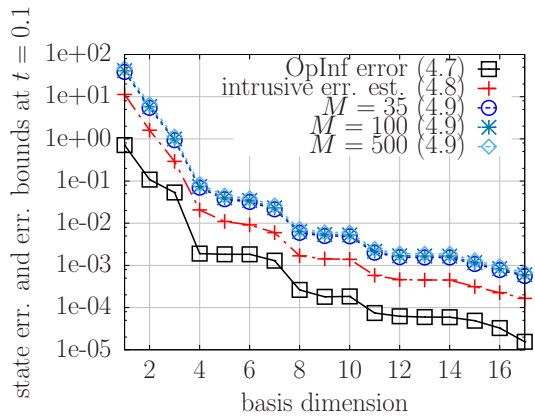
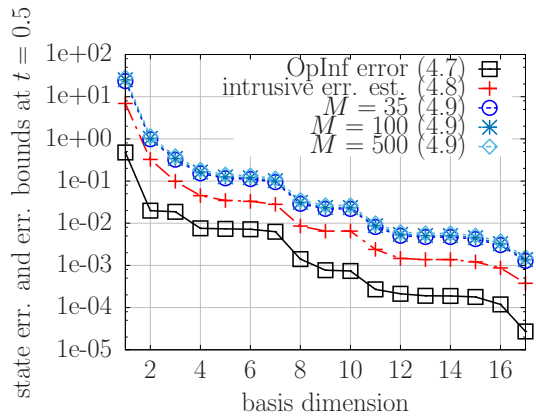

We now assess the variation in the realizations of the learned error estimator. The simulation is carried out for $\gamma=1,M=35$. We generated 50 sets of $M = 35$ realizations of $\bfTheta^{(l)}$ to produce 50 realizations of $\Xi_l$ and
of the learned error estimate~\eqref{eq:Err_TimeProbBnd}. The mean (solid) of the 50 realizations of \eqref{eq:Err_TimeProbBnd} for $t=0.1$ and $t=0.5$ are illustrated in the panels of Figure~\ref{fig:convDiffeasy_ErrBar} together with their minimum and maximum values (vertical bars). For reference,  the error estimate \eqref{eq:Err_TimeActualErr} under intrusive model reduction is also included. The plots show that the variation among samples of the learned error estimator is low.

\begin{figure}
  \begin{subfigure}[b]{0.45\textwidth}
    \begin{center}
{{\LARGE\resizebox{1\columnwidth}{!}{\input{convDiffAnormEasyMCSamples01}}}}
\end{center}
    \caption{error estimates at $t=0.1$}
    \label{fig:convDiffeasy_ErrBar01}
  \end{subfigure}
  \begin{subfigure}[b]{0.45\textwidth}
    \begin{center}
{{\LARGE\resizebox{1\columnwidth}{!}{\input{convDiffAnormEasyMCSamples05}}}}
\end{center}
    \caption{error estimates at $t=0.5$}
    \label{fig:convDiffeasy_ErrBar05}
  \end{subfigure}
  \caption{Convection-diffusion equation (Section~\ref{subsubsec:ConvDiffEasy}). The mean, minimum, and maximum of~\eqref{eq:Err_TimeProbBnd} based on 50 samples of the learned error estimator are depicted together with the error estimate from intrusive model reduction. The parameters used were $\gamma=1,M=35$. There is a low variation among the samples of the learned error estimator.}
  \label{fig:convDiffeasy_ErrBar}
\end{figure}
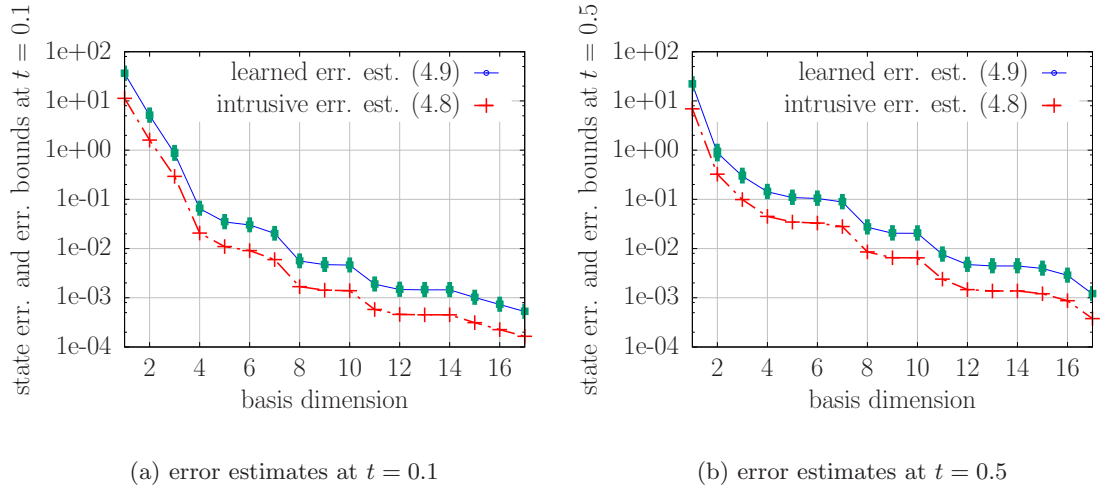

\subsubsection{Results for sinusoidal control input} \label{subsubsec:ConvDiffHard}

In this case, the diffusivity parameter is set to $\mu=1$. The control input $\bfubasis(t)$ for constructing the basis $\bfV_{\nr}$ consists of
\begin{align} \label{eq:SysSignal}
& \ubasis_1(t) = 5t \sin (\pi t)\\
& \ubasis_2(t) = e^{5t} \sin(2 \pi t)\notag \\
& \ubasis_3(t) = \sqrt{3+t^2} \sin (3 \pi t) \notag \\
& \ubasis_4(t) = 50t^2 \sin (4 \pi t) \notag \\
& \ubasis_5(t) = e^{2t} \sin(5 \pi t) \notag
\end{align}
while the components of the control input $\bfutest(t)$ for prediction were chosen as $\utest_j(t) =  \sin(j\pi t z_j), j =1,\dots,5$ where $z_j$ is a realization of a $N(0,1)$ random variable.

Figure~\ref{fig:convDiffHard_residual_relAveErr} summarizes the predictive capabilities of operator inference. The quantity \eqref{eq:Err_IntVSNonIntUnitNorm} is plotted in Figure~\ref{fig:convDiffHard_residual} wherein we see concordance between the intrusive and operator inference approaches. Figure~\ref{fig:convDiffHard_relAveErr} contains graphs of the learned reduced model error \eqref{eq:Err_AveRelErrActualErr} and the intrusive \eqref{eq:Err_AveRelErrActualBnd} and learned (probabilistic) \eqref{eq:Err_AveRelErrProbBnd} error estimates in which 1 sample of the learned error estimator was generated. The parameters for the learned error estimator were set to $\gamma = 1,M=40, J =5 \times 10^4$ so that $\ProbLB(\gamma,M,J) \approx 0.9883$. The learned error estimate is close to the error estimate given by the intrusive approach.

\begin{figure}
  \begin{subfigure}[b]{0.45\textwidth}
    \begin{center}
{{\LARGE\resizebox{1\columnwidth}{!}{\input{convDiffAnormHard_IntVsNonIntResidual}}}}
\end{center}
    \caption{test, intrusive vs operator inference}
    \label{fig:convDiffHard_residual}
  \end{subfigure}
  \begin{subfigure}[b]{0.45\textwidth}
    \begin{center}
{{\LARGE\resizebox{1\columnwidth}{!}{\input{convDiffAnormHard_RelAveErr}}}}
\end{center}
    \caption{$\gamma=1,M=40, J=5\times 10^4, \ProbLB \approx 0.9883$}
    \label{fig:convDiffHard_relAveErr}
  \end{subfigure}
  \caption{Convection-diffusion equation (Section~\ref{subsubsec:ConvDiffHard}). The panels show that operator inference recovers the residual norm necessary for deriving error estimates of the state. Furthermore, the learned state error estimate is only slightly higher than the error estimate provided by the intrusive method.}
  \label{fig:convDiffHard_residual_relAveErr}
\end{figure}
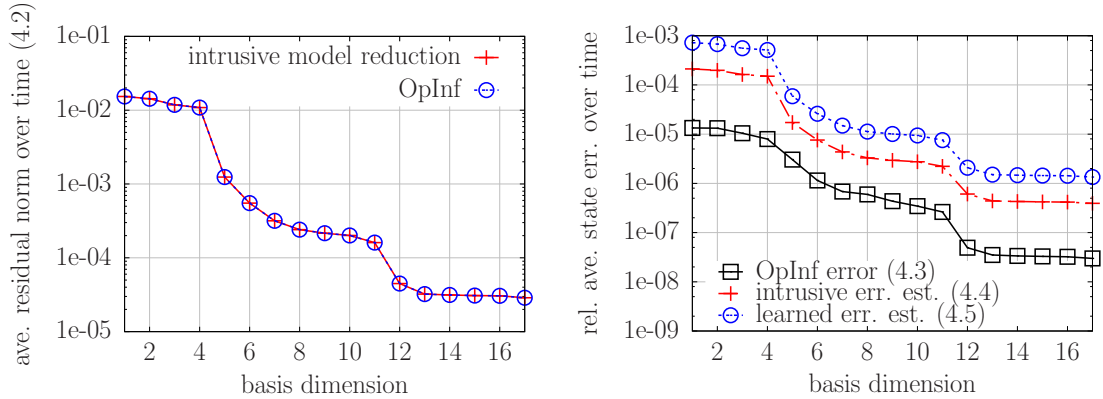

\subsubsection{Results for bound on output error}
We now study the efficiency of  the learned error estimator for the state in constructing bounds for an output. We resume the setup in the previous subsection wherein the control input is sinusoidal. We consider two quantities of interest for this case, namely, $y_k^{(j)} = \bfC^{(j)} \bfw_k$ for $j \in \{1,2\}$ with the control input $\bfGtest$. The matrices $\bfC^{(1)}$ and $\bfC^{(2)}$ are defined as follows: The first output is the average of the state components at each time $\bfw_k$ which is
        \begin{align} \label{eq:OutputAve}
            y_k^{(1)} = \bfC^{(1)} \bfw_k \quad \text{where} \quad \bfC^{(1)} = \biggl[\frac{1}{N},\dots,\frac{1}{N} \biggr] \in \R^{1 \times N}.
        \end{align}
The second output is the integral of the finite element approximation to $w(\bfx,t)$ over the edge $E_5$ at each time given by
        \begin{align} \label{eq:OutputIntegral}
            y_k^{(2)} = \bfC^{(2)} \bfw_k \quad \text{where} \quad \bfC^{(2)} = \biggl[\int_{E_5} \varphi_1 \,d\Gamma,\dots,\int_{E_5} \varphi_N \,d\Gamma \biggr] \in \R^{1 \times N}.
        \end{align}

The output $y_k$ and its bounds $\ROMOutput_k \pm \OutErrBnd_k$ over time are displayed in Figures~\ref{fig:convDiffHard_Output1} (first output) and~\ref{fig:convDiffHard_Output2} (second output); cf.~Remark~\ref{rm:Output}. These quantities were sketched for $\nr \in \{7,12,17\}$ basis dimensions in the first output and $\nr \in \{5,10,15\}$ in the second output. The output bound $\OutErrBnd_k$ is computed via the learned error estimator for the state $\StErrBnd_k$ with the same parameters above, i.e. $\gamma=1,M=40,J = 5\times 10^4$. The panels show that increasing $n$ yields a decrease in the output bound width  $2\OutErrBnd_k$ over time, i.e. the bounds are sharper with respect to the output value. This is supported by Figure~\ref{fig:convDiffHard_residual_relAveErr} which demonstrates decrease of the learned state error estimate as a function of the basis dimension.

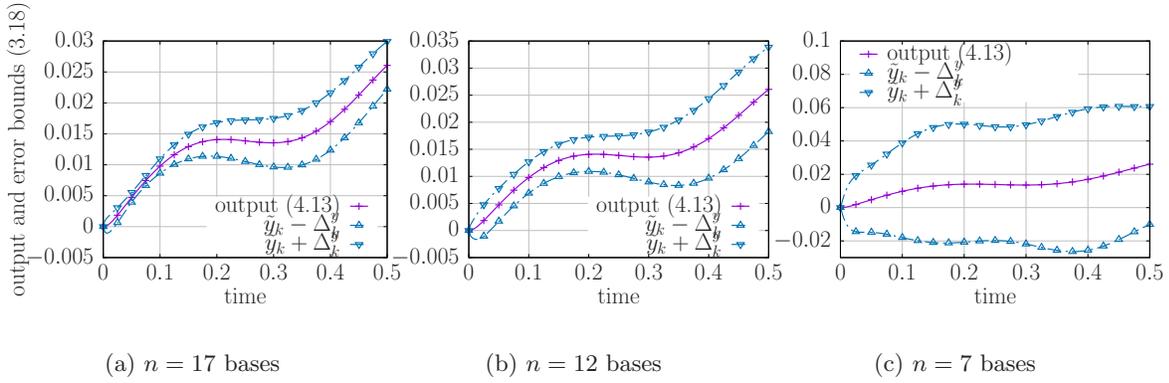
\begin{figure}
  \begin{subfigure}[b]{0.3\textwidth}
    \begin{center}
{{\huge\resizebox{1.1\columnwidth}{!}{\input{convDiffAnormHard_Output1Basis17}}}}
\end{center}
    \caption{$n=17$ bases}
    \label{fig:convDiffHard_Output1_17basis}
  \end{subfigure}
  \begin{subfigure}[b]{0.3\textwidth}
   \begin{center}
{{\huge\resizebox{1.1\columnwidth}{!}{\input{convDiffAnormHard_Output1Basis12}}}}
\end{center}
    \caption{$n=12$ bases}
    \label{fig:convDiffHard_Output1_12basis}
  \end{subfigure}
  \begin{subfigure}[b]{0.3\textwidth}
\begin{center}
{{\huge\resizebox{1.1\columnwidth}{!}{\input{convDiffAnormHard_Output1Basis7}}}}
\end{center}
    \caption{$n=7$ bases}
    \label{fig:convDiffHard_Output1_7basis}
  \end{subfigure}
  \caption{Convection-diffusion equation (Section~\ref{subsubsec:ConvDiffHard}). The panels show the predictive capability of the learned error estimator for the state error in constructing lower and upper bounds for the output (average of state components). The bounds correctly indicate that the errors of the predicted reduced-model outputs decreases if the basis dimension is increased.}
  \label{fig:convDiffHard_Output1}
\end{figure}

\begin{figure}
  \begin{subfigure}[b]{0.3\textwidth}
    \begin{center}
{{\huge\resizebox{1.1\columnwidth}{!}{\input{convDiffAnormHard_Output2Basis15}}}}
\end{center}
    \caption{$n=15$ bases}
    \label{fig:convDiffHard_Output2_15basis}
  \end{subfigure}
  \begin{subfigure}[b]{0.3\textwidth}
   \begin{center}
{{\huge\resizebox{1.1\columnwidth}{!}{\input{convDiffAnormHard_Output2Basis10}}}}
\end{center}
    \caption{$n=10$ bases}
    \label{fig:convDiffHard_Output2_10basis}
  \end{subfigure}
  \begin{subfigure}[b]{0.3\textwidth}
\begin{center}
{{\huge\resizebox{1.1\columnwidth}{!}{\input{convDiffAnormHard_Output2Basis5}}}}
\end{center}
    \caption{$n=5$ bases}
    \label{fig:convDiffHard_Output2_5basis}
  \end{subfigure}
  \caption{Convection-diffusion equation (Section~\ref{subsubsec:ConvDiffHard}). Similar behavior as described in Figure~\ref{fig:convDiffHard_Output1} is observed in these panels for the quantity of interest based on the integral over the Neumann boundary.}
  \label{fig:convDiffHard_Output2}
\end{figure}
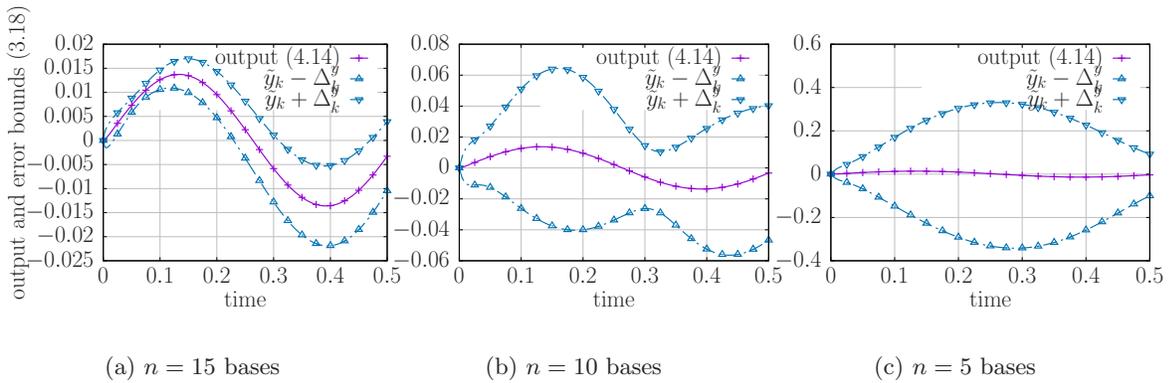

\section{Conclusions} \label{sec:Concl}
This work proposes a probabilistic \emph{a posteriori} error estimator that is applicable with non-intrusive model reduction under certain assumptions. The key is that quantities that are necessary for error estimators developed for intrusive model reduction can be derived via least-squares regression from input and solution trajectories whereas other quantities that are necessary can be bounded in a probabilistic sense by sampling the high-dimensional system in a judicious and black-box way. The learned estimators can be used to rigorously upper bound the error of reduced models learned from data for initial conditions and inputs that are different than during training (offline phase).
Thus, the proposed approach establishes trust in decisions made from data by realizing the full workflow from data to reduced models to certified predictions.

\section*{Acknowledgments}
This work was partially supported by US Department of Energy, Office of Advanced Scientific Computing Research, Applied Mathematics Program (Program Manager Dr. Steven Lee), DOE Award DESC0019334, and by the National Science Foundation under Grant No.~1901091.

\bibliographystyle{abbrv}
\bibliography{linest.bib}

\end{document}

%% file: sheatBwd_IntNonIntVsFOM.tex
\begingroup
  \makeatletter
  \providecommand\color[2][]{%
    \GenericError{(gnuplot) \space\space\space\@spaces}{%
      Package color not loaded in conjunction with
      terminal option `colourtext'%
    }{See the gnuplot documentation for explanation.%
    }{Either use 'blacktext' in gnuplot or load the package
      color.sty in LaTeX.}%
    \renewcommand\color[2][]{}%
  }%
  \providecommand\includegraphics[2][]{%
    \GenericError{(gnuplot) \space\space\space\@spaces}{%
      Package graphicx or graphics not loaded%
    }{See the gnuplot documentation for explanation.%
    }{The gnuplot epslatex terminal needs graphicx.sty or graphics.sty.}%
    \renewcommand\includegraphics[2][]{}%
  }%
  \providecommand\rotatebox[2]{#2}%
  \@ifundefined{ifGPcolor}{%
    \newif\ifGPcolor
    \GPcolortrue
  }{}%
  \@ifundefined{ifGPblacktext}{%
    \newif\ifGPblacktext
    \GPblacktexttrue
  }{}%
  \let\gplgaddtomacro\g@addto@macro
  \gdef\gplbacktext{}%
  \gdef\gplfronttext{}%
  \makeatother
  \ifGPblacktext
    \def\colorrgb#1{}%
    \def\colorgray#1{}%
  \else
    \ifGPcolor
      \def\colorrgb#1{\color[rgb]{#1}}%
      \def\colorgray#1{\color[gray]{#1}}%
      \expandafter\def\csname LTw\endcsname{\color{white}}%
      \expandafter\def\csname LTb\endcsname{\color{black}}%
      \expandafter\def\csname LTa\endcsname{\color{black}}%
      \expandafter\def\csname LT0\endcsname{\color[rgb]{1,0,0}}%
      \expandafter\def\csname LT1\endcsname{\color[rgb]{0,1,0}}%
      \expandafter\def\csname LT2\endcsname{\color[rgb]{0,0,1}}%
      \expandafter\def\csname LT3\endcsname{\color[rgb]{1,0,1}}%
      \expandafter\def\csname LT4\endcsname{\color[rgb]{0,1,1}}%
      \expandafter\def\csname LT5\endcsname{\color[rgb]{1,1,0}}%
      \expandafter\def\csname LT6\endcsname{\color[rgb]{0,0,0}}%
      \expandafter\def\csname LT7\endcsname{\color[rgb]{1,0.3,0}}%
      \expandafter\def\csname LT8\endcsname{\color[rgb]{0.5,0.5,0.5}}%
    \else
      \def\colorrgb#1{\color{black}}%
      \def\colorgray#1{\color[gray]{#1}}%
      \expandafter\def\csname LTw\endcsname{\color{white}}%
      \expandafter\def\csname LTb\endcsname{\color{black}}%
      \expandafter\def\csname LTa\endcsname{\color{black}}%
      \expandafter\def\csname LT0\endcsname{\color{black}}%
      \expandafter\def\csname LT1\endcsname{\color{black}}%
      \expandafter\def\csname LT2\endcsname{\color{black}}%
      \expandafter\def\csname LT3\endcsname{\color{black}}%
      \expandafter\def\csname LT4\endcsname{\color{black}}%
      \expandafter\def\csname LT5\endcsname{\color{black}}%
      \expandafter\def\csname LT6\endcsname{\color{black}}%
      \expandafter\def\csname LT7\endcsname{\color{black}}%
      \expandafter\def\csname LT8\endcsname{\color{black}}%
    \fi
  \fi
    \setlength{\unitlength}{0.0500bp}%
    \ifx\gptboxheight\undefined%
      \newlength{\gptboxheight}%
      \newlength{\gptboxwidth}%
      \newsavebox{\gptboxtext}%
    \fi%
    \setlength{\fboxrule}{0.5pt}%
    \setlength{\fboxsep}{1pt}%
\begin{picture}(7200.00,5040.00)%
    \gplgaddtomacro\gplbacktext{%
      \csname LTb\endcsname%
      \put(1372,896){\makebox(0,0)[r]{\strut{}1e-05}}%
      \csname LTb\endcsname%
      \put(1372,1657){\makebox(0,0)[r]{\strut{}1e-04}}%
      \csname LTb\endcsname%
      \put(1372,2419){\makebox(0,0)[r]{\strut{}1e-03}}%
      \csname LTb\endcsname%
      \put(1372,3180){\makebox(0,0)[r]{\strut{}1e-02}}%
      \csname LTb\endcsname%
      \put(1372,3942){\makebox(0,0)[r]{\strut{}1e-01}}%
      \csname LTb\endcsname%
      \put(1372,4703){\makebox(0,0)[r]{\strut{}1e+00}}%
      \csname LTb\endcsname%
      \put(1540,616){\makebox(0,0){\strut{}$1$}}%
      \csname LTb\endcsname%
      \put(2276,616){\makebox(0,0){\strut{}$2$}}%
      \csname LTb\endcsname%
      \put(3013,616){\makebox(0,0){\strut{}$3$}}%
      \csname LTb\endcsname%
      \put(3749,616){\makebox(0,0){\strut{}$4$}}%
      \csname LTb\endcsname%
      \put(4486,616){\makebox(0,0){\strut{}$5$}}%
      \csname LTb\endcsname%
      \put(5222,616){\makebox(0,0){\strut{}$6$}}%
      \csname LTb\endcsname%
      \put(5959,616){\makebox(0,0){\strut{}$7$}}%
      \csname LTb\endcsname%
      \put(6695,616){\makebox(0,0){\strut{}$8$}}%
    }%
    \gplgaddtomacro\gplfronttext{%
      \csname LTb\endcsname%
      \put(224,2799){\rotatebox{-270}{\makebox(0,0){\strut{}error of reduced model  \eqref{eq:Err_IntNonIntVSFOM}}}}%
      \put(4117,196){\makebox(0,0){\strut{}basis dimension}}%
      \csname LTb\endcsname%
      \put(5876,4430){\makebox(0,0)[r]{\strut{}intrusive model reduction}}%
      \csname LTb\endcsname%
      \put(5876,4010){\makebox(0,0)[r]{\strut{}OpInf}}%
    }%
    \gplbacktext
    \put(0,0){\includegraphics{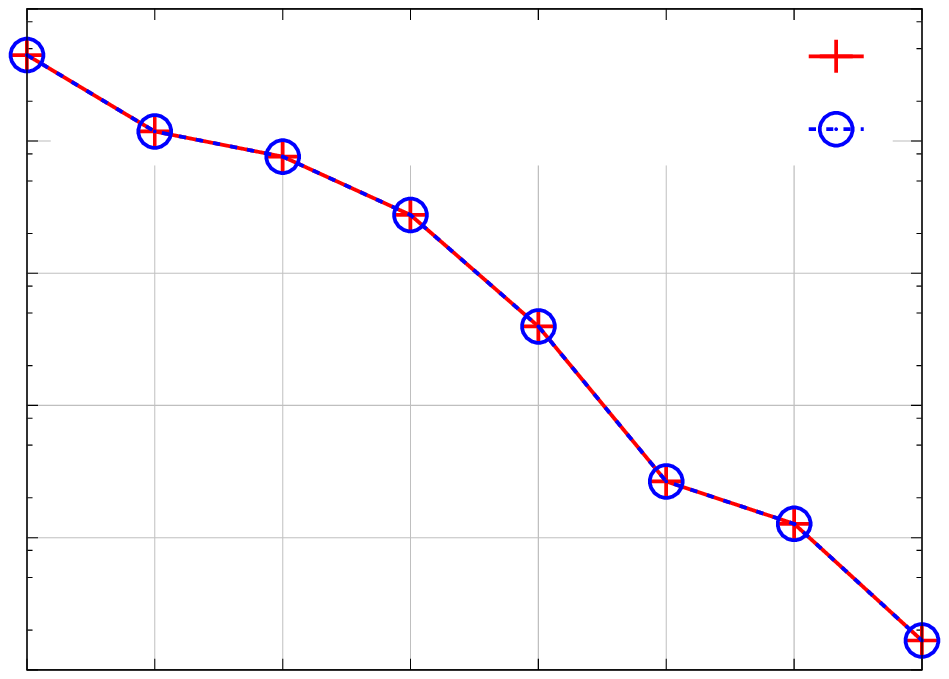}}%
    \gplfronttext
  \end{picture}%
\endgroup

%% file: sheatBwd_IntVsNonIntResidual.tex
\begingroup
  \makeatletter
  \providecommand\color[2][]{%
    \GenericError{(gnuplot) \space\space\space\@spaces}{%
      Package color not loaded in conjunction with
      terminal option `colourtext'%
    }{See the gnuplot documentation for explanation.%
    }{Either use 'blacktext' in gnuplot or load the package
      color.sty in LaTeX.}%
    \renewcommand\color[2][]{}%
  }%
  \providecommand\includegraphics[2][]{%
    \GenericError{(gnuplot) \space\space\space\@spaces}{%
      Package graphicx or graphics not loaded%
    }{See the gnuplot documentation for explanation.%
    }{The gnuplot epslatex terminal needs graphicx.sty or graphics.sty.}%
    \renewcommand\includegraphics[2][]{}%
  }%
  \providecommand\rotatebox[2]{#2}%
  \@ifundefined{ifGPcolor}{%
    \newif\ifGPcolor
    \GPcolortrue
  }{}%
  \@ifundefined{ifGPblacktext}{%
    \newif\ifGPblacktext
    \GPblacktexttrue
  }{}%
  \let\gplgaddtomacro\g@addto@macro
  \gdef\gplbacktext{}%
  \gdef\gplfronttext{}%
  \makeatother
  \ifGPblacktext
    \def\colorrgb#1{}%
    \def\colorgray#1{}%
  \else
    \ifGPcolor
      \def\colorrgb#1{\color[rgb]{#1}}%
      \def\colorgray#1{\color[gray]{#1}}%
      \expandafter\def\csname LTw\endcsname{\color{white}}%
      \expandafter\def\csname LTb\endcsname{\color{black}}%
      \expandafter\def\csname LTa\endcsname{\color{black}}%
      \expandafter\def\csname LT0\endcsname{\color[rgb]{1,0,0}}%
      \expandafter\def\csname LT1\endcsname{\color[rgb]{0,1,0}}%
      \expandafter\def\csname LT2\endcsname{\color[rgb]{0,0,1}}%
      \expandafter\def\csname LT3\endcsname{\color[rgb]{1,0,1}}%
      \expandafter\def\csname LT4\endcsname{\color[rgb]{0,1,1}}%
      \expandafter\def\csname LT5\endcsname{\color[rgb]{1,1,0}}%
      \expandafter\def\csname LT6\endcsname{\color[rgb]{0,0,0}}%
      \expandafter\def\csname LT7\endcsname{\color[rgb]{1,0.3,0}}%
      \expandafter\def\csname LT8\endcsname{\color[rgb]{0.5,0.5,0.5}}%
    \else
      \def\colorrgb#1{\color{black}}%
      \def\colorgray#1{\color[gray]{#1}}%
      \expandafter\def\csname LTw\endcsname{\color{white}}%
      \expandafter\def\csname LTb\endcsname{\color{black}}%
      \expandafter\def\csname LTa\endcsname{\color{black}}%
      \expandafter\def\csname LT0\endcsname{\color{black}}%
      \expandafter\def\csname LT1\endcsname{\color{black}}%
      \expandafter\def\csname LT2\endcsname{\color{black}}%
      \expandafter\def\csname LT3\endcsname{\color{black}}%
      \expandafter\def\csname LT4\endcsname{\color{black}}%
      \expandafter\def\csname LT5\endcsname{\color{black}}%
      \expandafter\def\csname LT6\endcsname{\color{black}}%
      \expandafter\def\csname LT7\endcsname{\color{black}}%
      \expandafter\def\csname LT8\endcsname{\color{black}}%
    \fi
  \fi
    \setlength{\unitlength}{0.0500bp}%
    \ifx\gptboxheight\undefined%
      \newlength{\gptboxheight}%
      \newlength{\gptboxwidth}%
      \newsavebox{\gptboxtext}%
    \fi%
    \setlength{\fboxrule}{0.5pt}%
    \setlength{\fboxsep}{1pt}%
\begin{picture}(7200.00,5040.00)%
    \gplgaddtomacro\gplbacktext{%
      \csname LTb\endcsname%
      \put(1372,896){\makebox(0,0)[r]{\strut{}1e-05}}%
      \csname LTb\endcsname%
      \put(1372,1657){\makebox(0,0)[r]{\strut{}1e-04}}%
      \csname LTb\endcsname%
      \put(1372,2419){\makebox(0,0)[r]{\strut{}1e-03}}%
      \csname LTb\endcsname%
      \put(1372,3180){\makebox(0,0)[r]{\strut{}1e-02}}%
      \csname LTb\endcsname%
      \put(1372,3942){\makebox(0,0)[r]{\strut{}1e-01}}%
      \csname LTb\endcsname%
      \put(1372,4703){\makebox(0,0)[r]{\strut{}1e+00}}%
      \csname LTb\endcsname%
      \put(1540,616){\makebox(0,0){\strut{}$1$}}%
      \csname LTb\endcsname%
      \put(2276,616){\makebox(0,0){\strut{}$2$}}%
      \csname LTb\endcsname%
      \put(3013,616){\makebox(0,0){\strut{}$3$}}%
      \csname LTb\endcsname%
      \put(3749,616){\makebox(0,0){\strut{}$4$}}%
      \csname LTb\endcsname%
      \put(4486,616){\makebox(0,0){\strut{}$5$}}%
      \csname LTb\endcsname%
      \put(5222,616){\makebox(0,0){\strut{}$6$}}%
      \csname LTb\endcsname%
      \put(5959,616){\makebox(0,0){\strut{}$7$}}%
      \csname LTb\endcsname%
      \put(6695,616){\makebox(0,0){\strut{}$8$}}%
    }%
    \gplgaddtomacro\gplfronttext{%
      \csname LTb\endcsname%
      \put(224,2799){\rotatebox{-270}{\makebox(0,0){\strut{}ave. residual norm over time \eqref{eq:Err_IntVSNonIntUnitNorm}}}}%
      \put(4117,196){\makebox(0,0){\strut{}basis dimension}}%
      \csname LTb\endcsname%
      \put(5876,4430){\makebox(0,0)[r]{\strut{}intrusive model reduction}}%
      \csname LTb\endcsname%
      \put(5876,4010){\makebox(0,0)[r]{\strut{}OpInf}}%
    }%
    \gplbacktext
    \put(0,0){\includegraphics{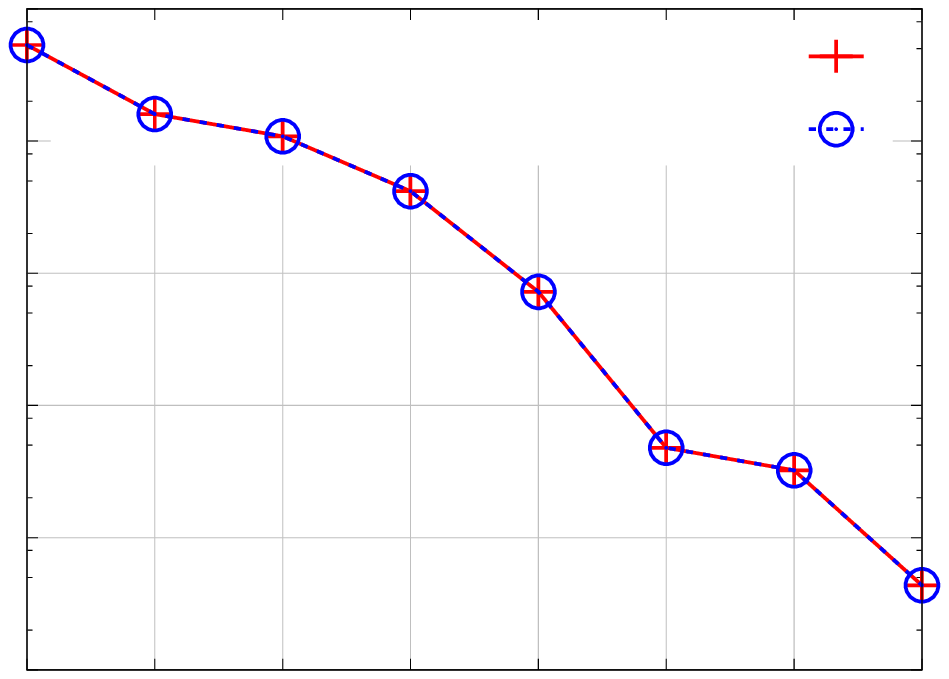}}%
    \gplfronttext
  \end{picture}%
\endgroup

%% file: sheatBwd_ErrAtTimes1.tex
\begingroup
  \makeatletter
  \providecommand\color[2][]{%
    \GenericError{(gnuplot) \space\space\space\@spaces}{%
      Package color not loaded in conjunction with
      terminal option `colourtext'%
    }{See the gnuplot documentation for explanation.%
    }{Either use 'blacktext' in gnuplot or load the package
      color.sty in LaTeX.}%
    \renewcommand\color[2][]{}%
  }%
  \providecommand\includegraphics[2][]{%
    \GenericError{(gnuplot) \space\space\space\@spaces}{%
      Package graphicx or graphics not loaded%
    }{See the gnuplot documentation for explanation.%
    }{The gnuplot epslatex terminal needs graphicx.sty or graphics.sty.}%
    \renewcommand\includegraphics[2][]{}%
  }%
  \providecommand\rotatebox[2]{#2}%
  \@ifundefined{ifGPcolor}{%
    \newif\ifGPcolor
    \GPcolortrue
  }{}%
  \@ifundefined{ifGPblacktext}{%
    \newif\ifGPblacktext
    \GPblacktexttrue
  }{}%
  \let\gplgaddtomacro\g@addto@macro
  \gdef\gplbacktext{}%
  \gdef\gplfronttext{}%
  \makeatother
  \ifGPblacktext
    \def\colorrgb#1{}%
    \def\colorgray#1{}%
  \else
    \ifGPcolor
      \def\colorrgb#1{\color[rgb]{#1}}%
      \def\colorgray#1{\color[gray]{#1}}%
      \expandafter\def\csname LTw\endcsname{\color{white}}%
      \expandafter\def\csname LTb\endcsname{\color{black}}%
      \expandafter\def\csname LTa\endcsname{\color{black}}%
      \expandafter\def\csname LT0\endcsname{\color[rgb]{1,0,0}}%
      \expandafter\def\csname LT1\endcsname{\color[rgb]{0,1,0}}%
      \expandafter\def\csname LT2\endcsname{\color[rgb]{0,0,1}}%
      \expandafter\def\csname LT3\endcsname{\color[rgb]{1,0,1}}%
      \expandafter\def\csname LT4\endcsname{\color[rgb]{0,1,1}}%
      \expandafter\def\csname LT5\endcsname{\color[rgb]{1,1,0}}%
      \expandafter\def\csname LT6\endcsname{\color[rgb]{0,0,0}}%
      \expandafter\def\csname LT7\endcsname{\color[rgb]{1,0.3,0}}%
      \expandafter\def\csname LT8\endcsname{\color[rgb]{0.5,0.5,0.5}}%
    \else
      \def\colorrgb#1{\color{black}}%
      \def\colorgray#1{\color[gray]{#1}}%
      \expandafter\def\csname LTw\endcsname{\color{white}}%
      \expandafter\def\csname LTb\endcsname{\color{black}}%
      \expandafter\def\csname LTa\endcsname{\color{black}}%
      \expandafter\def\csname LT0\endcsname{\color{black}}%
      \expandafter\def\csname LT1\endcsname{\color{black}}%
      \expandafter\def\csname LT2\endcsname{\color{black}}%
      \expandafter\def\csname LT3\endcsname{\color{black}}%
      \expandafter\def\csname LT4\endcsname{\color{black}}%
      \expandafter\def\csname LT5\endcsname{\color{black}}%
      \expandafter\def\csname LT6\endcsname{\color{black}}%
      \expandafter\def\csname LT7\endcsname{\color{black}}%
      \expandafter\def\csname LT8\endcsname{\color{black}}%
    \fi
  \fi
    \setlength{\unitlength}{0.0500bp}%
    \ifx\gptboxheight\undefined%
      \newlength{\gptboxheight}%
      \newlength{\gptboxwidth}%
      \newsavebox{\gptboxtext}%
    \fi%
    \setlength{\fboxrule}{0.5pt}%
    \setlength{\fboxsep}{1pt}%
\begin{picture}(7200.00,5040.00)%
    \gplgaddtomacro\gplbacktext{%
      \csname LTb\endcsname%
      \put(1372,896){\makebox(0,0)[r]{\strut{}1e-05}}%
      \csname LTb\endcsname%
      \put(1372,1319){\makebox(0,0)[r]{\strut{}1e-04}}%
      \csname LTb\endcsname%
      \put(1372,1742){\makebox(0,0)[r]{\strut{}1e-03}}%
      \csname LTb\endcsname%
      \put(1372,2165){\makebox(0,0)[r]{\strut{}1e-02}}%
      \csname LTb\endcsname%
      \put(1372,2588){\makebox(0,0)[r]{\strut{}1e-01}}%
      \csname LTb\endcsname%
      \put(1372,3011){\makebox(0,0)[r]{\strut{}1e+00}}%
      \csname LTb\endcsname%
      \put(1372,3434){\makebox(0,0)[r]{\strut{}1e+01}}%
      \csname LTb\endcsname%
      \put(1372,3857){\makebox(0,0)[r]{\strut{}1e+02}}%
      \csname LTb\endcsname%
      \put(1372,4280){\makebox(0,0)[r]{\strut{}1e+03}}%
      \csname LTb\endcsname%
      \put(1372,4703){\makebox(0,0)[r]{\strut{}1e+04}}%
      \csname LTb\endcsname%
      \put(1540,616){\makebox(0,0){\strut{}$1$}}%
      \csname LTb\endcsname%
      \put(2276,616){\makebox(0,0){\strut{}$2$}}%
      \csname LTb\endcsname%
      \put(3013,616){\makebox(0,0){\strut{}$3$}}%
      \csname LTb\endcsname%
      \put(3749,616){\makebox(0,0){\strut{}$4$}}%
      \csname LTb\endcsname%
      \put(4486,616){\makebox(0,0){\strut{}$5$}}%
      \csname LTb\endcsname%
      \put(5222,616){\makebox(0,0){\strut{}$6$}}%
      \csname LTb\endcsname%
      \put(5959,616){\makebox(0,0){\strut{}$7$}}%
      \csname LTb\endcsname%
      \put(6695,616){\makebox(0,0){\strut{}$8$}}%
    }%
    \gplgaddtomacro\gplfronttext{%
      \csname LTb\endcsname%
      \put(224,2799){\rotatebox{-270}{\makebox(0,0){\strut{}state err. and err. bounds at $t=1$}}}%
      \put(4117,196){\makebox(0,0){\strut{}basis dimension}}%
      \csname LTb\endcsname%
      \put(5876,4500){\makebox(0,0)[r]{\strut{}OpInf error \eqref{eq:Err_TimeActualErr}}}%
      \csname LTb\endcsname%
      \put(5876,4220){\makebox(0,0)[r]{\strut{}intrusive err. est. \eqref{eq:Err_TimeActualBnd}}}%
      \csname LTb\endcsname%
      \put(5876,3940){\makebox(0,0)[r]{\strut{}deterministic err. est. \eqref{eq:Err_TimeDetBnd}}}%
      \csname LTb\endcsname%
      \put(5876,3660){\makebox(0,0)[r]{\strut{}learned err. est. \eqref{eq:Err_TimeProbBnd}}}%
    }%
    \gplbacktext
    \put(0,0){\includegraphics{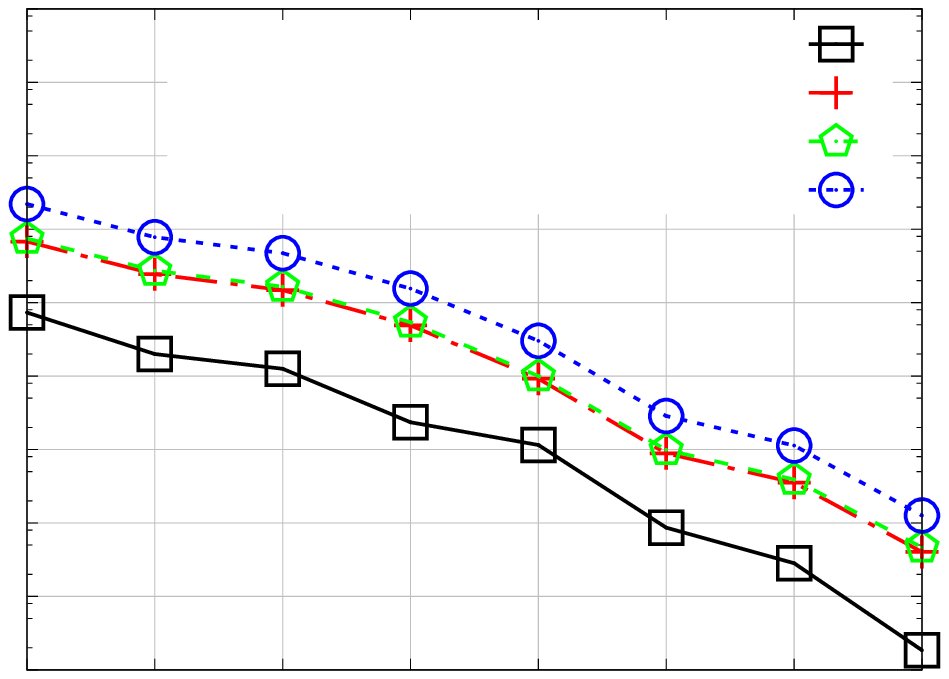}}%
    \gplfronttext
  \end{picture}%
\endgroup

%% file: sheatBwd_ErrAtTimes5.tex
\begingroup
  \makeatletter
  \providecommand\color[2][]{%
    \GenericError{(gnuplot) \space\space\space\@spaces}{%
      Package color not loaded in conjunction with
      terminal option `colourtext'%
    }{See the gnuplot documentation for explanation.%
    }{Either use 'blacktext' in gnuplot or load the package
      color.sty in LaTeX.}%
    \renewcommand\color[2][]{}%
  }%
  \providecommand\includegraphics[2][]{%
    \GenericError{(gnuplot) \space\space\space\@spaces}{%
      Package graphicx or graphics not loaded%
    }{See the gnuplot documentation for explanation.%
    }{The gnuplot epslatex terminal needs graphicx.sty or graphics.sty.}%
    \renewcommand\includegraphics[2][]{}%
  }%
  \providecommand\rotatebox[2]{#2}%
  \@ifundefined{ifGPcolor}{%
    \newif\ifGPcolor
    \GPcolortrue
  }{}%
  \@ifundefined{ifGPblacktext}{%
    \newif\ifGPblacktext
    \GPblacktexttrue
  }{}%
  \let\gplgaddtomacro\g@addto@macro
  \gdef\gplbacktext{}%
  \gdef\gplfronttext{}%
  \makeatother
  \ifGPblacktext
    \def\colorrgb#1{}%
    \def\colorgray#1{}%
  \else
    \ifGPcolor
      \def\colorrgb#1{\color[rgb]{#1}}%
      \def\colorgray#1{\color[gray]{#1}}%
      \expandafter\def\csname LTw\endcsname{\color{white}}%
      \expandafter\def\csname LTb\endcsname{\color{black}}%
      \expandafter\def\csname LTa\endcsname{\color{black}}%
      \expandafter\def\csname LT0\endcsname{\color[rgb]{1,0,0}}%
      \expandafter\def\csname LT1\endcsname{\color[rgb]{0,1,0}}%
      \expandafter\def\csname LT2\endcsname{\color[rgb]{0,0,1}}%
      \expandafter\def\csname LT3\endcsname{\color[rgb]{1,0,1}}%
      \expandafter\def\csname LT4\endcsname{\color[rgb]{0,1,1}}%
      \expandafter\def\csname LT5\endcsname{\color[rgb]{1,1,0}}%
      \expandafter\def\csname LT6\endcsname{\color[rgb]{0,0,0}}%
      \expandafter\def\csname LT7\endcsname{\color[rgb]{1,0.3,0}}%
      \expandafter\def\csname LT8\endcsname{\color[rgb]{0.5,0.5,0.5}}%
    \else
      \def\colorrgb#1{\color{black}}%
      \def\colorgray#1{\color[gray]{#1}}%
      \expandafter\def\csname LTw\endcsname{\color{white}}%
      \expandafter\def\csname LTb\endcsname{\color{black}}%
      \expandafter\def\csname LTa\endcsname{\color{black}}%
      \expandafter\def\csname LT0\endcsname{\color{black}}%
      \expandafter\def\csname LT1\endcsname{\color{black}}%
      \expandafter\def\csname LT2\endcsname{\color{black}}%
      \expandafter\def\csname LT3\endcsname{\color{black}}%
      \expandafter\def\csname LT4\endcsname{\color{black}}%
      \expandafter\def\csname LT5\endcsname{\color{black}}%
      \expandafter\def\csname LT6\endcsname{\color{black}}%
      \expandafter\def\csname LT7\endcsname{\color{black}}%
      \expandafter\def\csname LT8\endcsname{\color{black}}%
    \fi
  \fi
    \setlength{\unitlength}{0.0500bp}%
    \ifx\gptboxheight\undefined%
      \newlength{\gptboxheight}%
      \newlength{\gptboxwidth}%
      \newsavebox{\gptboxtext}%
    \fi%
    \setlength{\fboxrule}{0.5pt}%
    \setlength{\fboxsep}{1pt}%
\begin{picture}(7200.00,5040.00)%
    \gplgaddtomacro\gplbacktext{%
      \csname LTb\endcsname%
      \put(1372,896){\makebox(0,0)[r]{\strut{}1e-05}}%
      \csname LTb\endcsname%
      \put(1372,1319){\makebox(0,0)[r]{\strut{}1e-04}}%
      \csname LTb\endcsname%
      \put(1372,1742){\makebox(0,0)[r]{\strut{}1e-03}}%
      \csname LTb\endcsname%
      \put(1372,2165){\makebox(0,0)[r]{\strut{}1e-02}}%
      \csname LTb\endcsname%
      \put(1372,2588){\makebox(0,0)[r]{\strut{}1e-01}}%
      \csname LTb\endcsname%
      \put(1372,3011){\makebox(0,0)[r]{\strut{}1e+00}}%
      \csname LTb\endcsname%
      \put(1372,3434){\makebox(0,0)[r]{\strut{}1e+01}}%
      \csname LTb\endcsname%
      \put(1372,3857){\makebox(0,0)[r]{\strut{}1e+02}}%
      \csname LTb\endcsname%
      \put(1372,4280){\makebox(0,0)[r]{\strut{}1e+03}}%
      \csname LTb\endcsname%
      \put(1372,4703){\makebox(0,0)[r]{\strut{}1e+04}}%
      \csname LTb\endcsname%
      \put(1540,616){\makebox(0,0){\strut{}$1$}}%
      \csname LTb\endcsname%
      \put(2276,616){\makebox(0,0){\strut{}$2$}}%
      \csname LTb\endcsname%
      \put(3013,616){\makebox(0,0){\strut{}$3$}}%
      \csname LTb\endcsname%
      \put(3749,616){\makebox(0,0){\strut{}$4$}}%
      \csname LTb\endcsname%
      \put(4486,616){\makebox(0,0){\strut{}$5$}}%
      \csname LTb\endcsname%
      \put(5222,616){\makebox(0,0){\strut{}$6$}}%
      \csname LTb\endcsname%
      \put(5959,616){\makebox(0,0){\strut{}$7$}}%
      \csname LTb\endcsname%
      \put(6695,616){\makebox(0,0){\strut{}$8$}}%
    }%
    \gplgaddtomacro\gplfronttext{%
      \csname LTb\endcsname%
      \put(224,2799){\rotatebox{-270}{\makebox(0,0){\strut{}state err. and err. bounds at $t=5$}}}%
      \put(4117,196){\makebox(0,0){\strut{}basis dimension}}%
      \csname LTb\endcsname%
      \put(5876,4500){\makebox(0,0)[r]{\strut{}OpInf error \eqref{eq:Err_TimeActualErr}}}%
      \csname LTb\endcsname%
      \put(5876,4220){\makebox(0,0)[r]{\strut{}intrusive err. est. \eqref{eq:Err_TimeActualBnd}}}%
      \csname LTb\endcsname%
      \put(5876,3940){\makebox(0,0)[r]{\strut{}deterministic err. est. \eqref{eq:Err_TimeDetBnd}}}%
      \csname LTb\endcsname%
      \put(5876,3660){\makebox(0,0)[r]{\strut{}learned err. est. \eqref{eq:Err_TimeProbBnd}}}%
    }%
    \gplbacktext
    \put(0,0){\includegraphics{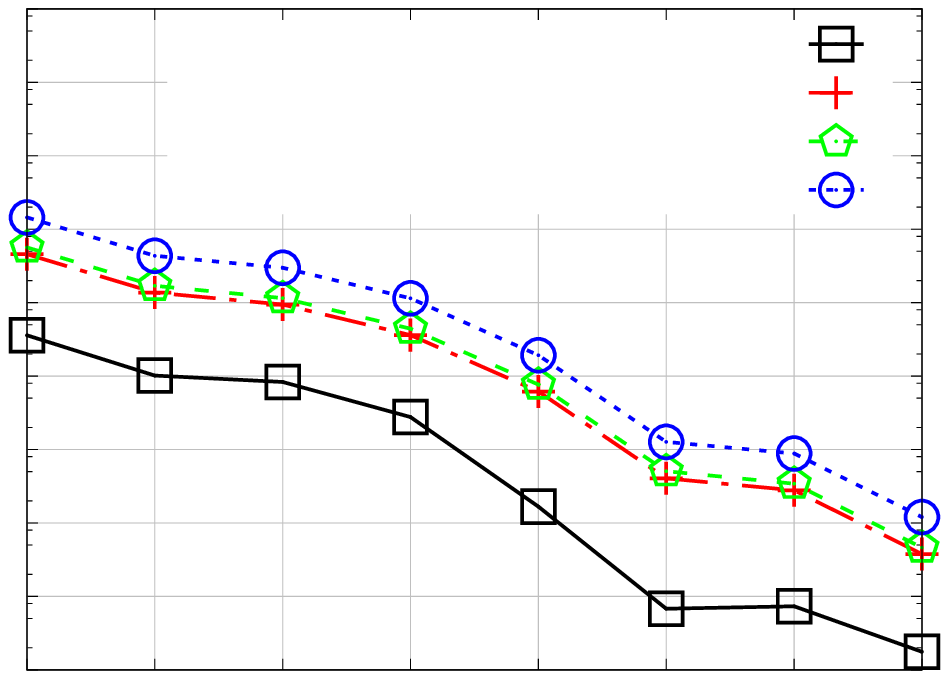}}%
    \gplfronttext
  \end{picture}%
\endgroup

%% file: sheatBwd_RelAveErr.tex
\begingroup
  \makeatletter
  \providecommand\color[2][]{%
    \GenericError{(gnuplot) \space\space\space\@spaces}{%
      Package color not loaded in conjunction with
      terminal option `colourtext'%
    }{See the gnuplot documentation for explanation.%
    }{Either use 'blacktext' in gnuplot or load the package
      color.sty in LaTeX.}%
    \renewcommand\color[2][]{}%
  }%
  \providecommand\includegraphics[2][]{%
    \GenericError{(gnuplot) \space\space\space\@spaces}{%
      Package graphicx or graphics not loaded%
    }{See the gnuplot documentation for explanation.%
    }{The gnuplot epslatex terminal needs graphicx.sty or graphics.sty.}%
    \renewcommand\includegraphics[2][]{}%
  }%
  \providecommand\rotatebox[2]{#2}%
  \@ifundefined{ifGPcolor}{%
    \newif\ifGPcolor
    \GPcolortrue
  }{}%
  \@ifundefined{ifGPblacktext}{%
    \newif\ifGPblacktext
    \GPblacktexttrue
  }{}%
  \let\gplgaddtomacro\g@addto@macro
  \gdef\gplbacktext{}%
  \gdef\gplfronttext{}%
  \makeatother
  \ifGPblacktext
    \def\colorrgb#1{}%
    \def\colorgray#1{}%
  \else
    \ifGPcolor
      \def\colorrgb#1{\color[rgb]{#1}}%
      \def\colorgray#1{\color[gray]{#1}}%
      \expandafter\def\csname LTw\endcsname{\color{white}}%
      \expandafter\def\csname LTb\endcsname{\color{black}}%
      \expandafter\def\csname LTa\endcsname{\color{black}}%
      \expandafter\def\csname LT0\endcsname{\color[rgb]{1,0,0}}%
      \expandafter\def\csname LT1\endcsname{\color[rgb]{0,1,0}}%
      \expandafter\def\csname LT2\endcsname{\color[rgb]{0,0,1}}%
      \expandafter\def\csname LT3\endcsname{\color[rgb]{1,0,1}}%
      \expandafter\def\csname LT4\endcsname{\color[rgb]{0,1,1}}%
      \expandafter\def\csname LT5\endcsname{\color[rgb]{1,1,0}}%
      \expandafter\def\csname LT6\endcsname{\color[rgb]{0,0,0}}%
      \expandafter\def\csname LT7\endcsname{\color[rgb]{1,0.3,0}}%
      \expandafter\def\csname LT8\endcsname{\color[rgb]{0.5,0.5,0.5}}%
    \else
      \def\colorrgb#1{\color{black}}%
      \def\colorgray#1{\color[gray]{#1}}%
      \expandafter\def\csname LTw\endcsname{\color{white}}%
      \expandafter\def\csname LTb\endcsname{\color{black}}%
      \expandafter\def\csname LTa\endcsname{\color{black}}%
      \expandafter\def\csname LT0\endcsname{\color{black}}%
      \expandafter\def\csname LT1\endcsname{\color{black}}%
      \expandafter\def\csname LT2\endcsname{\color{black}}%
      \expandafter\def\csname LT3\endcsname{\color{black}}%
      \expandafter\def\csname LT4\endcsname{\color{black}}%
      \expandafter\def\csname LT5\endcsname{\color{black}}%
      \expandafter\def\csname LT6\endcsname{\color{black}}%
      \expandafter\def\csname LT7\endcsname{\color{black}}%
      \expandafter\def\csname LT8\endcsname{\color{black}}%
    \fi
  \fi
    \setlength{\unitlength}{0.0500bp}%
    \ifx\gptboxheight\undefined%
      \newlength{\gptboxheight}%
      \newlength{\gptboxwidth}%
      \newsavebox{\gptboxtext}%
    \fi%
    \setlength{\fboxrule}{0.5pt}%
    \setlength{\fboxsep}{1pt}%
\begin{picture}(7200.00,5040.00)%
    \gplgaddtomacro\gplbacktext{%
      \csname LTb\endcsname%
      \put(1372,896){\makebox(0,0)[r]{\strut{}1e-09}}%
      \csname LTb\endcsname%
      \put(1372,1372){\makebox(0,0)[r]{\strut{}1e-08}}%
      \csname LTb\endcsname%
      \put(1372,1848){\makebox(0,0)[r]{\strut{}1e-07}}%
      \csname LTb\endcsname%
      \put(1372,2324){\makebox(0,0)[r]{\strut{}1e-06}}%
      \csname LTb\endcsname%
      \put(1372,2799){\makebox(0,0)[r]{\strut{}1e-05}}%
      \csname LTb\endcsname%
      \put(1372,3275){\makebox(0,0)[r]{\strut{}1e-04}}%
      \csname LTb\endcsname%
      \put(1372,3751){\makebox(0,0)[r]{\strut{}1e-03}}%
      \csname LTb\endcsname%
      \put(1372,4227){\makebox(0,0)[r]{\strut{}1e-02}}%
      \csname LTb\endcsname%
      \put(1372,4703){\makebox(0,0)[r]{\strut{}1e-01}}%
      \csname LTb\endcsname%
      \put(1540,616){\makebox(0,0){\strut{}$1$}}%
      \csname LTb\endcsname%
      \put(2276,616){\makebox(0,0){\strut{}$2$}}%
      \csname LTb\endcsname%
      \put(3013,616){\makebox(0,0){\strut{}$3$}}%
      \csname LTb\endcsname%
      \put(3749,616){\makebox(0,0){\strut{}$4$}}%
      \csname LTb\endcsname%
      \put(4486,616){\makebox(0,0){\strut{}$5$}}%
      \csname LTb\endcsname%
      \put(5222,616){\makebox(0,0){\strut{}$6$}}%
      \csname LTb\endcsname%
      \put(5959,616){\makebox(0,0){\strut{}$7$}}%
      \csname LTb\endcsname%
      \put(6695,616){\makebox(0,0){\strut{}$8$}}%
    }%
    \gplgaddtomacro\gplfronttext{%
      \csname LTb\endcsname%
      \put(224,2799){\rotatebox{-270}{\makebox(0,0){\strut{}rel. ave. state err. over time}}}%
      \put(4117,196){\makebox(0,0){\strut{}basis dimension}}%
      \csname LTb\endcsname%
      \put(2359,1939){\makebox(0,0)[l]{\strut{}OpInf error \eqref{eq:Err_AveRelErrActualErr}}}%
      \csname LTb\endcsname%
      \put(2359,1659){\makebox(0,0)[l]{\strut{}intrusive err. est. \eqref{eq:Err_AveRelErrActualBnd}}}%
      \csname LTb\endcsname%
      \put(2359,1379){\makebox(0,0)[l]{\strut{}deterministic err. est. \eqref{eq:Err_AveRelErrDetBnd}}}%
      \csname LTb\endcsname%
      \put(2359,1099){\makebox(0,0)[l]{\strut{}learned err. est. \eqref{eq:Err_AveRelErrProbBnd}}}%
    }%
    \gplbacktext
    \put(0,0){\includegraphics{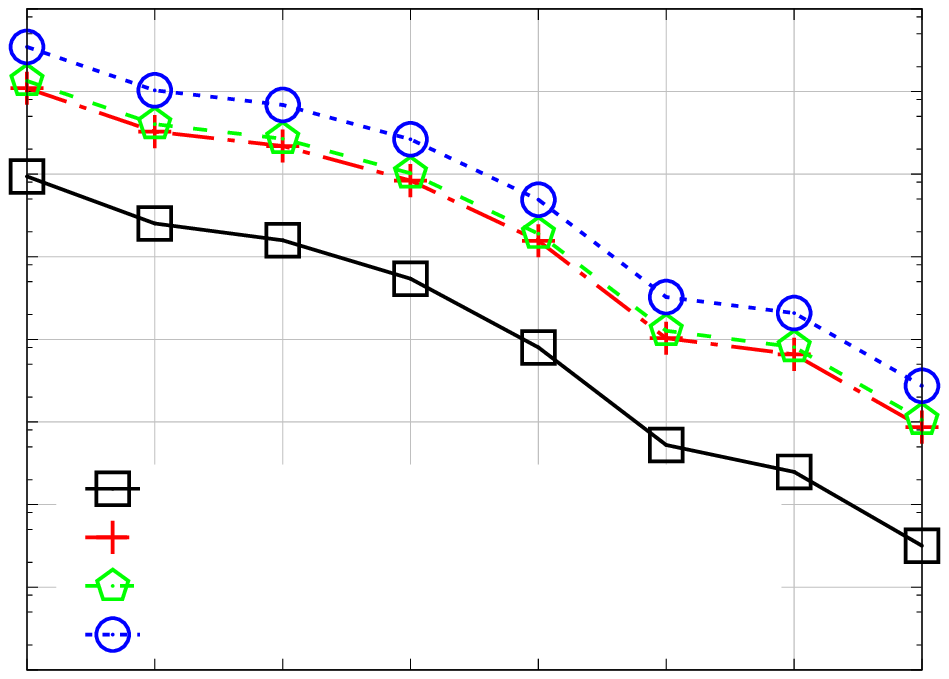}}%
    \gplfronttext
  \end{picture}%
\endgroup

%% file: sheatBwd_MCSamples01.tex
\begingroup
  \makeatletter
  \providecommand\color[2][]{%
    \GenericError{(gnuplot) \space\space\space\@spaces}{%
      Package color not loaded in conjunction with
      terminal option `colourtext'%
    }{See the gnuplot documentation for explanation.%
    }{Either use 'blacktext' in gnuplot or load the package
      color.sty in LaTeX.}%
    \renewcommand\color[2][]{}%
  }%
  \providecommand\includegraphics[2][]{%
    \GenericError{(gnuplot) \space\space\space\@spaces}{%
      Package graphicx or graphics not loaded%
    }{See the gnuplot documentation for explanation.%
    }{The gnuplot epslatex terminal needs graphicx.sty or graphics.sty.}%
    \renewcommand\includegraphics[2][]{}%
  }%
  \providecommand\rotatebox[2]{#2}%
  \@ifundefined{ifGPcolor}{%
    \newif\ifGPcolor
    \GPcolortrue
  }{}%
  \@ifundefined{ifGPblacktext}{%
    \newif\ifGPblacktext
    \GPblacktexttrue
  }{}%
  \let\gplgaddtomacro\g@addto@macro
  \gdef\gplbacktext{}%
  \gdef\gplfronttext{}%
  \makeatother
  \ifGPblacktext
    \def\colorrgb#1{}%
    \def\colorgray#1{}%
  \else
    \ifGPcolor
      \def\colorrgb#1{\color[rgb]{#1}}%
      \def\colorgray#1{\color[gray]{#1}}%
      \expandafter\def\csname LTw\endcsname{\color{white}}%
      \expandafter\def\csname LTb\endcsname{\color{black}}%
      \expandafter\def\csname LTa\endcsname{\color{black}}%
      \expandafter\def\csname LT0\endcsname{\color[rgb]{1,0,0}}%
      \expandafter\def\csname LT1\endcsname{\color[rgb]{0,1,0}}%
      \expandafter\def\csname LT2\endcsname{\color[rgb]{0,0,1}}%
      \expandafter\def\csname LT3\endcsname{\color[rgb]{1,0,1}}%
      \expandafter\def\csname LT4\endcsname{\color[rgb]{0,1,1}}%
      \expandafter\def\csname LT5\endcsname{\color[rgb]{1,1,0}}%
      \expandafter\def\csname LT6\endcsname{\color[rgb]{0,0,0}}%
      \expandafter\def\csname LT7\endcsname{\color[rgb]{1,0.3,0}}%
      \expandafter\def\csname LT8\endcsname{\color[rgb]{0.5,0.5,0.5}}%
    \else
      \def\colorrgb#1{\color{black}}%
      \def\colorgray#1{\color[gray]{#1}}%
      \expandafter\def\csname LTw\endcsname{\color{white}}%
      \expandafter\def\csname LTb\endcsname{\color{black}}%
      \expandafter\def\csname LTa\endcsname{\color{black}}%
      \expandafter\def\csname LT0\endcsname{\color{black}}%
      \expandafter\def\csname LT1\endcsname{\color{black}}%
      \expandafter\def\csname LT2\endcsname{\color{black}}%
      \expandafter\def\csname LT3\endcsname{\color{black}}%
      \expandafter\def\csname LT4\endcsname{\color{black}}%
      \expandafter\def\csname LT5\endcsname{\color{black}}%
      \expandafter\def\csname LT6\endcsname{\color{black}}%
      \expandafter\def\csname LT7\endcsname{\color{black}}%
      \expandafter\def\csname LT8\endcsname{\color{black}}%
    \fi
  \fi
    \setlength{\unitlength}{0.0500bp}%
    \ifx\gptboxheight\undefined%
      \newlength{\gptboxheight}%
      \newlength{\gptboxwidth}%
      \newsavebox{\gptboxtext}%
    \fi%
    \setlength{\fboxrule}{0.5pt}%
    \setlength{\fboxsep}{1pt}%
\begin{picture}(7200.00,5040.00)%
    \gplgaddtomacro\gplbacktext{%
      \csname LTb\endcsname%
      \put(1372,896){\makebox(0,0)[r]{\strut{}1e-04}}%
      \csname LTb\endcsname%
      \put(1372,1531){\makebox(0,0)[r]{\strut{}1e-03}}%
      \csname LTb\endcsname%
      \put(1372,2165){\makebox(0,0)[r]{\strut{}1e-02}}%
      \csname LTb\endcsname%
      \put(1372,2800){\makebox(0,0)[r]{\strut{}1e-01}}%
      \csname LTb\endcsname%
      \put(1372,3434){\makebox(0,0)[r]{\strut{}1e+00}}%
      \csname LTb\endcsname%
      \put(1372,4069){\makebox(0,0)[r]{\strut{}1e+01}}%
      \csname LTb\endcsname%
      \put(1372,4703){\makebox(0,0)[r]{\strut{}1e+02}}%
      \csname LTb\endcsname%
      \put(1540,616){\makebox(0,0){\strut{}$1$}}%
      \csname LTb\endcsname%
      \put(2276,616){\makebox(0,0){\strut{}$2$}}%
      \csname LTb\endcsname%
      \put(3013,616){\makebox(0,0){\strut{}$3$}}%
      \csname LTb\endcsname%
      \put(3749,616){\makebox(0,0){\strut{}$4$}}%
      \csname LTb\endcsname%
      \put(4486,616){\makebox(0,0){\strut{}$5$}}%
      \csname LTb\endcsname%
      \put(5222,616){\makebox(0,0){\strut{}$6$}}%
      \csname LTb\endcsname%
      \put(5959,616){\makebox(0,0){\strut{}$7$}}%
      \csname LTb\endcsname%
      \put(6695,616){\makebox(0,0){\strut{}$8$}}%
    }%
    \gplgaddtomacro\gplfronttext{%
      \csname LTb\endcsname%
      \put(224,2799){\rotatebox{-270}{\makebox(0,0){\strut{}state err. and err. bounds at $t=1$}}}%
      \put(4117,196){\makebox(0,0){\strut{}basis dimension}}%
      \csname LTb\endcsname%
      \put(2359,1589){\makebox(0,0)[l]{\strut{}learned err. est. \eqref{eq:Err_TimeProbBnd}}}%
      \csname LTb\endcsname%
      \put(2359,1169){\makebox(0,0)[l]{\strut{}intrusive err. est. \eqref{eq:Err_TimeActualBnd}}}%
    }%
    \gplbacktext
    \put(0,0){\includegraphics{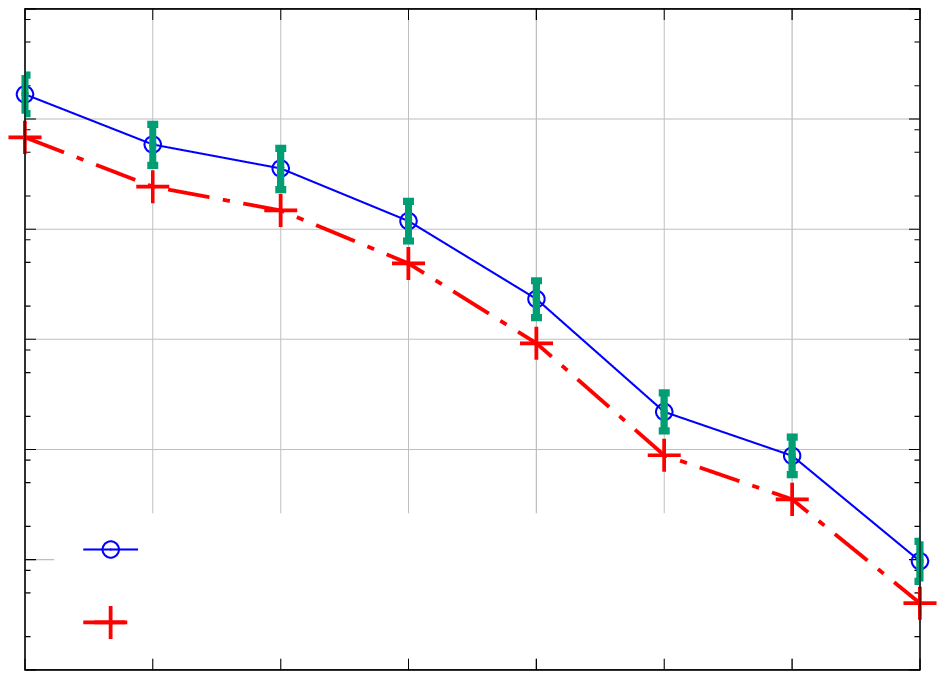}}%
    \gplfronttext
  \end{picture}%
\endgroup

%% file: sheatBwd_MCSamples05.tex
\begingroup
  \makeatletter
  \providecommand\color[2][]{%
    \GenericError{(gnuplot) \space\space\space\@spaces}{%
      Package color not loaded in conjunction with
      terminal option `colourtext'%
    }{See the gnuplot documentation for explanation.%
    }{Either use 'blacktext' in gnuplot or load the package
      color.sty in LaTeX.}%
    \renewcommand\color[2][]{}%
  }%
  \providecommand\includegraphics[2][]{%
    \GenericError{(gnuplot) \space\space\space\@spaces}{%
      Package graphicx or graphics not loaded%
    }{See the gnuplot documentation for explanation.%
    }{The gnuplot epslatex terminal needs graphicx.sty or graphics.sty.}%
    \renewcommand\includegraphics[2][]{}%
  }%
  \providecommand\rotatebox[2]{#2}%
  \@ifundefined{ifGPcolor}{%
    \newif\ifGPcolor
    \GPcolortrue
  }{}%
  \@ifundefined{ifGPblacktext}{%
    \newif\ifGPblacktext
    \GPblacktexttrue
  }{}%
  \let\gplgaddtomacro\g@addto@macro
  \gdef\gplbacktext{}%
  \gdef\gplfronttext{}%
  \makeatother
  \ifGPblacktext
    \def\colorrgb#1{}%
    \def\colorgray#1{}%
  \else
    \ifGPcolor
      \def\colorrgb#1{\color[rgb]{#1}}%
      \def\colorgray#1{\color[gray]{#1}}%
      \expandafter\def\csname LTw\endcsname{\color{white}}%
      \expandafter\def\csname LTb\endcsname{\color{black}}%
      \expandafter\def\csname LTa\endcsname{\color{black}}%
      \expandafter\def\csname LT0\endcsname{\color[rgb]{1,0,0}}%
      \expandafter\def\csname LT1\endcsname{\color[rgb]{0,1,0}}%
      \expandafter\def\csname LT2\endcsname{\color[rgb]{0,0,1}}%
      \expandafter\def\csname LT3\endcsname{\color[rgb]{1,0,1}}%
      \expandafter\def\csname LT4\endcsname{\color[rgb]{0,1,1}}%
      \expandafter\def\csname LT5\endcsname{\color[rgb]{1,1,0}}%
      \expandafter\def\csname LT6\endcsname{\color[rgb]{0,0,0}}%
      \expandafter\def\csname LT7\endcsname{\color[rgb]{1,0.3,0}}%
      \expandafter\def\csname LT8\endcsname{\color[rgb]{0.5,0.5,0.5}}%
    \else
      \def\colorrgb#1{\color{black}}%
      \def\colorgray#1{\color[gray]{#1}}%
      \expandafter\def\csname LTw\endcsname{\color{white}}%
      \expandafter\def\csname LTb\endcsname{\color{black}}%
      \expandafter\def\csname LTa\endcsname{\color{black}}%
      \expandafter\def\csname LT0\endcsname{\color{black}}%
      \expandafter\def\csname LT1\endcsname{\color{black}}%
      \expandafter\def\csname LT2\endcsname{\color{black}}%
      \expandafter\def\csname LT3\endcsname{\color{black}}%
      \expandafter\def\csname LT4\endcsname{\color{black}}%
      \expandafter\def\csname LT5\endcsname{\color{black}}%
      \expandafter\def\csname LT6\endcsname{\color{black}}%
      \expandafter\def\csname LT7\endcsname{\color{black}}%
      \expandafter\def\csname LT8\endcsname{\color{black}}%
    \fi
  \fi
    \setlength{\unitlength}{0.0500bp}%
    \ifx\gptboxheight\undefined%
      \newlength{\gptboxheight}%
      \newlength{\gptboxwidth}%
      \newsavebox{\gptboxtext}%
    \fi%
    \setlength{\fboxrule}{0.5pt}%
    \setlength{\fboxsep}{1pt}%
\begin{picture}(7200.00,5040.00)%
    \gplgaddtomacro\gplbacktext{%
      \csname LTb\endcsname%
      \put(1372,896){\makebox(0,0)[r]{\strut{}1e-04}}%
      \csname LTb\endcsname%
      \put(1372,1531){\makebox(0,0)[r]{\strut{}1e-03}}%
      \csname LTb\endcsname%
      \put(1372,2165){\makebox(0,0)[r]{\strut{}1e-02}}%
      \csname LTb\endcsname%
      \put(1372,2800){\makebox(0,0)[r]{\strut{}1e-01}}%
      \csname LTb\endcsname%
      \put(1372,3434){\makebox(0,0)[r]{\strut{}1e+00}}%
      \csname LTb\endcsname%
      \put(1372,4069){\makebox(0,0)[r]{\strut{}1e+01}}%
      \csname LTb\endcsname%
      \put(1372,4703){\makebox(0,0)[r]{\strut{}1e+02}}%
      \csname LTb\endcsname%
      \put(1540,616){\makebox(0,0){\strut{}$1$}}%
      \csname LTb\endcsname%
      \put(2276,616){\makebox(0,0){\strut{}$2$}}%
      \csname LTb\endcsname%
      \put(3013,616){\makebox(0,0){\strut{}$3$}}%
      \csname LTb\endcsname%
      \put(3749,616){\makebox(0,0){\strut{}$4$}}%
      \csname LTb\endcsname%
      \put(4486,616){\makebox(0,0){\strut{}$5$}}%
      \csname LTb\endcsname%
      \put(5222,616){\makebox(0,0){\strut{}$6$}}%
      \csname LTb\endcsname%
      \put(5959,616){\makebox(0,0){\strut{}$7$}}%
      \csname LTb\endcsname%
      \put(6695,616){\makebox(0,0){\strut{}$8$}}%
    }%
    \gplgaddtomacro\gplfronttext{%
      \csname LTb\endcsname%
      \put(224,2799){\rotatebox{-270}{\makebox(0,0){\strut{}state err. and err. bounds at $t=5$}}}%
      \put(4117,196){\makebox(0,0){\strut{}basis dimension}}%
      \csname LTb\endcsname%
      \put(2359,1589){\makebox(0,0)[l]{\strut{}learned err. est. \eqref{eq:Err_TimeProbBnd}}}%
      \csname LTb\endcsname%
      \put(2359,1169){\makebox(0,0)[l]{\strut{}intrusive err. est. \eqref{eq:Err_TimeActualBnd}}}%
    }%
    \gplbacktext
    \put(0,0){\includegraphics{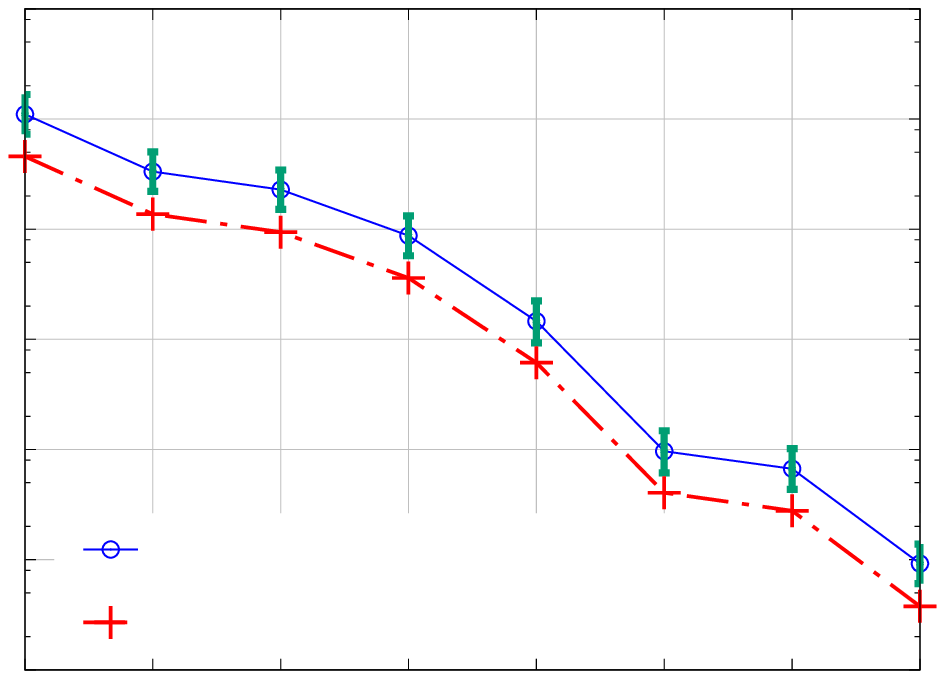}}%
    \gplfronttext
  \end{picture}%
\endgroup

%% file: sheatBwd_MCSamplesRelErr.tex
\begingroup
  \makeatletter
  \providecommand\color[2][]{%
    \GenericError{(gnuplot) \space\space\space\@spaces}{%
      Package color not loaded in conjunction with
      terminal option `colourtext'%
    }{See the gnuplot documentation for explanation.%
    }{Either use 'blacktext' in gnuplot or load the package
      color.sty in LaTeX.}%
    \renewcommand\color[2][]{}%
  }%
  \providecommand\includegraphics[2][]{%
    \GenericError{(gnuplot) \space\space\space\@spaces}{%
      Package graphicx or graphics not loaded%
    }{See the gnuplot documentation for explanation.%
    }{The gnuplot epslatex terminal needs graphicx.sty or graphics.sty.}%
    \renewcommand\includegraphics[2][]{}%
  }%
  \providecommand\rotatebox[2]{#2}%
  \@ifundefined{ifGPcolor}{%
    \newif\ifGPcolor
    \GPcolortrue
  }{}%
  \@ifundefined{ifGPblacktext}{%
    \newif\ifGPblacktext
    \GPblacktexttrue
  }{}%
  \let\gplgaddtomacro\g@addto@macro
  \gdef\gplbacktext{}%
  \gdef\gplfronttext{}%
  \makeatother
  \ifGPblacktext
    \def\colorrgb#1{}%
    \def\colorgray#1{}%
  \else
    \ifGPcolor
      \def\colorrgb#1{\color[rgb]{#1}}%
      \def\colorgray#1{\color[gray]{#1}}%
      \expandafter\def\csname LTw\endcsname{\color{white}}%
      \expandafter\def\csname LTb\endcsname{\color{black}}%
      \expandafter\def\csname LTa\endcsname{\color{black}}%
      \expandafter\def\csname LT0\endcsname{\color[rgb]{1,0,0}}%
      \expandafter\def\csname LT1\endcsname{\color[rgb]{0,1,0}}%
      \expandafter\def\csname LT2\endcsname{\color[rgb]{0,0,1}}%
      \expandafter\def\csname LT3\endcsname{\color[rgb]{1,0,1}}%
      \expandafter\def\csname LT4\endcsname{\color[rgb]{0,1,1}}%
      \expandafter\def\csname LT5\endcsname{\color[rgb]{1,1,0}}%
      \expandafter\def\csname LT6\endcsname{\color[rgb]{0,0,0}}%
      \expandafter\def\csname LT7\endcsname{\color[rgb]{1,0.3,0}}%
      \expandafter\def\csname LT8\endcsname{\color[rgb]{0.5,0.5,0.5}}%
    \else
      \def\colorrgb#1{\color{black}}%
      \def\colorgray#1{\color[gray]{#1}}%
      \expandafter\def\csname LTw\endcsname{\color{white}}%
      \expandafter\def\csname LTb\endcsname{\color{black}}%
      \expandafter\def\csname LTa\endcsname{\color{black}}%
      \expandafter\def\csname LT0\endcsname{\color{black}}%
      \expandafter\def\csname LT1\endcsname{\color{black}}%
      \expandafter\def\csname LT2\endcsname{\color{black}}%
      \expandafter\def\csname LT3\endcsname{\color{black}}%
      \expandafter\def\csname LT4\endcsname{\color{black}}%
      \expandafter\def\csname LT5\endcsname{\color{black}}%
      \expandafter\def\csname LT6\endcsname{\color{black}}%
      \expandafter\def\csname LT7\endcsname{\color{black}}%
      \expandafter\def\csname LT8\endcsname{\color{black}}%
    \fi
  \fi
    \setlength{\unitlength}{0.0500bp}%
    \ifx\gptboxheight\undefined%
      \newlength{\gptboxheight}%
      \newlength{\gptboxwidth}%
      \newsavebox{\gptboxtext}%
    \fi%
    \setlength{\fboxrule}{0.5pt}%
    \setlength{\fboxsep}{1pt}%
\begin{picture}(7200.00,5040.00)%
    \gplgaddtomacro\gplbacktext{%
      \csname LTb\endcsname%
      \put(1372,896){\makebox(0,0)[r]{\strut{}1e-07}}%
      \csname LTb\endcsname%
      \put(1372,1531){\makebox(0,0)[r]{\strut{}1e-06}}%
      \csname LTb\endcsname%
      \put(1372,2165){\makebox(0,0)[r]{\strut{}1e-05}}%
      \csname LTb\endcsname%
      \put(1372,2800){\makebox(0,0)[r]{\strut{}1e-04}}%
      \csname LTb\endcsname%
      \put(1372,3434){\makebox(0,0)[r]{\strut{}1e-03}}%
      \csname LTb\endcsname%
      \put(1372,4069){\makebox(0,0)[r]{\strut{}1e-02}}%
      \csname LTb\endcsname%
      \put(1372,4703){\makebox(0,0)[r]{\strut{}1e-01}}%
      \csname LTb\endcsname%
      \put(1540,616){\makebox(0,0){\strut{}$1$}}%
      \csname LTb\endcsname%
      \put(2276,616){\makebox(0,0){\strut{}$2$}}%
      \csname LTb\endcsname%
      \put(3013,616){\makebox(0,0){\strut{}$3$}}%
      \csname LTb\endcsname%
      \put(3749,616){\makebox(0,0){\strut{}$4$}}%
      \csname LTb\endcsname%
      \put(4486,616){\makebox(0,0){\strut{}$5$}}%
      \csname LTb\endcsname%
      \put(5222,616){\makebox(0,0){\strut{}$6$}}%
      \csname LTb\endcsname%
      \put(5959,616){\makebox(0,0){\strut{}$7$}}%
      \csname LTb\endcsname%
      \put(6695,616){\makebox(0,0){\strut{}$8$}}%
    }%
    \gplgaddtomacro\gplfronttext{%
      \csname LTb\endcsname%
      \put(224,2799){\rotatebox{-270}{\makebox(0,0){\strut{}rel. ave. state err. over time}}}%
      \put(4117,196){\makebox(0,0){\strut{}basis dimension}}%
      \csname LTb\endcsname%
      \put(2359,1589){\makebox(0,0)[l]{\strut{}learned err. est. \eqref{eq:Err_AveRelErrProbBnd}}}%
      \csname LTb\endcsname%
      \put(2359,1169){\makebox(0,0)[l]{\strut{}intrusive err. est. \eqref{eq:Err_AveRelErrActualBnd}}}%
    }%
    \gplbacktext
    \put(0,0){\includegraphics{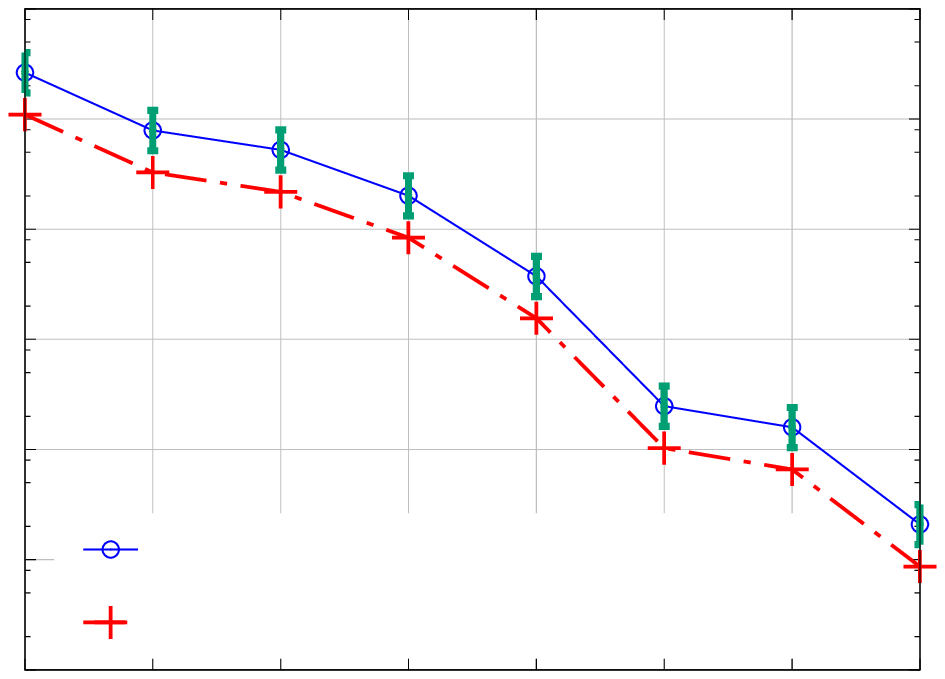}}%
    \gplfronttext
  \end{picture}%
\endgroup

%% file: convDiffAnormEasySurf_t01.tex
\begingroup
  \makeatletter
  \providecommand\color[2][]{%
    \GenericError{(gnuplot) \space\space\space\@spaces}{%
      Package color not loaded in conjunction with
      terminal option `colourtext'%
    }{See the gnuplot documentation for explanation.%
    }{Either use 'blacktext' in gnuplot or load the package
      color.sty in LaTeX.}%
    \renewcommand\color[2][]{}%
  }%
  \providecommand\includegraphics[2][]{%
    \GenericError{(gnuplot) \space\space\space\@spaces}{%
      Package graphicx or graphics not loaded%
    }{See the gnuplot documentation for explanation.%
    }{The gnuplot epslatex terminal needs graphicx.sty or graphics.sty.}%
    \renewcommand\includegraphics[2][]{}%
  }%
  \providecommand\rotatebox[2]{#2}%
  \@ifundefined{ifGPcolor}{%
    \newif\ifGPcolor
    \GPcolortrue
  }{}%
  \@ifundefined{ifGPblacktext}{%
    \newif\ifGPblacktext
    \GPblacktexttrue
  }{}%
  \let\gplgaddtomacro\g@addto@macro
  \gdef\gplbacktext{}%
  \gdef\gplfronttext{}%
  \makeatother
  \ifGPblacktext
    \def\colorrgb#1{}%
    \def\colorgray#1{}%
  \else
    \ifGPcolor
      \def\colorrgb#1{\color[rgb]{#1}}%
      \def\colorgray#1{\color[gray]{#1}}%
      \expandafter\def\csname LTw\endcsname{\color{white}}%
      \expandafter\def\csname LTb\endcsname{\color{black}}%
      \expandafter\def\csname LTa\endcsname{\color{black}}%
      \expandafter\def\csname LT0\endcsname{\color[rgb]{1,0,0}}%
      \expandafter\def\csname LT1\endcsname{\color[rgb]{0,1,0}}%
      \expandafter\def\csname LT2\endcsname{\color[rgb]{0,0,1}}%
      \expandafter\def\csname LT3\endcsname{\color[rgb]{1,0,1}}%
      \expandafter\def\csname LT4\endcsname{\color[rgb]{0,1,1}}%
      \expandafter\def\csname LT5\endcsname{\color[rgb]{1,1,0}}%
      \expandafter\def\csname LT6\endcsname{\color[rgb]{0,0,0}}%
      \expandafter\def\csname LT7\endcsname{\color[rgb]{1,0.3,0}}%
      \expandafter\def\csname LT8\endcsname{\color[rgb]{0.5,0.5,0.5}}%
    \else
      \def\colorrgb#1{\color{black}}%
      \def\colorgray#1{\color[gray]{#1}}%
      \expandafter\def\csname LTw\endcsname{\color{white}}%
      \expandafter\def\csname LTb\endcsname{\color{black}}%
      \expandafter\def\csname LTa\endcsname{\color{black}}%
      \expandafter\def\csname LT0\endcsname{\color{black}}%
      \expandafter\def\csname LT1\endcsname{\color{black}}%
      \expandafter\def\csname LT2\endcsname{\color{black}}%
      \expandafter\def\csname LT3\endcsname{\color{black}}%
      \expandafter\def\csname LT4\endcsname{\color{black}}%
      \expandafter\def\csname LT5\endcsname{\color{black}}%
      \expandafter\def\csname LT6\endcsname{\color{black}}%
      \expandafter\def\csname LT7\endcsname{\color{black}}%
      \expandafter\def\csname LT8\endcsname{\color{black}}%
    \fi
  \fi
    \setlength{\unitlength}{0.0500bp}%
    \ifx\gptboxheight\undefined%
      \newlength{\gptboxheight}%
      \newlength{\gptboxwidth}%
      \newsavebox{\gptboxtext}%
    \fi%
    \setlength{\fboxrule}{0.5pt}%
    \setlength{\fboxsep}{1pt}%
\begin{picture}(7200.00,5040.00)%
    \gplgaddtomacro\gplbacktext{%
      \csname LTb\endcsname%
      \put(3236,806){\makebox(0,0){\strut{}$0$}}%
      \put(3869,918){\makebox(0,0){\strut{}$0.2$}}%
      \put(4503,1030){\makebox(0,0){\strut{}$0.4$}}%
      \put(5136,1142){\makebox(0,0){\strut{}$0.6$}}%
      \put(5770,1253){\makebox(0,0){\strut{}$0.8$}}%
      \put(6403,1365){\makebox(0,0){\strut{}$1$}}%
      \put(2451,884){\makebox(0,0){\strut{}$0$}}%
      \put(1993,1127){\makebox(0,0){\strut{}$0.05$}}%
      \put(1536,1369){\makebox(0,0){\strut{}$0.1$}}%
      \put(1079,1611){\makebox(0,0){\strut{}$0.15$}}%
      \put(622,1853){\makebox(0,0){\strut{}$0.2$}}%
      \put(976,2542){\makebox(0,0)[r]{\strut{}0.00}}%
      \put(976,2864){\makebox(0,0)[r]{\strut{}0.02}}%
      \put(976,3187){\makebox(0,0)[r]{\strut{}0.04}}%
      \put(976,3510){\makebox(0,0)[r]{\strut{}0.06}}%
      \put(976,3833){\makebox(0,0)[r]{\strut{}0.08}}%
      \put(-241,3187){\rotatebox{90}{\makebox(0,0){\strut{}$w(x_1,x_2,0.1)$}}}%
    }%
    \gplgaddtomacro\gplfronttext{%
      \csname LTb\endcsname%
      \put(5307,686){\makebox(0,0){\strut{}$x_1$}}%
      \put(890,993){\makebox(0,0){\strut{}$x_2$}}%
      \put(-241,3187){\rotatebox{90}{\makebox(0,0){\strut{}$w(x_1,x_2,0.1)$}}}%
      \put(6613,2331){\makebox(0,0)[l]{\strut{}0.00}}%
      \put(6613,2691){\makebox(0,0)[l]{\strut{}0.02}}%
      \put(6613,3051){\makebox(0,0)[l]{\strut{}0.04}}%
      \put(6613,3411){\makebox(0,0)[l]{\strut{}0.06}}%
      \put(6613,3772){\makebox(0,0)[l]{\strut{}0.08}}%
    }%
    \gplbacktext
    \put(0,0){\includegraphics{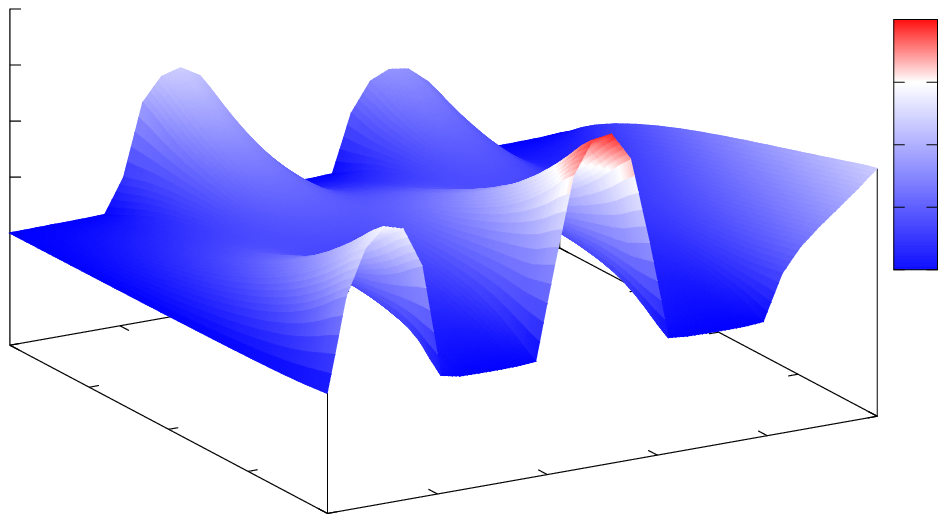}}%
    \gplfronttext
  \end{picture}%
\endgroup

%% file: convDiffAnormEasySurf_t05.tex
\begingroup
  \makeatletter
  \providecommand\color[2][]{%
    \GenericError{(gnuplot) \space\space\space\@spaces}{%
      Package color not loaded in conjunction with
      terminal option `colourtext'%
    }{See the gnuplot documentation for explanation.%
    }{Either use 'blacktext' in gnuplot or load the package
      color.sty in LaTeX.}%
    \renewcommand\color[2][]{}%
  }%
  \providecommand\includegraphics[2][]{%
    \GenericError{(gnuplot) \space\space\space\@spaces}{%
      Package graphicx or graphics not loaded%
    }{See the gnuplot documentation for explanation.%
    }{The gnuplot epslatex terminal needs graphicx.sty or graphics.sty.}%
    \renewcommand\includegraphics[2][]{}%
  }%
  \providecommand\rotatebox[2]{#2}%
  \@ifundefined{ifGPcolor}{%
    \newif\ifGPcolor
    \GPcolortrue
  }{}%
  \@ifundefined{ifGPblacktext}{%
    \newif\ifGPblacktext
    \GPblacktexttrue
  }{}%
  \let\gplgaddtomacro\g@addto@macro
  \gdef\gplbacktext{}%
  \gdef\gplfronttext{}%
  \makeatother
  \ifGPblacktext
    \def\colorrgb#1{}%
    \def\colorgray#1{}%
  \else
    \ifGPcolor
      \def\colorrgb#1{\color[rgb]{#1}}%
      \def\colorgray#1{\color[gray]{#1}}%
      \expandafter\def\csname LTw\endcsname{\color{white}}%
      \expandafter\def\csname LTb\endcsname{\color{black}}%
      \expandafter\def\csname LTa\endcsname{\color{black}}%
      \expandafter\def\csname LT0\endcsname{\color[rgb]{1,0,0}}%
      \expandafter\def\csname LT1\endcsname{\color[rgb]{0,1,0}}%
      \expandafter\def\csname LT2\endcsname{\color[rgb]{0,0,1}}%
      \expandafter\def\csname LT3\endcsname{\color[rgb]{1,0,1}}%
      \expandafter\def\csname LT4\endcsname{\color[rgb]{0,1,1}}%
      \expandafter\def\csname LT5\endcsname{\color[rgb]{1,1,0}}%
      \expandafter\def\csname LT6\endcsname{\color[rgb]{0,0,0}}%
      \expandafter\def\csname LT7\endcsname{\color[rgb]{1,0.3,0}}%
      \expandafter\def\csname LT8\endcsname{\color[rgb]{0.5,0.5,0.5}}%
    \else
      \def\colorrgb#1{\color{black}}%
      \def\colorgray#1{\color[gray]{#1}}%
      \expandafter\def\csname LTw\endcsname{\color{white}}%
      \expandafter\def\csname LTb\endcsname{\color{black}}%
      \expandafter\def\csname LTa\endcsname{\color{black}}%
      \expandafter\def\csname LT0\endcsname{\color{black}}%
      \expandafter\def\csname LT1\endcsname{\color{black}}%
      \expandafter\def\csname LT2\endcsname{\color{black}}%
      \expandafter\def\csname LT3\endcsname{\color{black}}%
      \expandafter\def\csname LT4\endcsname{\color{black}}%
      \expandafter\def\csname LT5\endcsname{\color{black}}%
      \expandafter\def\csname LT6\endcsname{\color{black}}%
      \expandafter\def\csname LT7\endcsname{\color{black}}%
      \expandafter\def\csname LT8\endcsname{\color{black}}%
    \fi
  \fi
    \setlength{\unitlength}{0.0500bp}%
    \ifx\gptboxheight\undefined%
      \newlength{\gptboxheight}%
      \newlength{\gptboxwidth}%
      \newsavebox{\gptboxtext}%
    \fi%
    \setlength{\fboxrule}{0.5pt}%
    \setlength{\fboxsep}{1pt}%
\begin{picture}(7200.00,5040.00)%
    \gplgaddtomacro\gplbacktext{%
      \csname LTb\endcsname%
      \put(3236,806){\makebox(0,0){\strut{}$0$}}%
      \put(3869,918){\makebox(0,0){\strut{}$0.2$}}%
      \put(4503,1030){\makebox(0,0){\strut{}$0.4$}}%
      \put(5136,1142){\makebox(0,0){\strut{}$0.6$}}%
      \put(5770,1253){\makebox(0,0){\strut{}$0.8$}}%
      \put(6403,1365){\makebox(0,0){\strut{}$1$}}%
      \put(2451,884){\makebox(0,0){\strut{}$0$}}%
      \put(1993,1127){\makebox(0,0){\strut{}$0.05$}}%
      \put(1536,1369){\makebox(0,0){\strut{}$0.1$}}%
      \put(1079,1611){\makebox(0,0){\strut{}$0.15$}}%
      \put(622,1853){\makebox(0,0){\strut{}$0.2$}}%
      \put(976,2542){\makebox(0,0)[r]{\strut{}-0.24}}%
      \put(976,2800){\makebox(0,0)[r]{\strut{}-0.12}}%
      \put(976,3058){\makebox(0,0)[r]{\strut{}0.00}}%
      \put(976,3317){\makebox(0,0)[r]{\strut{}0.12}}%
      \put(976,3575){\makebox(0,0)[r]{\strut{}0.24}}%
      \put(976,3833){\makebox(0,0)[r]{\strut{}0.36}}%
      \put(-241,3187){\rotatebox{90}{\makebox(0,0){\strut{}$w(x_1,x_2,0.5)$}}}%
    }%
    \gplgaddtomacro\gplfronttext{%
      \csname LTb\endcsname%
      \put(5307,686){\makebox(0,0){\strut{}$x_1$}}%
      \put(890,993){\makebox(0,0){\strut{}$x_2$}}%
      \put(-241,3187){\rotatebox{90}{\makebox(0,0){\strut{}$w(x_1,x_2,0.5)$}}}%
      \put(6613,2331){\makebox(0,0)[l]{\strut{}-0.24}}%
      \put(6613,2619){\makebox(0,0)[l]{\strut{}-0.12}}%
      \put(6613,2907){\makebox(0,0)[l]{\strut{}0.00}}%
      \put(6613,3195){\makebox(0,0)[l]{\strut{}0.12}}%
      \put(6613,3483){\makebox(0,0)[l]{\strut{}0.24}}%
      \put(6613,3772){\makebox(0,0)[l]{\strut{}0.36}}%
    }%
    \gplbacktext
    \put(0,0){\includegraphics{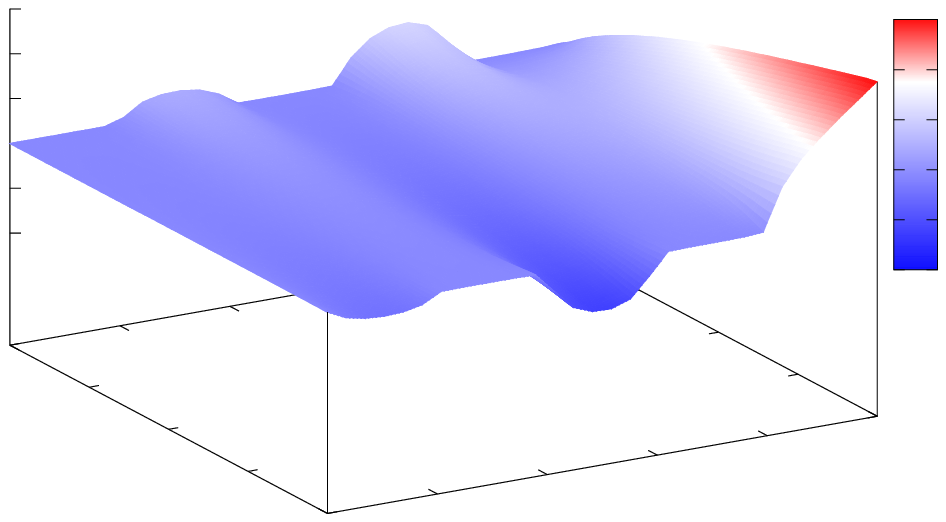}}%
    \gplfronttext
  \end{picture}%
\endgroup

%% file: convDiffAnormEasy_IntVsNonIntResidual.tex
\begingroup
  \makeatletter
  \providecommand\color[2][]{%
    \GenericError{(gnuplot) \space\space\space\@spaces}{%
      Package color not loaded in conjunction with
      terminal option `colourtext'%
    }{See the gnuplot documentation for explanation.%
    }{Either use 'blacktext' in gnuplot or load the package
      color.sty in LaTeX.}%
    \renewcommand\color[2][]{}%
  }%
  \providecommand\includegraphics[2][]{%
    \GenericError{(gnuplot) \space\space\space\@spaces}{%
      Package graphicx or graphics not loaded%
    }{See the gnuplot documentation for explanation.%
    }{The gnuplot epslatex terminal needs graphicx.sty or graphics.sty.}%
    \renewcommand\includegraphics[2][]{}%
  }%
  \providecommand\rotatebox[2]{#2}%
  \@ifundefined{ifGPcolor}{%
    \newif\ifGPcolor
    \GPcolortrue
  }{}%
  \@ifundefined{ifGPblacktext}{%
    \newif\ifGPblacktext
    \GPblacktexttrue
  }{}%
  \let\gplgaddtomacro\g@addto@macro
  \gdef\gplbacktext{}%
  \gdef\gplfronttext{}%
  \makeatother
  \ifGPblacktext
    \def\colorrgb#1{}%
    \def\colorgray#1{}%
  \else
    \ifGPcolor
      \def\colorrgb#1{\color[rgb]{#1}}%
      \def\colorgray#1{\color[gray]{#1}}%
      \expandafter\def\csname LTw\endcsname{\color{white}}%
      \expandafter\def\csname LTb\endcsname{\color{black}}%
      \expandafter\def\csname LTa\endcsname{\color{black}}%
      \expandafter\def\csname LT0\endcsname{\color[rgb]{1,0,0}}%
      \expandafter\def\csname LT1\endcsname{\color[rgb]{0,1,0}}%
      \expandafter\def\csname LT2\endcsname{\color[rgb]{0,0,1}}%
      \expandafter\def\csname LT3\endcsname{\color[rgb]{1,0,1}}%
      \expandafter\def\csname LT4\endcsname{\color[rgb]{0,1,1}}%
      \expandafter\def\csname LT5\endcsname{\color[rgb]{1,1,0}}%
      \expandafter\def\csname LT6\endcsname{\color[rgb]{0,0,0}}%
      \expandafter\def\csname LT7\endcsname{\color[rgb]{1,0.3,0}}%
      \expandafter\def\csname LT8\endcsname{\color[rgb]{0.5,0.5,0.5}}%
    \else
      \def\colorrgb#1{\color{black}}%
      \def\colorgray#1{\color[gray]{#1}}%
      \expandafter\def\csname LTw\endcsname{\color{white}}%
      \expandafter\def\csname LTb\endcsname{\color{black}}%
      \expandafter\def\csname LTa\endcsname{\color{black}}%
      \expandafter\def\csname LT0\endcsname{\color{black}}%
      \expandafter\def\csname LT1\endcsname{\color{black}}%
      \expandafter\def\csname LT2\endcsname{\color{black}}%
      \expandafter\def\csname LT3\endcsname{\color{black}}%
      \expandafter\def\csname LT4\endcsname{\color{black}}%
      \expandafter\def\csname LT5\endcsname{\color{black}}%
      \expandafter\def\csname LT6\endcsname{\color{black}}%
      \expandafter\def\csname LT7\endcsname{\color{black}}%
      \expandafter\def\csname LT8\endcsname{\color{black}}%
    \fi
  \fi
    \setlength{\unitlength}{0.0500bp}%
    \ifx\gptboxheight\undefined%
      \newlength{\gptboxheight}%
      \newlength{\gptboxwidth}%
      \newsavebox{\gptboxtext}%
    \fi%
    \setlength{\fboxrule}{0.5pt}%
    \setlength{\fboxsep}{1pt}%
\begin{picture}(7200.00,5040.00)%
    \gplgaddtomacro\gplbacktext{%
      \csname LTb\endcsname%
      \put(1372,896){\makebox(0,0)[r]{\strut{}1e-07}}%
      \csname LTb\endcsname%
      \put(1372,1657){\makebox(0,0)[r]{\strut{}1e-06}}%
      \csname LTb\endcsname%
      \put(1372,2419){\makebox(0,0)[r]{\strut{}1e-05}}%
      \csname LTb\endcsname%
      \put(1372,3180){\makebox(0,0)[r]{\strut{}1e-04}}%
      \csname LTb\endcsname%
      \put(1372,3942){\makebox(0,0)[r]{\strut{}1e-03}}%
      \csname LTb\endcsname%
      \put(1372,4703){\makebox(0,0)[r]{\strut{}1e-02}}%
      \csname LTb\endcsname%
      \put(1862,616){\makebox(0,0){\strut{}$2$}}%
      \csname LTb\endcsname%
      \put(2507,616){\makebox(0,0){\strut{}$4$}}%
      \csname LTb\endcsname%
      \put(3151,616){\makebox(0,0){\strut{}$6$}}%
      \csname LTb\endcsname%
      \put(3795,616){\makebox(0,0){\strut{}$8$}}%
      \csname LTb\endcsname%
      \put(4440,616){\makebox(0,0){\strut{}$10$}}%
      \csname LTb\endcsname%
      \put(5084,616){\makebox(0,0){\strut{}$12$}}%
      \csname LTb\endcsname%
      \put(5728,616){\makebox(0,0){\strut{}$14$}}%
      \csname LTb\endcsname%
      \put(6373,616){\makebox(0,0){\strut{}$16$}}%
    }%
    \gplgaddtomacro\gplfronttext{%
      \csname LTb\endcsname%
      \put(224,2799){\rotatebox{-270}{\makebox(0,0){\strut{}ave. residual norm over time \eqref{eq:Err_IntVSNonIntUnitNorm}}}}%
      \put(4117,196){\makebox(0,0){\strut{}basis dimension}}%
      \csname LTb\endcsname%
      \put(5876,4430){\makebox(0,0)[r]{\strut{}intrusive model reduction}}%
      \csname LTb\endcsname%
      \put(5876,4010){\makebox(0,0)[r]{\strut{}OpInf}}%
    }%
    \gplbacktext
    \put(0,0){\includegraphics{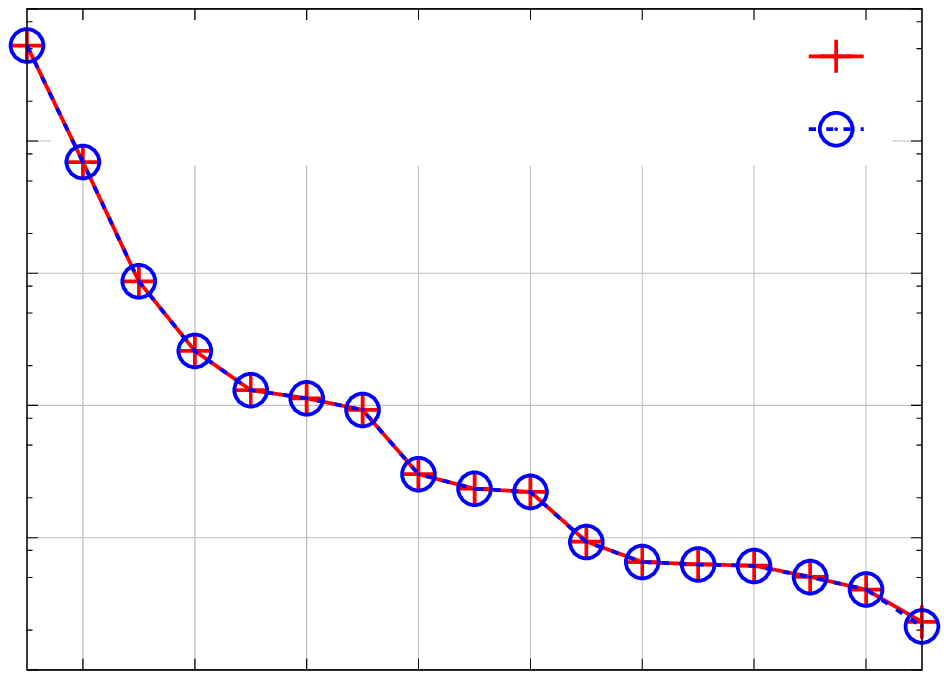}}%
    \gplfronttext
  \end{picture}%
\endgroup

%% file: convDiffAnormEasy_ErrAtTime01VaryGamma.tex
\begingroup
  \makeatletter
  \providecommand\color[2][]{%
    \GenericError{(gnuplot) \space\space\space\@spaces}{%
      Package color not loaded in conjunction with
      terminal option `colourtext'%
    }{See the gnuplot documentation for explanation.%
    }{Either use 'blacktext' in gnuplot or load the package
      color.sty in LaTeX.}%
    \renewcommand\color[2][]{}%
  }%
  \providecommand\includegraphics[2][]{%
    \GenericError{(gnuplot) \space\space\space\@spaces}{%
      Package graphicx or graphics not loaded%
    }{See the gnuplot documentation for explanation.%
    }{The gnuplot epslatex terminal needs graphicx.sty or graphics.sty.}%
    \renewcommand\includegraphics[2][]{}%
  }%
  \providecommand\rotatebox[2]{#2}%
  \@ifundefined{ifGPcolor}{%
    \newif\ifGPcolor
    \GPcolortrue
  }{}%
  \@ifundefined{ifGPblacktext}{%
    \newif\ifGPblacktext
    \GPblacktexttrue
  }{}%
  \let\gplgaddtomacro\g@addto@macro
  \gdef\gplbacktext{}%
  \gdef\gplfronttext{}%
  \makeatother
  \ifGPblacktext
    \def\colorrgb#1{}%
    \def\colorgray#1{}%
  \else
    \ifGPcolor
      \def\colorrgb#1{\color[rgb]{#1}}%
      \def\colorgray#1{\color[gray]{#1}}%
      \expandafter\def\csname LTw\endcsname{\color{white}}%
      \expandafter\def\csname LTb\endcsname{\color{black}}%
      \expandafter\def\csname LTa\endcsname{\color{black}}%
      \expandafter\def\csname LT0\endcsname{\color[rgb]{1,0,0}}%
      \expandafter\def\csname LT1\endcsname{\color[rgb]{0,1,0}}%
      \expandafter\def\csname LT2\endcsname{\color[rgb]{0,0,1}}%
      \expandafter\def\csname LT3\endcsname{\color[rgb]{1,0,1}}%
      \expandafter\def\csname LT4\endcsname{\color[rgb]{0,1,1}}%
      \expandafter\def\csname LT5\endcsname{\color[rgb]{1,1,0}}%
      \expandafter\def\csname LT6\endcsname{\color[rgb]{0,0,0}}%
      \expandafter\def\csname LT7\endcsname{\color[rgb]{1,0.3,0}}%
      \expandafter\def\csname LT8\endcsname{\color[rgb]{0.5,0.5,0.5}}%
    \else
      \def\colorrgb#1{\color{black}}%
      \def\colorgray#1{\color[gray]{#1}}%
      \expandafter\def\csname LTw\endcsname{\color{white}}%
      \expandafter\def\csname LTb\endcsname{\color{black}}%
      \expandafter\def\csname LTa\endcsname{\color{black}}%
      \expandafter\def\csname LT0\endcsname{\color{black}}%
      \expandafter\def\csname LT1\endcsname{\color{black}}%
      \expandafter\def\csname LT2\endcsname{\color{black}}%
      \expandafter\def\csname LT3\endcsname{\color{black}}%
      \expandafter\def\csname LT4\endcsname{\color{black}}%
      \expandafter\def\csname LT5\endcsname{\color{black}}%
      \expandafter\def\csname LT6\endcsname{\color{black}}%
      \expandafter\def\csname LT7\endcsname{\color{black}}%
      \expandafter\def\csname LT8\endcsname{\color{black}}%
    \fi
  \fi
    \setlength{\unitlength}{0.0500bp}%
    \ifx\gptboxheight\undefined%
      \newlength{\gptboxheight}%
      \newlength{\gptboxwidth}%
      \newsavebox{\gptboxtext}%
    \fi%
    \setlength{\fboxrule}{0.5pt}%
    \setlength{\fboxsep}{1pt}%
\begin{picture}(7200.00,5040.00)%
    \gplgaddtomacro\gplbacktext{%
      \csname LTb\endcsname%
      \put(1372,896){\makebox(0,0)[r]{\strut{}1e-05}}%
      \csname LTb\endcsname%
      \put(1372,1372){\makebox(0,0)[r]{\strut{}1e-04}}%
      \csname LTb\endcsname%
      \put(1372,1848){\makebox(0,0)[r]{\strut{}1e-03}}%
      \csname LTb\endcsname%
      \put(1372,2324){\makebox(0,0)[r]{\strut{}1e-02}}%
      \csname LTb\endcsname%
      \put(1372,2800){\makebox(0,0)[r]{\strut{}1e-01}}%
      \csname LTb\endcsname%
      \put(1372,3275){\makebox(0,0)[r]{\strut{}1e+00}}%
      \csname LTb\endcsname%
      \put(1372,3751){\makebox(0,0)[r]{\strut{}1e+01}}%
      \csname LTb\endcsname%
      \put(1372,4227){\makebox(0,0)[r]{\strut{}1e+02}}%
      \csname LTb\endcsname%
      \put(1372,4703){\makebox(0,0)[r]{\strut{}1e+03}}%
      \csname LTb\endcsname%
      \put(1862,616){\makebox(0,0){\strut{}$2$}}%
      \csname LTb\endcsname%
      \put(2507,616){\makebox(0,0){\strut{}$4$}}%
      \csname LTb\endcsname%
      \put(3151,616){\makebox(0,0){\strut{}$6$}}%
      \csname LTb\endcsname%
      \put(3795,616){\makebox(0,0){\strut{}$8$}}%
      \csname LTb\endcsname%
      \put(4440,616){\makebox(0,0){\strut{}$10$}}%
      \csname LTb\endcsname%
      \put(5084,616){\makebox(0,0){\strut{}$12$}}%
      \csname LTb\endcsname%
      \put(5728,616){\makebox(0,0){\strut{}$14$}}%
      \csname LTb\endcsname%
      \put(6373,616){\makebox(0,0){\strut{}$16$}}%
    }%
    \gplgaddtomacro\gplfronttext{%
      \csname LTb\endcsname%
      \put(224,2799){\rotatebox{-270}{\makebox(0,0){\strut{}state err. and err. bounds at $t=0.1$}}}%
      \put(4117,196){\makebox(0,0){\strut{}basis dimension}}%
      \csname LTb\endcsname%
      \put(5939,4515){\makebox(0,0)[r]{\strut{}OpInf error \eqref{eq:Err_TimeActualErr}}}%
      \csname LTb\endcsname%
      \put(5939,4265){\makebox(0,0)[r]{\strut{}intrusive err. est. \eqref{eq:Err_TimeActualBnd}}}%
      \csname LTb\endcsname%
      \put(5939,4015){\makebox(0,0)[r]{\strut{}$\gamma = 7$ \eqref{eq:Err_TimeProbBnd}}}%
      \csname LTb\endcsname%
      \put(5939,3765){\makebox(0,0)[r]{\strut{}$\gamma = 20$ \eqref{eq:Err_TimeProbBnd}}}%
      \csname LTb\endcsname%
      \put(5939,3515){\makebox(0,0)[r]{\strut{}$\gamma=50$  \eqref{eq:Err_TimeProbBnd}}}%
    }%
    \gplbacktext
    \put(0,0){\includegraphics{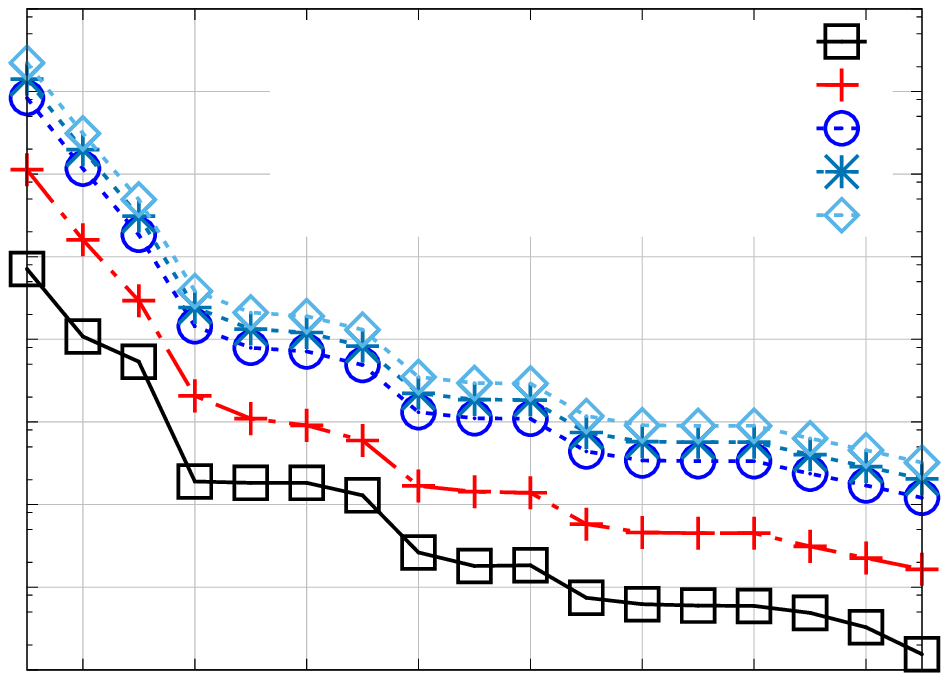}}%
    \gplfronttext
  \end{picture}%
\endgroup

%% file: convDiffAnormEasy_ErrAtTime05VaryGamma.tex
\begingroup
  \makeatletter
  \providecommand\color[2][]{%
    \GenericError{(gnuplot) \space\space\space\@spaces}{%
      Package color not loaded in conjunction with
      terminal option `colourtext'%
    }{See the gnuplot documentation for explanation.%
    }{Either use 'blacktext' in gnuplot or load the package
      color.sty in LaTeX.}%
    \renewcommand\color[2][]{}%
  }%
  \providecommand\includegraphics[2][]{%
    \GenericError{(gnuplot) \space\space\space\@spaces}{%
      Package graphicx or graphics not loaded%
    }{See the gnuplot documentation for explanation.%
    }{The gnuplot epslatex terminal needs graphicx.sty or graphics.sty.}%
    \renewcommand\includegraphics[2][]{}%
  }%
  \providecommand\rotatebox[2]{#2}%
  \@ifundefined{ifGPcolor}{%
    \newif\ifGPcolor
    \GPcolortrue
  }{}%
  \@ifundefined{ifGPblacktext}{%
    \newif\ifGPblacktext
    \GPblacktexttrue
  }{}%
  \let\gplgaddtomacro\g@addto@macro
  \gdef\gplbacktext{}%
  \gdef\gplfronttext{}%
  \makeatother
  \ifGPblacktext
    \def\colorrgb#1{}%
    \def\colorgray#1{}%
  \else
    \ifGPcolor
      \def\colorrgb#1{\color[rgb]{#1}}%
      \def\colorgray#1{\color[gray]{#1}}%
      \expandafter\def\csname LTw\endcsname{\color{white}}%
      \expandafter\def\csname LTb\endcsname{\color{black}}%
      \expandafter\def\csname LTa\endcsname{\color{black}}%
      \expandafter\def\csname LT0\endcsname{\color[rgb]{1,0,0}}%
      \expandafter\def\csname LT1\endcsname{\color[rgb]{0,1,0}}%
      \expandafter\def\csname LT2\endcsname{\color[rgb]{0,0,1}}%
      \expandafter\def\csname LT3\endcsname{\color[rgb]{1,0,1}}%
      \expandafter\def\csname LT4\endcsname{\color[rgb]{0,1,1}}%
      \expandafter\def\csname LT5\endcsname{\color[rgb]{1,1,0}}%
      \expandafter\def\csname LT6\endcsname{\color[rgb]{0,0,0}}%
      \expandafter\def\csname LT7\endcsname{\color[rgb]{1,0.3,0}}%
      \expandafter\def\csname LT8\endcsname{\color[rgb]{0.5,0.5,0.5}}%
    \else
      \def\colorrgb#1{\color{black}}%
      \def\colorgray#1{\color[gray]{#1}}%
      \expandafter\def\csname LTw\endcsname{\color{white}}%
      \expandafter\def\csname LTb\endcsname{\color{black}}%
      \expandafter\def\csname LTa\endcsname{\color{black}}%
      \expandafter\def\csname LT0\endcsname{\color{black}}%
      \expandafter\def\csname LT1\endcsname{\color{black}}%
      \expandafter\def\csname LT2\endcsname{\color{black}}%
      \expandafter\def\csname LT3\endcsname{\color{black}}%
      \expandafter\def\csname LT4\endcsname{\color{black}}%
      \expandafter\def\csname LT5\endcsname{\color{black}}%
      \expandafter\def\csname LT6\endcsname{\color{black}}%
      \expandafter\def\csname LT7\endcsname{\color{black}}%
      \expandafter\def\csname LT8\endcsname{\color{black}}%
    \fi
  \fi
    \setlength{\unitlength}{0.0500bp}%
    \ifx\gptboxheight\undefined%
      \newlength{\gptboxheight}%
      \newlength{\gptboxwidth}%
      \newsavebox{\gptboxtext}%
    \fi%
    \setlength{\fboxrule}{0.5pt}%
    \setlength{\fboxsep}{1pt}%
\begin{picture}(7200.00,5040.00)%
    \gplgaddtomacro\gplbacktext{%
      \csname LTb\endcsname%
      \put(1372,896){\makebox(0,0)[r]{\strut{}1e-05}}%
      \csname LTb\endcsname%
      \put(1372,1372){\makebox(0,0)[r]{\strut{}1e-04}}%
      \csname LTb\endcsname%
      \put(1372,1848){\makebox(0,0)[r]{\strut{}1e-03}}%
      \csname LTb\endcsname%
      \put(1372,2324){\makebox(0,0)[r]{\strut{}1e-02}}%
      \csname LTb\endcsname%
      \put(1372,2800){\makebox(0,0)[r]{\strut{}1e-01}}%
      \csname LTb\endcsname%
      \put(1372,3275){\makebox(0,0)[r]{\strut{}1e+00}}%
      \csname LTb\endcsname%
      \put(1372,3751){\makebox(0,0)[r]{\strut{}1e+01}}%
      \csname LTb\endcsname%
      \put(1372,4227){\makebox(0,0)[r]{\strut{}1e+02}}%
      \csname LTb\endcsname%
      \put(1372,4703){\makebox(0,0)[r]{\strut{}1e+03}}%
      \csname LTb\endcsname%
      \put(1862,616){\makebox(0,0){\strut{}$2$}}%
      \csname LTb\endcsname%
      \put(2507,616){\makebox(0,0){\strut{}$4$}}%
      \csname LTb\endcsname%
      \put(3151,616){\makebox(0,0){\strut{}$6$}}%
      \csname LTb\endcsname%
      \put(3795,616){\makebox(0,0){\strut{}$8$}}%
      \csname LTb\endcsname%
      \put(4440,616){\makebox(0,0){\strut{}$10$}}%
      \csname LTb\endcsname%
      \put(5084,616){\makebox(0,0){\strut{}$12$}}%
      \csname LTb\endcsname%
      \put(5728,616){\makebox(0,0){\strut{}$14$}}%
      \csname LTb\endcsname%
      \put(6373,616){\makebox(0,0){\strut{}$16$}}%
    }%
    \gplgaddtomacro\gplfronttext{%
      \csname LTb\endcsname%
      \put(224,2799){\rotatebox{-270}{\makebox(0,0){\strut{}state err. and err. bounds at $t=0.5$}}}%
      \put(4117,196){\makebox(0,0){\strut{}basis dimension}}%
      \csname LTb\endcsname%
      \put(5939,4515){\makebox(0,0)[r]{\strut{}OpInf error \eqref{eq:Err_TimeActualErr}}}%
      \csname LTb\endcsname%
      \put(5939,4265){\makebox(0,0)[r]{\strut{}intrusive err. est. \eqref{eq:Err_TimeActualBnd}}}%
      \csname LTb\endcsname%
      \put(5939,4015){\makebox(0,0)[r]{\strut{}$\gamma = 7$ \eqref{eq:Err_TimeProbBnd}}}%
      \csname LTb\endcsname%
      \put(5939,3765){\makebox(0,0)[r]{\strut{}$\gamma = 20$ \eqref{eq:Err_TimeProbBnd}}}%
      \csname LTb\endcsname%
      \put(5939,3515){\makebox(0,0)[r]{\strut{}$\gamma=50$  \eqref{eq:Err_TimeProbBnd}}}%
    }%
    \gplbacktext
    \put(0,0){\includegraphics{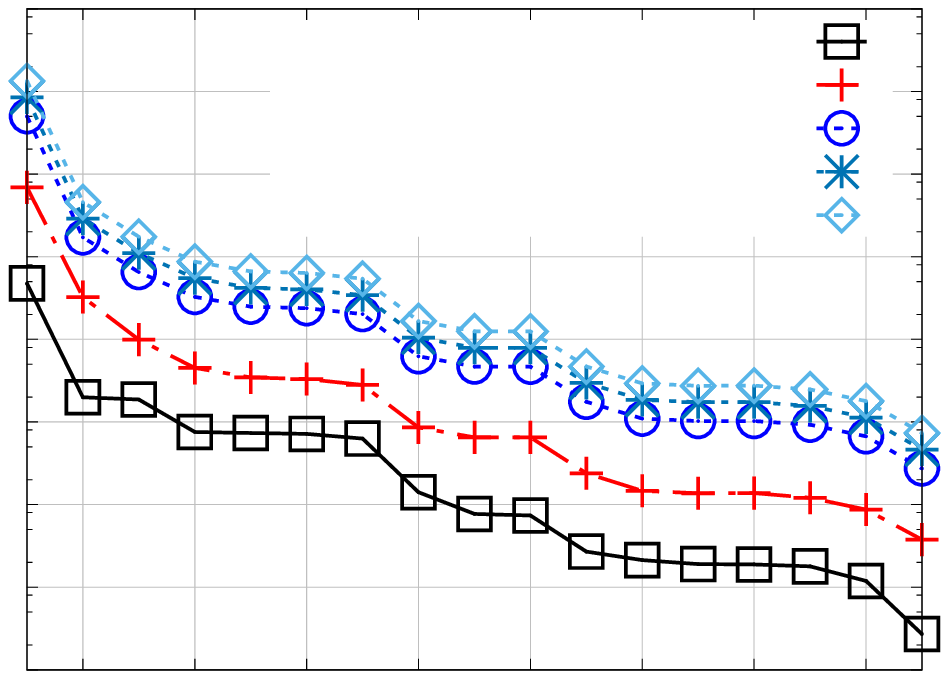}}%
    \gplfronttext
  \end{picture}%
\endgroup

%% file: convDiffAnormEasy_ErrAtTime01VaryM.tex
\begingroup
  \makeatletter
  \providecommand\color[2][]{%
    \GenericError{(gnuplot) \space\space\space\@spaces}{%
      Package color not loaded in conjunction with
      terminal option `colourtext'%
    }{See the gnuplot documentation for explanation.%
    }{Either use 'blacktext' in gnuplot or load the package
      color.sty in LaTeX.}%
    \renewcommand\color[2][]{}%
  }%
  \providecommand\includegraphics[2][]{%
    \GenericError{(gnuplot) \space\space\space\@spaces}{%
      Package graphicx or graphics not loaded%
    }{See the gnuplot documentation for explanation.%
    }{The gnuplot epslatex terminal needs graphicx.sty or graphics.sty.}%
    \renewcommand\includegraphics[2][]{}%
  }%
  \providecommand\rotatebox[2]{#2}%
  \@ifundefined{ifGPcolor}{%
    \newif\ifGPcolor
    \GPcolortrue
  }{}%
  \@ifundefined{ifGPblacktext}{%
    \newif\ifGPblacktext
    \GPblacktexttrue
  }{}%
  \let\gplgaddtomacro\g@addto@macro
  \gdef\gplbacktext{}%
  \gdef\gplfronttext{}%
  \makeatother
  \ifGPblacktext
    \def\colorrgb#1{}%
    \def\colorgray#1{}%
  \else
    \ifGPcolor
      \def\colorrgb#1{\color[rgb]{#1}}%
      \def\colorgray#1{\color[gray]{#1}}%
      \expandafter\def\csname LTw\endcsname{\color{white}}%
      \expandafter\def\csname LTb\endcsname{\color{black}}%
      \expandafter\def\csname LTa\endcsname{\color{black}}%
      \expandafter\def\csname LT0\endcsname{\color[rgb]{1,0,0}}%
      \expandafter\def\csname LT1\endcsname{\color[rgb]{0,1,0}}%
      \expandafter\def\csname LT2\endcsname{\color[rgb]{0,0,1}}%
      \expandafter\def\csname LT3\endcsname{\color[rgb]{1,0,1}}%
      \expandafter\def\csname LT4\endcsname{\color[rgb]{0,1,1}}%
      \expandafter\def\csname LT5\endcsname{\color[rgb]{1,1,0}}%
      \expandafter\def\csname LT6\endcsname{\color[rgb]{0,0,0}}%
      \expandafter\def\csname LT7\endcsname{\color[rgb]{1,0.3,0}}%
      \expandafter\def\csname LT8\endcsname{\color[rgb]{0.5,0.5,0.5}}%
    \else
      \def\colorrgb#1{\color{black}}%
      \def\colorgray#1{\color[gray]{#1}}%
      \expandafter\def\csname LTw\endcsname{\color{white}}%
      \expandafter\def\csname LTb\endcsname{\color{black}}%
      \expandafter\def\csname LTa\endcsname{\color{black}}%
      \expandafter\def\csname LT0\endcsname{\color{black}}%
      \expandafter\def\csname LT1\endcsname{\color{black}}%
      \expandafter\def\csname LT2\endcsname{\color{black}}%
      \expandafter\def\csname LT3\endcsname{\color{black}}%
      \expandafter\def\csname LT4\endcsname{\color{black}}%
      \expandafter\def\csname LT5\endcsname{\color{black}}%
      \expandafter\def\csname LT6\endcsname{\color{black}}%
      \expandafter\def\csname LT7\endcsname{\color{black}}%
      \expandafter\def\csname LT8\endcsname{\color{black}}%
    \fi
  \fi
    \setlength{\unitlength}{0.0500bp}%
    \ifx\gptboxheight\undefined%
      \newlength{\gptboxheight}%
      \newlength{\gptboxwidth}%
      \newsavebox{\gptboxtext}%
    \fi%
    \setlength{\fboxrule}{0.5pt}%
    \setlength{\fboxsep}{1pt}%
\begin{picture}(7200.00,5040.00)%
    \gplgaddtomacro\gplbacktext{%
      \csname LTb\endcsname%
      \put(1372,896){\makebox(0,0)[r]{\strut{}1e-05}}%
      \csname LTb\endcsname%
      \put(1372,1440){\makebox(0,0)[r]{\strut{}1e-04}}%
      \csname LTb\endcsname%
      \put(1372,1984){\makebox(0,0)[r]{\strut{}1e-03}}%
      \csname LTb\endcsname%
      \put(1372,2528){\makebox(0,0)[r]{\strut{}1e-02}}%
      \csname LTb\endcsname%
      \put(1372,3071){\makebox(0,0)[r]{\strut{}1e-01}}%
      \csname LTb\endcsname%
      \put(1372,3615){\makebox(0,0)[r]{\strut{}1e+00}}%
      \csname LTb\endcsname%
      \put(1372,4159){\makebox(0,0)[r]{\strut{}1e+01}}%
      \csname LTb\endcsname%
      \put(1372,4703){\makebox(0,0)[r]{\strut{}1e+02}}%
      \csname LTb\endcsname%
      \put(1862,616){\makebox(0,0){\strut{}$2$}}%
      \csname LTb\endcsname%
      \put(2507,616){\makebox(0,0){\strut{}$4$}}%
      \csname LTb\endcsname%
      \put(3151,616){\makebox(0,0){\strut{}$6$}}%
      \csname LTb\endcsname%
      \put(3795,616){\makebox(0,0){\strut{}$8$}}%
      \csname LTb\endcsname%
      \put(4440,616){\makebox(0,0){\strut{}$10$}}%
      \csname LTb\endcsname%
      \put(5084,616){\makebox(0,0){\strut{}$12$}}%
      \csname LTb\endcsname%
      \put(5728,616){\makebox(0,0){\strut{}$14$}}%
      \csname LTb\endcsname%
      \put(6373,616){\makebox(0,0){\strut{}$16$}}%
    }%
    \gplgaddtomacro\gplfronttext{%
      \csname LTb\endcsname%
      \put(224,2799){\rotatebox{-270}{\makebox(0,0){\strut{}state err. and err. bounds at $t=0.1$}}}%
      \put(4117,196){\makebox(0,0){\strut{}basis dimension}}%
      \csname LTb\endcsname%
      \put(5939,4515){\makebox(0,0)[r]{\strut{}OpInf error \eqref{eq:Err_TimeActualErr}}}%
      \csname LTb\endcsname%
      \put(5939,4265){\makebox(0,0)[r]{\strut{}intrusive err. est. \eqref{eq:Err_TimeActualBnd}}}%
      \csname LTb\endcsname%
      \put(5939,4015){\makebox(0,0)[r]{\strut{}$M = 35$ \eqref{eq:Err_TimeProbBnd}}}%
      \csname LTb\endcsname%
      \put(5939,3765){\makebox(0,0)[r]{\strut{}$M = 100$ \eqref{eq:Err_TimeProbBnd}}}%
      \csname LTb\endcsname%
      \put(5939,3515){\makebox(0,0)[r]{\strut{}$M = 500$  \eqref{eq:Err_TimeProbBnd}}}%
    }%
    \gplbacktext
    \put(0,0){\includegraphics{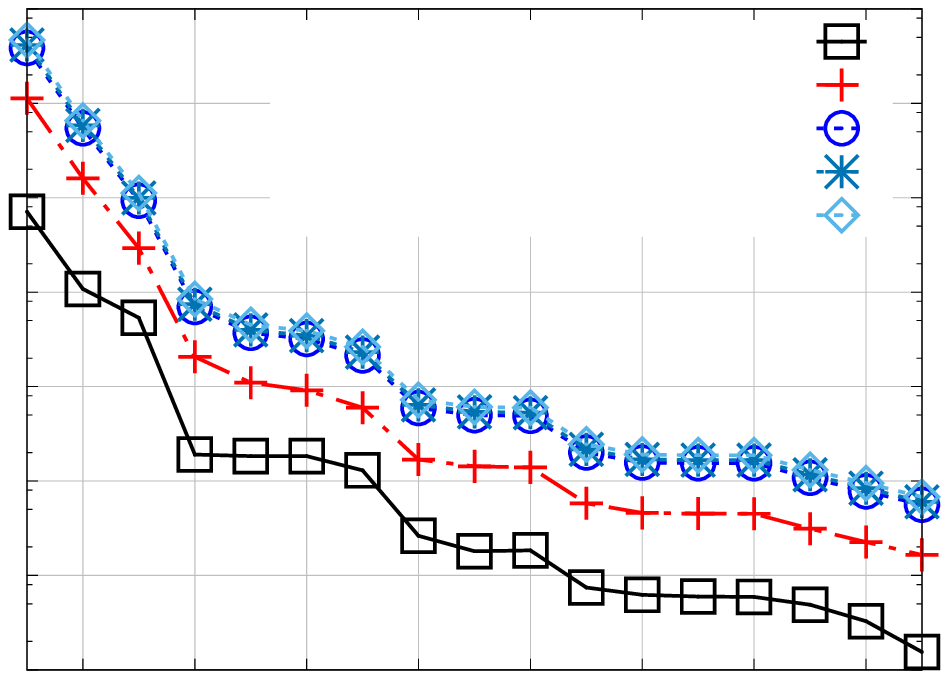}}%
    \gplfronttext
  \end{picture}%
\endgroup

%% file: convDiffAnormEasy_ErrAtTime05VaryM.tex
\begingroup
  \makeatletter
  \providecommand\color[2][]{%
    \GenericError{(gnuplot) \space\space\space\@spaces}{%
      Package color not loaded in conjunction with
      terminal option `colourtext'%
    }{See the gnuplot documentation for explanation.%
    }{Either use 'blacktext' in gnuplot or load the package
      color.sty in LaTeX.}%
    \renewcommand\color[2][]{}%
  }%
  \providecommand\includegraphics[2][]{%
    \GenericError{(gnuplot) \space\space\space\@spaces}{%
      Package graphicx or graphics not loaded%
    }{See the gnuplot documentation for explanation.%
    }{The gnuplot epslatex terminal needs graphicx.sty or graphics.sty.}%
    \renewcommand\includegraphics[2][]{}%
  }%
  \providecommand\rotatebox[2]{#2}%
  \@ifundefined{ifGPcolor}{%
    \newif\ifGPcolor
    \GPcolortrue
  }{}%
  \@ifundefined{ifGPblacktext}{%
    \newif\ifGPblacktext
    \GPblacktexttrue
  }{}%
  \let\gplgaddtomacro\g@addto@macro
  \gdef\gplbacktext{}%
  \gdef\gplfronttext{}%
  \makeatother
  \ifGPblacktext
    \def\colorrgb#1{}%
    \def\colorgray#1{}%
  \else
    \ifGPcolor
      \def\colorrgb#1{\color[rgb]{#1}}%
      \def\colorgray#1{\color[gray]{#1}}%
      \expandafter\def\csname LTw\endcsname{\color{white}}%
      \expandafter\def\csname LTb\endcsname{\color{black}}%
      \expandafter\def\csname LTa\endcsname{\color{black}}%
      \expandafter\def\csname LT0\endcsname{\color[rgb]{1,0,0}}%
      \expandafter\def\csname LT1\endcsname{\color[rgb]{0,1,0}}%
      \expandafter\def\csname LT2\endcsname{\color[rgb]{0,0,1}}%
      \expandafter\def\csname LT3\endcsname{\color[rgb]{1,0,1}}%
      \expandafter\def\csname LT4\endcsname{\color[rgb]{0,1,1}}%
      \expandafter\def\csname LT5\endcsname{\color[rgb]{1,1,0}}%
      \expandafter\def\csname LT6\endcsname{\color[rgb]{0,0,0}}%
      \expandafter\def\csname LT7\endcsname{\color[rgb]{1,0.3,0}}%
      \expandafter\def\csname LT8\endcsname{\color[rgb]{0.5,0.5,0.5}}%
    \else
      \def\colorrgb#1{\color{black}}%
      \def\colorgray#1{\color[gray]{#1}}%
      \expandafter\def\csname LTw\endcsname{\color{white}}%
      \expandafter\def\csname LTb\endcsname{\color{black}}%
      \expandafter\def\csname LTa\endcsname{\color{black}}%
      \expandafter\def\csname LT0\endcsname{\color{black}}%
      \expandafter\def\csname LT1\endcsname{\color{black}}%
      \expandafter\def\csname LT2\endcsname{\color{black}}%
      \expandafter\def\csname LT3\endcsname{\color{black}}%
      \expandafter\def\csname LT4\endcsname{\color{black}}%
      \expandafter\def\csname LT5\endcsname{\color{black}}%
      \expandafter\def\csname LT6\endcsname{\color{black}}%
      \expandafter\def\csname LT7\endcsname{\color{black}}%
      \expandafter\def\csname LT8\endcsname{\color{black}}%
    \fi
  \fi
    \setlength{\unitlength}{0.0500bp}%
    \ifx\gptboxheight\undefined%
      \newlength{\gptboxheight}%
      \newlength{\gptboxwidth}%
      \newsavebox{\gptboxtext}%
    \fi%
    \setlength{\fboxrule}{0.5pt}%
    \setlength{\fboxsep}{1pt}%
\begin{picture}(7200.00,5040.00)%
    \gplgaddtomacro\gplbacktext{%
      \csname LTb\endcsname%
      \put(1372,896){\makebox(0,0)[r]{\strut{}1e-05}}%
      \csname LTb\endcsname%
      \put(1372,1440){\makebox(0,0)[r]{\strut{}1e-04}}%
      \csname LTb\endcsname%
      \put(1372,1984){\makebox(0,0)[r]{\strut{}1e-03}}%
      \csname LTb\endcsname%
      \put(1372,2528){\makebox(0,0)[r]{\strut{}1e-02}}%
      \csname LTb\endcsname%
      \put(1372,3071){\makebox(0,0)[r]{\strut{}1e-01}}%
      \csname LTb\endcsname%
      \put(1372,3615){\makebox(0,0)[r]{\strut{}1e+00}}%
      \csname LTb\endcsname%
      \put(1372,4159){\makebox(0,0)[r]{\strut{}1e+01}}%
      \csname LTb\endcsname%
      \put(1372,4703){\makebox(0,0)[r]{\strut{}1e+02}}%
      \csname LTb\endcsname%
      \put(1862,616){\makebox(0,0){\strut{}$2$}}%
      \csname LTb\endcsname%
      \put(2507,616){\makebox(0,0){\strut{}$4$}}%
      \csname LTb\endcsname%
      \put(3151,616){\makebox(0,0){\strut{}$6$}}%
      \csname LTb\endcsname%
      \put(3795,616){\makebox(0,0){\strut{}$8$}}%
      \csname LTb\endcsname%
      \put(4440,616){\makebox(0,0){\strut{}$10$}}%
      \csname LTb\endcsname%
      \put(5084,616){\makebox(0,0){\strut{}$12$}}%
      \csname LTb\endcsname%
      \put(5728,616){\makebox(0,0){\strut{}$14$}}%
      \csname LTb\endcsname%
      \put(6373,616){\makebox(0,0){\strut{}$16$}}%
    }%
    \gplgaddtomacro\gplfronttext{%
      \csname LTb\endcsname%
      \put(224,2799){\rotatebox{-270}{\makebox(0,0){\strut{}state err. and err. bounds at $t=0.5$}}}%
      \put(4117,196){\makebox(0,0){\strut{}basis dimension}}%
      \csname LTb\endcsname%
      \put(5939,4515){\makebox(0,0)[r]{\strut{}OpInf error \eqref{eq:Err_TimeActualErr}}}%
      \csname LTb\endcsname%
      \put(5939,4265){\makebox(0,0)[r]{\strut{}intrusive err. est. \eqref{eq:Err_TimeActualBnd}}}%
      \csname LTb\endcsname%
      \put(5939,4015){\makebox(0,0)[r]{\strut{}$M = 35$ \eqref{eq:Err_TimeProbBnd}}}%
      \csname LTb\endcsname%
      \put(5939,3765){\makebox(0,0)[r]{\strut{}$M = 100$ \eqref{eq:Err_TimeProbBnd}}}%
      \csname LTb\endcsname%
      \put(5939,3515){\makebox(0,0)[r]{\strut{}$M = 500$  \eqref{eq:Err_TimeProbBnd}}}%
    }%
    \gplbacktext
    \put(0,0){\includegraphics{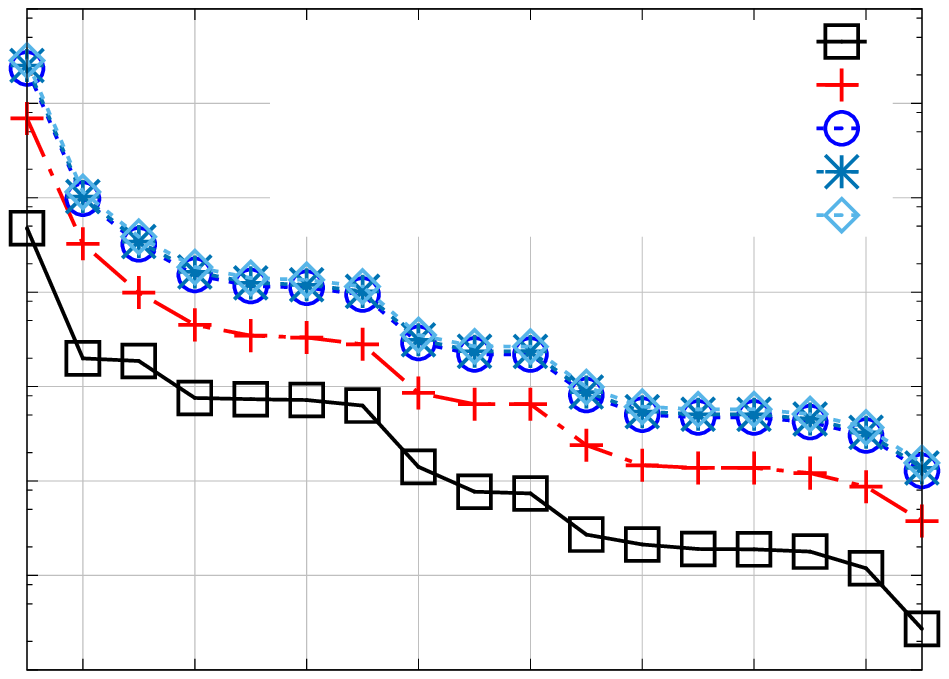}}%
    \gplfronttext
  \end{picture}%
\endgroup

%% file: convDiffAnormEasyMCSamples01.tex
\begingroup
  \makeatletter
  \providecommand\color[2][]{%
    \GenericError{(gnuplot) \space\space\space\@spaces}{%
      Package color not loaded in conjunction with
      terminal option `colourtext'%
    }{See the gnuplot documentation for explanation.%
    }{Either use 'blacktext' in gnuplot or load the package
      color.sty in LaTeX.}%
    \renewcommand\color[2][]{}%
  }%
  \providecommand\includegraphics[2][]{%
    \GenericError{(gnuplot) \space\space\space\@spaces}{%
      Package graphicx or graphics not loaded%
    }{See the gnuplot documentation for explanation.%
    }{The gnuplot epslatex terminal needs graphicx.sty or graphics.sty.}%
    \renewcommand\includegraphics[2][]{}%
  }%
  \providecommand\rotatebox[2]{#2}%
  \@ifundefined{ifGPcolor}{%
    \newif\ifGPcolor
    \GPcolortrue
  }{}%
  \@ifundefined{ifGPblacktext}{%
    \newif\ifGPblacktext
    \GPblacktexttrue
  }{}%
  \let\gplgaddtomacro\g@addto@macro
  \gdef\gplbacktext{}%
  \gdef\gplfronttext{}%
  \makeatother
  \ifGPblacktext
    \def\colorrgb#1{}%
    \def\colorgray#1{}%
  \else
    \ifGPcolor
      \def\colorrgb#1{\color[rgb]{#1}}%
      \def\colorgray#1{\color[gray]{#1}}%
      \expandafter\def\csname LTw\endcsname{\color{white}}%
      \expandafter\def\csname LTb\endcsname{\color{black}}%
      \expandafter\def\csname LTa\endcsname{\color{black}}%
      \expandafter\def\csname LT0\endcsname{\color[rgb]{1,0,0}}%
      \expandafter\def\csname LT1\endcsname{\color[rgb]{0,1,0}}%
      \expandafter\def\csname LT2\endcsname{\color[rgb]{0,0,1}}%
      \expandafter\def\csname LT3\endcsname{\color[rgb]{1,0,1}}%
      \expandafter\def\csname LT4\endcsname{\color[rgb]{0,1,1}}%
      \expandafter\def\csname LT5\endcsname{\color[rgb]{1,1,0}}%
      \expandafter\def\csname LT6\endcsname{\color[rgb]{0,0,0}}%
      \expandafter\def\csname LT7\endcsname{\color[rgb]{1,0.3,0}}%
      \expandafter\def\csname LT8\endcsname{\color[rgb]{0.5,0.5,0.5}}%
    \else
      \def\colorrgb#1{\color{black}}%
      \def\colorgray#1{\color[gray]{#1}}%
      \expandafter\def\csname LTw\endcsname{\color{white}}%
      \expandafter\def\csname LTb\endcsname{\color{black}}%
      \expandafter\def\csname LTa\endcsname{\color{black}}%
      \expandafter\def\csname LT0\endcsname{\color{black}}%
      \expandafter\def\csname LT1\endcsname{\color{black}}%
      \expandafter\def\csname LT2\endcsname{\color{black}}%
      \expandafter\def\csname LT3\endcsname{\color{black}}%
      \expandafter\def\csname LT4\endcsname{\color{black}}%
      \expandafter\def\csname LT5\endcsname{\color{black}}%
      \expandafter\def\csname LT6\endcsname{\color{black}}%
      \expandafter\def\csname LT7\endcsname{\color{black}}%
      \expandafter\def\csname LT8\endcsname{\color{black}}%
    \fi
  \fi
    \setlength{\unitlength}{0.0500bp}%
    \ifx\gptboxheight\undefined%
      \newlength{\gptboxheight}%
      \newlength{\gptboxwidth}%
      \newsavebox{\gptboxtext}%
    \fi%
    \setlength{\fboxrule}{0.5pt}%
    \setlength{\fboxsep}{1pt}%
\begin{picture}(7200.00,5040.00)%
    \gplgaddtomacro\gplbacktext{%
      \csname LTb\endcsname%
      \put(1372,896){\makebox(0,0)[r]{\strut{}1e-04}}%
      \csname LTb\endcsname%
      \put(1372,1531){\makebox(0,0)[r]{\strut{}1e-03}}%
      \csname LTb\endcsname%
      \put(1372,2165){\makebox(0,0)[r]{\strut{}1e-02}}%
      \csname LTb\endcsname%
      \put(1372,2800){\makebox(0,0)[r]{\strut{}1e-01}}%
      \csname LTb\endcsname%
      \put(1372,3434){\makebox(0,0)[r]{\strut{}1e+00}}%
      \csname LTb\endcsname%
      \put(1372,4069){\makebox(0,0)[r]{\strut{}1e+01}}%
      \csname LTb\endcsname%
      \put(1372,4703){\makebox(0,0)[r]{\strut{}1e+02}}%
      \csname LTb\endcsname%
      \put(1862,616){\makebox(0,0){\strut{}$2$}}%
      \csname LTb\endcsname%
      \put(2507,616){\makebox(0,0){\strut{}$4$}}%
      \csname LTb\endcsname%
      \put(3151,616){\makebox(0,0){\strut{}$6$}}%
      \csname LTb\endcsname%
      \put(3795,616){\makebox(0,0){\strut{}$8$}}%
      \csname LTb\endcsname%
      \put(4440,616){\makebox(0,0){\strut{}$10$}}%
      \csname LTb\endcsname%
      \put(5084,616){\makebox(0,0){\strut{}$12$}}%
      \csname LTb\endcsname%
      \put(5728,616){\makebox(0,0){\strut{}$14$}}%
      \csname LTb\endcsname%
      \put(6373,616){\makebox(0,0){\strut{}$16$}}%
    }%
    \gplgaddtomacro\gplfronttext{%
      \csname LTb\endcsname%
      \put(224,2799){\rotatebox{-270}{\makebox(0,0){\strut{}state err. and err. bounds at $t=0.1$}}}%
      \put(4117,196){\makebox(0,0){\strut{}basis dimension}}%
      \csname LTb\endcsname%
      \put(5876,4430){\makebox(0,0)[r]{\strut{}learned err. est. \eqref{eq:Err_TimeProbBnd}}}%
      \csname LTb\endcsname%
      \put(5876,4010){\makebox(0,0)[r]{\strut{}intrusive err. est. \eqref{eq:Err_TimeActualBnd}}}%
    }%
    \gplbacktext
    \put(0,0){\includegraphics{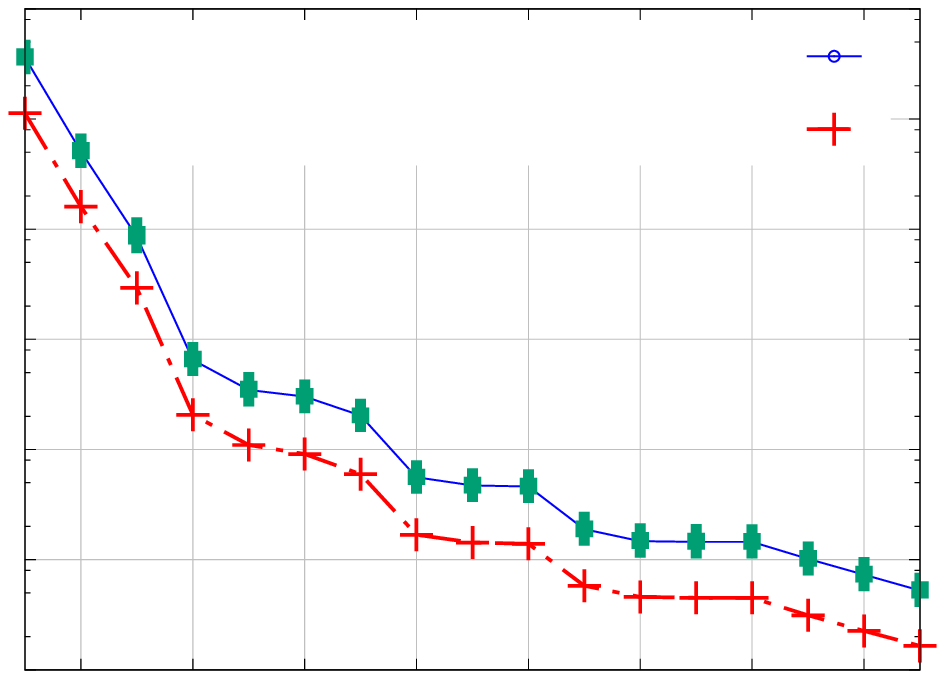}}%
    \gplfronttext
  \end{picture}%
\endgroup

%% file: convDiffAnormEasyMCSamples05.tex
\begingroup
  \makeatletter
  \providecommand\color[2][]{%
    \GenericError{(gnuplot) \space\space\space\@spaces}{%
      Package color not loaded in conjunction with
      terminal option `colourtext'%
    }{See the gnuplot documentation for explanation.%
    }{Either use 'blacktext' in gnuplot or load the package
      color.sty in LaTeX.}%
    \renewcommand\color[2][]{}%
  }%
  \providecommand\includegraphics[2][]{%
    \GenericError{(gnuplot) \space\space\space\@spaces}{%
      Package graphicx or graphics not loaded%
    }{See the gnuplot documentation for explanation.%
    }{The gnuplot epslatex terminal needs graphicx.sty or graphics.sty.}%
    \renewcommand\includegraphics[2][]{}%
  }%
  \providecommand\rotatebox[2]{#2}%
  \@ifundefined{ifGPcolor}{%
    \newif\ifGPcolor
    \GPcolortrue
  }{}%
  \@ifundefined{ifGPblacktext}{%
    \newif\ifGPblacktext
    \GPblacktexttrue
  }{}%
  \let\gplgaddtomacro\g@addto@macro
  \gdef\gplbacktext{}%
  \gdef\gplfronttext{}%
  \makeatother
  \ifGPblacktext
    \def\colorrgb#1{}%
    \def\colorgray#1{}%
  \else
    \ifGPcolor
      \def\colorrgb#1{\color[rgb]{#1}}%
      \def\colorgray#1{\color[gray]{#1}}%
      \expandafter\def\csname LTw\endcsname{\color{white}}%
      \expandafter\def\csname LTb\endcsname{\color{black}}%
      \expandafter\def\csname LTa\endcsname{\color{black}}%
      \expandafter\def\csname LT0\endcsname{\color[rgb]{1,0,0}}%
      \expandafter\def\csname LT1\endcsname{\color[rgb]{0,1,0}}%
      \expandafter\def\csname LT2\endcsname{\color[rgb]{0,0,1}}%
      \expandafter\def\csname LT3\endcsname{\color[rgb]{1,0,1}}%
      \expandafter\def\csname LT4\endcsname{\color[rgb]{0,1,1}}%
      \expandafter\def\csname LT5\endcsname{\color[rgb]{1,1,0}}%
      \expandafter\def\csname LT6\endcsname{\color[rgb]{0,0,0}}%
      \expandafter\def\csname LT7\endcsname{\color[rgb]{1,0.3,0}}%
      \expandafter\def\csname LT8\endcsname{\color[rgb]{0.5,0.5,0.5}}%
    \else
      \def\colorrgb#1{\color{black}}%
      \def\colorgray#1{\color[gray]{#1}}%
      \expandafter\def\csname LTw\endcsname{\color{white}}%
      \expandafter\def\csname LTb\endcsname{\color{black}}%
      \expandafter\def\csname LTa\endcsname{\color{black}}%
      \expandafter\def\csname LT0\endcsname{\color{black}}%
      \expandafter\def\csname LT1\endcsname{\color{black}}%
      \expandafter\def\csname LT2\endcsname{\color{black}}%
      \expandafter\def\csname LT3\endcsname{\color{black}}%
      \expandafter\def\csname LT4\endcsname{\color{black}}%
      \expandafter\def\csname LT5\endcsname{\color{black}}%
      \expandafter\def\csname LT6\endcsname{\color{black}}%
      \expandafter\def\csname LT7\endcsname{\color{black}}%
      \expandafter\def\csname LT8\endcsname{\color{black}}%
    \fi
  \fi
    \setlength{\unitlength}{0.0500bp}%
    \ifx\gptboxheight\undefined%
      \newlength{\gptboxheight}%
      \newlength{\gptboxwidth}%
      \newsavebox{\gptboxtext}%
    \fi%
    \setlength{\fboxrule}{0.5pt}%
    \setlength{\fboxsep}{1pt}%
\begin{picture}(7200.00,5040.00)%
    \gplgaddtomacro\gplbacktext{%
      \csname LTb\endcsname%
      \put(1372,896){\makebox(0,0)[r]{\strut{}1e-04}}%
      \csname LTb\endcsname%
      \put(1372,1531){\makebox(0,0)[r]{\strut{}1e-03}}%
      \csname LTb\endcsname%
      \put(1372,2165){\makebox(0,0)[r]{\strut{}1e-02}}%
      \csname LTb\endcsname%
      \put(1372,2800){\makebox(0,0)[r]{\strut{}1e-01}}%
      \csname LTb\endcsname%
      \put(1372,3434){\makebox(0,0)[r]{\strut{}1e+00}}%
      \csname LTb\endcsname%
      \put(1372,4069){\makebox(0,0)[r]{\strut{}1e+01}}%
      \csname LTb\endcsname%
      \put(1372,4703){\makebox(0,0)[r]{\strut{}1e+02}}%
      \csname LTb\endcsname%
      \put(1862,616){\makebox(0,0){\strut{}$2$}}%
      \csname LTb\endcsname%
      \put(2507,616){\makebox(0,0){\strut{}$4$}}%
      \csname LTb\endcsname%
      \put(3151,616){\makebox(0,0){\strut{}$6$}}%
      \csname LTb\endcsname%
      \put(3795,616){\makebox(0,0){\strut{}$8$}}%
      \csname LTb\endcsname%
      \put(4440,616){\makebox(0,0){\strut{}$10$}}%
      \csname LTb\endcsname%
      \put(5084,616){\makebox(0,0){\strut{}$12$}}%
      \csname LTb\endcsname%
      \put(5728,616){\makebox(0,0){\strut{}$14$}}%
      \csname LTb\endcsname%
      \put(6373,616){\makebox(0,0){\strut{}$16$}}%
    }%
    \gplgaddtomacro\gplfronttext{%
      \csname LTb\endcsname%
      \put(224,2799){\rotatebox{-270}{\makebox(0,0){\strut{}state err. and err. bounds at $t=0.5$}}}%
      \put(4117,196){\makebox(0,0){\strut{}basis dimension}}%
      \csname LTb\endcsname%
      \put(5876,4430){\makebox(0,0)[r]{\strut{}learned err. est. \eqref{eq:Err_TimeProbBnd}}}%
      \csname LTb\endcsname%
      \put(5876,4010){\makebox(0,0)[r]{\strut{}intrusive err. est. \eqref{eq:Err_TimeActualBnd}}}%
    }%
    \gplbacktext
    \put(0,0){\includegraphics{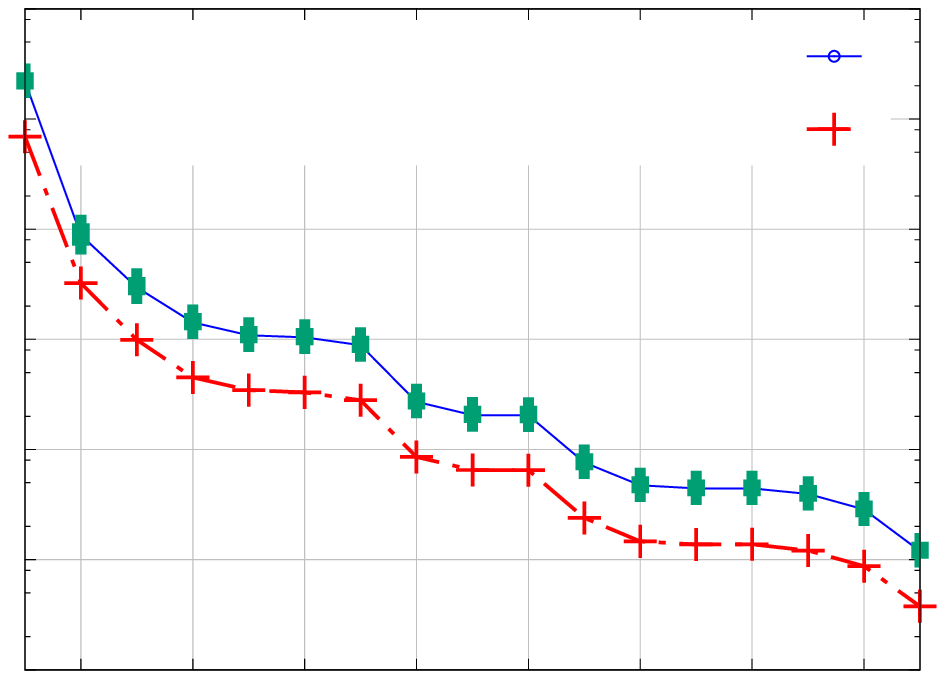}}%
    \gplfronttext
  \end{picture}%
\endgroup

%% file: convDiffAnormHard_IntVsNonIntResidual.tex
\begingroup
  \makeatletter
  \providecommand\color[2][]{%
    \GenericError{(gnuplot) \space\space\space\@spaces}{%
      Package color not loaded in conjunction with
      terminal option `colourtext'%
    }{See the gnuplot documentation for explanation.%
    }{Either use 'blacktext' in gnuplot or load the package
      color.sty in LaTeX.}%
    \renewcommand\color[2][]{}%
  }%
  \providecommand\includegraphics[2][]{%
    \GenericError{(gnuplot) \space\space\space\@spaces}{%
      Package graphicx or graphics not loaded%
    }{See the gnuplot documentation for explanation.%
    }{The gnuplot epslatex terminal needs graphicx.sty or graphics.sty.}%
    \renewcommand\includegraphics[2][]{}%
  }%
  \providecommand\rotatebox[2]{#2}%
  \@ifundefined{ifGPcolor}{%
    \newif\ifGPcolor
    \GPcolortrue
  }{}%
  \@ifundefined{ifGPblacktext}{%
    \newif\ifGPblacktext
    \GPblacktexttrue
  }{}%
  \let\gplgaddtomacro\g@addto@macro
  \gdef\gplbacktext{}%
  \gdef\gplfronttext{}%
  \makeatother
  \ifGPblacktext
    \def\colorrgb#1{}%
    \def\colorgray#1{}%
  \else
    \ifGPcolor
      \def\colorrgb#1{\color[rgb]{#1}}%
      \def\colorgray#1{\color[gray]{#1}}%
      \expandafter\def\csname LTw\endcsname{\color{white}}%
      \expandafter\def\csname LTb\endcsname{\color{black}}%
      \expandafter\def\csname LTa\endcsname{\color{black}}%
      \expandafter\def\csname LT0\endcsname{\color[rgb]{1,0,0}}%
      \expandafter\def\csname LT1\endcsname{\color[rgb]{0,1,0}}%
      \expandafter\def\csname LT2\endcsname{\color[rgb]{0,0,1}}%
      \expandafter\def\csname LT3\endcsname{\color[rgb]{1,0,1}}%
      \expandafter\def\csname LT4\endcsname{\color[rgb]{0,1,1}}%
      \expandafter\def\csname LT5\endcsname{\color[rgb]{1,1,0}}%
      \expandafter\def\csname LT6\endcsname{\color[rgb]{0,0,0}}%
      \expandafter\def\csname LT7\endcsname{\color[rgb]{1,0.3,0}}%
      \expandafter\def\csname LT8\endcsname{\color[rgb]{0.5,0.5,0.5}}%
    \else
      \def\colorrgb#1{\color{black}}%
      \def\colorgray#1{\color[gray]{#1}}%
      \expandafter\def\csname LTw\endcsname{\color{white}}%
      \expandafter\def\csname LTb\endcsname{\color{black}}%
      \expandafter\def\csname LTa\endcsname{\color{black}}%
      \expandafter\def\csname LT0\endcsname{\color{black}}%
      \expandafter\def\csname LT1\endcsname{\color{black}}%
      \expandafter\def\csname LT2\endcsname{\color{black}}%
      \expandafter\def\csname LT3\endcsname{\color{black}}%
      \expandafter\def\csname LT4\endcsname{\color{black}}%
      \expandafter\def\csname LT5\endcsname{\color{black}}%
      \expandafter\def\csname LT6\endcsname{\color{black}}%
      \expandafter\def\csname LT7\endcsname{\color{black}}%
      \expandafter\def\csname LT8\endcsname{\color{black}}%
    \fi
  \fi
    \setlength{\unitlength}{0.0500bp}%
    \ifx\gptboxheight\undefined%
      \newlength{\gptboxheight}%
      \newlength{\gptboxwidth}%
      \newsavebox{\gptboxtext}%
    \fi%
    \setlength{\fboxrule}{0.5pt}%
    \setlength{\fboxsep}{1pt}%
\begin{picture}(7200.00,5040.00)%
    \gplgaddtomacro\gplbacktext{%
      \csname LTb\endcsname%
      \put(1372,896){\makebox(0,0)[r]{\strut{}1e-05}}%
      \csname LTb\endcsname%
      \put(1372,1848){\makebox(0,0)[r]{\strut{}1e-04}}%
      \csname LTb\endcsname%
      \put(1372,2800){\makebox(0,0)[r]{\strut{}1e-03}}%
      \csname LTb\endcsname%
      \put(1372,3751){\makebox(0,0)[r]{\strut{}1e-02}}%
      \csname LTb\endcsname%
      \put(1372,4703){\makebox(0,0)[r]{\strut{}1e-01}}%
      \csname LTb\endcsname%
      \put(1862,616){\makebox(0,0){\strut{}$2$}}%
      \csname LTb\endcsname%
      \put(2507,616){\makebox(0,0){\strut{}$4$}}%
      \csname LTb\endcsname%
      \put(3151,616){\makebox(0,0){\strut{}$6$}}%
      \csname LTb\endcsname%
      \put(3795,616){\makebox(0,0){\strut{}$8$}}%
      \csname LTb\endcsname%
      \put(4440,616){\makebox(0,0){\strut{}$10$}}%
      \csname LTb\endcsname%
      \put(5084,616){\makebox(0,0){\strut{}$12$}}%
      \csname LTb\endcsname%
      \put(5728,616){\makebox(0,0){\strut{}$14$}}%
      \csname LTb\endcsname%
      \put(6373,616){\makebox(0,0){\strut{}$16$}}%
    }%
    \gplgaddtomacro\gplfronttext{%
      \csname LTb\endcsname%
      \put(224,2799){\rotatebox{-270}{\makebox(0,0){\strut{}ave. residual norm over time \eqref{eq:Err_IntVSNonIntUnitNorm}}}}%
      \put(4117,196){\makebox(0,0){\strut{}basis dimension}}%
      \csname LTb\endcsname%
      \put(5876,4430){\makebox(0,0)[r]{\strut{}intrusive model reduction}}%
      \csname LTb\endcsname%
      \put(5876,4010){\makebox(0,0)[r]{\strut{}OpInf}}%
    }%
    \gplbacktext
    \put(0,0){\includegraphics{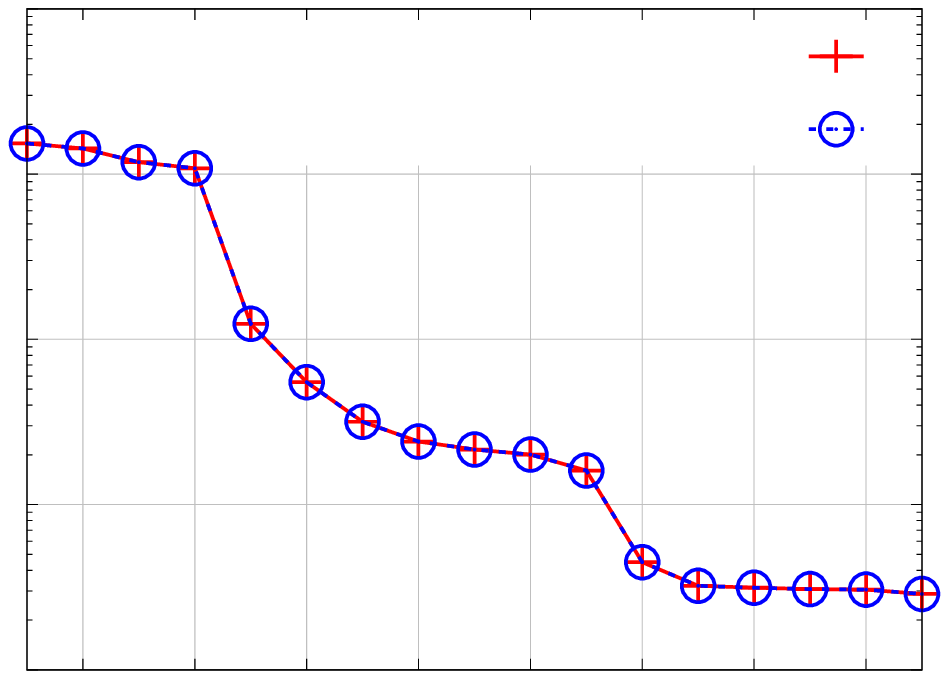}}%
    \gplfronttext
  \end{picture}%
\endgroup

%% file: convDiffAnormHard_RelAveErr.tex
\begingroup
  \makeatletter
  \providecommand\color[2][]{%
    \GenericError{(gnuplot) \space\space\space\@spaces}{%
      Package color not loaded in conjunction with
      terminal option `colourtext'%
    }{See the gnuplot documentation for explanation.%
    }{Either use 'blacktext' in gnuplot or load the package
      color.sty in LaTeX.}%
    \renewcommand\color[2][]{}%
  }%
  \providecommand\includegraphics[2][]{%
    \GenericError{(gnuplot) \space\space\space\@spaces}{%
      Package graphicx or graphics not loaded%
    }{See the gnuplot documentation for explanation.%
    }{The gnuplot epslatex terminal needs graphicx.sty or graphics.sty.}%
    \renewcommand\includegraphics[2][]{}%
  }%
  \providecommand\rotatebox[2]{#2}%
  \@ifundefined{ifGPcolor}{%
    \newif\ifGPcolor
    \GPcolortrue
  }{}%
  \@ifundefined{ifGPblacktext}{%
    \newif\ifGPblacktext
    \GPblacktexttrue
  }{}%
  \let\gplgaddtomacro\g@addto@macro
  \gdef\gplbacktext{}%
  \gdef\gplfronttext{}%
  \makeatother
  \ifGPblacktext
    \def\colorrgb#1{}%
    \def\colorgray#1{}%
  \else
    \ifGPcolor
      \def\colorrgb#1{\color[rgb]{#1}}%
      \def\colorgray#1{\color[gray]{#1}}%
      \expandafter\def\csname LTw\endcsname{\color{white}}%
      \expandafter\def\csname LTb\endcsname{\color{black}}%
      \expandafter\def\csname LTa\endcsname{\color{black}}%
      \expandafter\def\csname LT0\endcsname{\color[rgb]{1,0,0}}%
      \expandafter\def\csname LT1\endcsname{\color[rgb]{0,1,0}}%
      \expandafter\def\csname LT2\endcsname{\color[rgb]{0,0,1}}%
      \expandafter\def\csname LT3\endcsname{\color[rgb]{1,0,1}}%
      \expandafter\def\csname LT4\endcsname{\color[rgb]{0,1,1}}%
      \expandafter\def\csname LT5\endcsname{\color[rgb]{1,1,0}}%
      \expandafter\def\csname LT6\endcsname{\color[rgb]{0,0,0}}%
      \expandafter\def\csname LT7\endcsname{\color[rgb]{1,0.3,0}}%
      \expandafter\def\csname LT8\endcsname{\color[rgb]{0.5,0.5,0.5}}%
    \else
      \def\colorrgb#1{\color{black}}%
      \def\colorgray#1{\color[gray]{#1}}%
      \expandafter\def\csname LTw\endcsname{\color{white}}%
      \expandafter\def\csname LTb\endcsname{\color{black}}%
      \expandafter\def\csname LTa\endcsname{\color{black}}%
      \expandafter\def\csname LT0\endcsname{\color{black}}%
      \expandafter\def\csname LT1\endcsname{\color{black}}%
      \expandafter\def\csname LT2\endcsname{\color{black}}%
      \expandafter\def\csname LT3\endcsname{\color{black}}%
      \expandafter\def\csname LT4\endcsname{\color{black}}%
      \expandafter\def\csname LT5\endcsname{\color{black}}%
      \expandafter\def\csname LT6\endcsname{\color{black}}%
      \expandafter\def\csname LT7\endcsname{\color{black}}%
      \expandafter\def\csname LT8\endcsname{\color{black}}%
    \fi
  \fi
    \setlength{\unitlength}{0.0500bp}%
    \ifx\gptboxheight\undefined%
      \newlength{\gptboxheight}%
      \newlength{\gptboxwidth}%
      \newsavebox{\gptboxtext}%
    \fi%
    \setlength{\fboxrule}{0.5pt}%
    \setlength{\fboxsep}{1pt}%
\begin{picture}(7200.00,5040.00)%
    \gplgaddtomacro\gplbacktext{%
      \csname LTb\endcsname%
      \put(1372,896){\makebox(0,0)[r]{\strut{}1e-09}}%
      \csname LTb\endcsname%
      \put(1372,1530){\makebox(0,0)[r]{\strut{}1e-08}}%
      \csname LTb\endcsname%
      \put(1372,2165){\makebox(0,0)[r]{\strut{}1e-07}}%
      \csname LTb\endcsname%
      \put(1372,2799){\makebox(0,0)[r]{\strut{}1e-06}}%
      \csname LTb\endcsname%
      \put(1372,3434){\makebox(0,0)[r]{\strut{}1e-05}}%
      \csname LTb\endcsname%
      \put(1372,4069){\makebox(0,0)[r]{\strut{}1e-04}}%
      \csname LTb\endcsname%
      \put(1372,4703){\makebox(0,0)[r]{\strut{}1e-03}}%
      \csname LTb\endcsname%
      \put(1862,616){\makebox(0,0){\strut{}$2$}}%
      \csname LTb\endcsname%
      \put(2507,616){\makebox(0,0){\strut{}$4$}}%
      \csname LTb\endcsname%
      \put(3151,616){\makebox(0,0){\strut{}$6$}}%
      \csname LTb\endcsname%
      \put(3795,616){\makebox(0,0){\strut{}$8$}}%
      \csname LTb\endcsname%
      \put(4440,616){\makebox(0,0){\strut{}$10$}}%
      \csname LTb\endcsname%
      \put(5084,616){\makebox(0,0){\strut{}$12$}}%
      \csname LTb\endcsname%
      \put(5728,616){\makebox(0,0){\strut{}$14$}}%
      \csname LTb\endcsname%
      \put(6373,616){\makebox(0,0){\strut{}$16$}}%
    }%
    \gplgaddtomacro\gplfronttext{%
      \csname LTb\endcsname%
      \put(224,2799){\rotatebox{-270}{\makebox(0,0){\strut{}rel. ave. state err. over time}}}%
      \put(4117,196){\makebox(0,0){\strut{}basis dimension}}%
      \csname LTb\endcsname%
      \put(2359,1659){\makebox(0,0)[l]{\strut{}OpInf error \eqref{eq:Err_AveRelErrActualErr}}}%
      \csname LTb\endcsname%
      \put(2359,1379){\makebox(0,0)[l]{\strut{}intrusive err. est. \eqref{eq:Err_AveRelErrActualBnd}}}%
      \csname LTb\endcsname%
      \put(2359,1099){\makebox(0,0)[l]{\strut{}learned err. est. \eqref{eq:Err_AveRelErrProbBnd}}}%
    }%
    \gplbacktext
    \put(0,0){\includegraphics{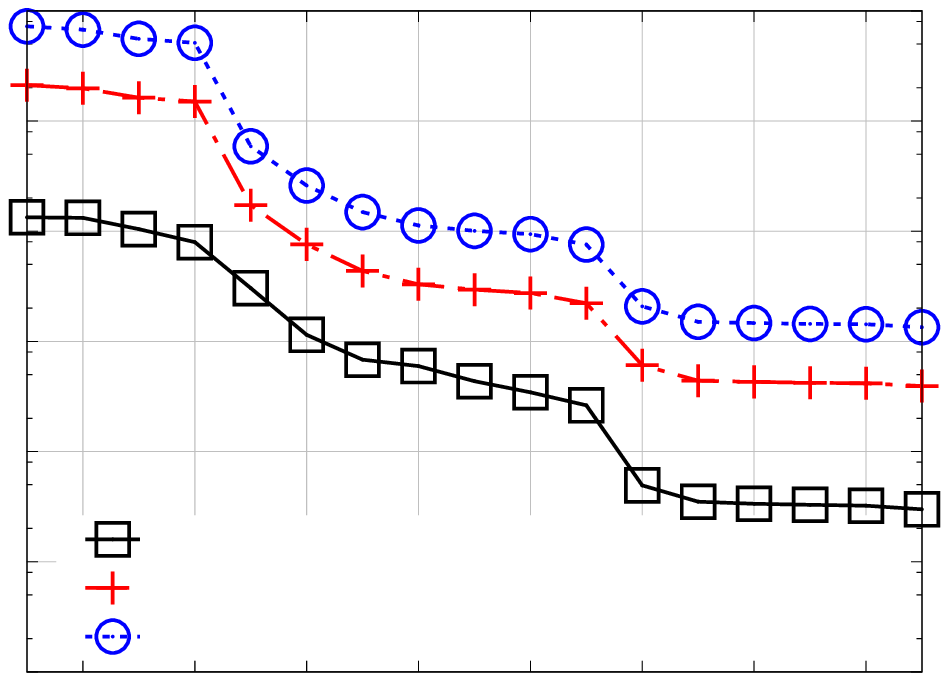}}%
    \gplfronttext
  \end{picture}%
\endgroup

%% file: convDiffAnormHard_Output1Basis17.tex
\begingroup
  \makeatletter
  \providecommand\color[2][]{%
    \GenericError{(gnuplot) \space\space\space\@spaces}{%
      Package color not loaded in conjunction with
      terminal option `colourtext'%
    }{See the gnuplot documentation for explanation.%
    }{Either use 'blacktext' in gnuplot or load the package
      color.sty in LaTeX.}%
    \renewcommand\color[2][]{}%
  }%
  \providecommand\includegraphics[2][]{%
    \GenericError{(gnuplot) \space\space\space\@spaces}{%
      Package graphicx or graphics not loaded%
    }{See the gnuplot documentation for explanation.%
    }{The gnuplot epslatex terminal needs graphicx.sty or graphics.sty.}%
    \renewcommand\includegraphics[2][]{}%
  }%
  \providecommand\rotatebox[2]{#2}%
  \@ifundefined{ifGPcolor}{%
    \newif\ifGPcolor
    \GPcolortrue
  }{}%
  \@ifundefined{ifGPblacktext}{%
    \newif\ifGPblacktext
    \GPblacktexttrue
  }{}%
  \let\gplgaddtomacro\g@addto@macro
  \gdef\gplbacktext{}%
  \gdef\gplfronttext{}%
  \makeatother
  \ifGPblacktext
    \def\colorrgb#1{}%
    \def\colorgray#1{}%
  \else
    \ifGPcolor
      \def\colorrgb#1{\color[rgb]{#1}}%
      \def\colorgray#1{\color[gray]{#1}}%
      \expandafter\def\csname LTw\endcsname{\color{white}}%
      \expandafter\def\csname LTb\endcsname{\color{black}}%
      \expandafter\def\csname LTa\endcsname{\color{black}}%
      \expandafter\def\csname LT0\endcsname{\color[rgb]{1,0,0}}%
      \expandafter\def\csname LT1\endcsname{\color[rgb]{0,1,0}}%
      \expandafter\def\csname LT2\endcsname{\color[rgb]{0,0,1}}%
      \expandafter\def\csname LT3\endcsname{\color[rgb]{1,0,1}}%
      \expandafter\def\csname LT4\endcsname{\color[rgb]{0,1,1}}%
      \expandafter\def\csname LT5\endcsname{\color[rgb]{1,1,0}}%
      \expandafter\def\csname LT6\endcsname{\color[rgb]{0,0,0}}%
      \expandafter\def\csname LT7\endcsname{\color[rgb]{1,0.3,0}}%
      \expandafter\def\csname LT8\endcsname{\color[rgb]{0.5,0.5,0.5}}%
    \else
      \def\colorrgb#1{\color{black}}%
      \def\colorgray#1{\color[gray]{#1}}%
      \expandafter\def\csname LTw\endcsname{\color{white}}%
      \expandafter\def\csname LTb\endcsname{\color{black}}%
      \expandafter\def\csname LTa\endcsname{\color{black}}%
      \expandafter\def\csname LT0\endcsname{\color{black}}%
      \expandafter\def\csname LT1\endcsname{\color{black}}%
      \expandafter\def\csname LT2\endcsname{\color{black}}%
      \expandafter\def\csname LT3\endcsname{\color{black}}%
      \expandafter\def\csname LT4\endcsname{\color{black}}%
      \expandafter\def\csname LT5\endcsname{\color{black}}%
      \expandafter\def\csname LT6\endcsname{\color{black}}%
      \expandafter\def\csname LT7\endcsname{\color{black}}%
      \expandafter\def\csname LT8\endcsname{\color{black}}%
    \fi
  \fi
    \setlength{\unitlength}{0.0500bp}%
    \ifx\gptboxheight\undefined%
      \newlength{\gptboxheight}%
      \newlength{\gptboxwidth}%
      \newsavebox{\gptboxtext}%
    \fi%
    \setlength{\fboxrule}{0.5pt}%
    \setlength{\fboxsep}{1pt}%
\begin{picture}(7200.00,5040.00)%
    \gplgaddtomacro\gplbacktext{%
      \csname LTb\endcsname%
      \put(1540,896){\makebox(0,0)[r]{\strut{}$-0.005$}}%
      \csname LTb\endcsname%
      \put(1540,1440){\makebox(0,0)[r]{\strut{}$0$}}%
      \csname LTb\endcsname%
      \put(1540,1984){\makebox(0,0)[r]{\strut{}$0.005$}}%
      \csname LTb\endcsname%
      \put(1540,2528){\makebox(0,0)[r]{\strut{}$0.01$}}%
      \csname LTb\endcsname%
      \put(1540,3071){\makebox(0,0)[r]{\strut{}$0.015$}}%
      \csname LTb\endcsname%
      \put(1540,3615){\makebox(0,0)[r]{\strut{}$0.02$}}%
      \csname LTb\endcsname%
      \put(1540,4159){\makebox(0,0)[r]{\strut{}$0.025$}}%
      \csname LTb\endcsname%
      \put(1540,4703){\makebox(0,0)[r]{\strut{}$0.03$}}%
      \csname LTb\endcsname%
      \put(1708,616){\makebox(0,0){\strut{}$0$}}%
      \csname LTb\endcsname%
      \put(2705,616){\makebox(0,0){\strut{}$0.1$}}%
      \csname LTb\endcsname%
      \put(3703,616){\makebox(0,0){\strut{}$0.2$}}%
      \csname LTb\endcsname%
      \put(4700,616){\makebox(0,0){\strut{}$0.3$}}%
      \csname LTb\endcsname%
      \put(5698,616){\makebox(0,0){\strut{}$0.4$}}%
      \csname LTb\endcsname%
      \put(6695,616){\makebox(0,0){\strut{}$0.5$}}%
    }%
    \gplgaddtomacro\gplfronttext{%
      \csname LTb\endcsname%
      \put(224,2799){\rotatebox{-270}{\makebox(0,0){\strut{}output and error bounds \eqref{eq:outputBounds}}}}%
      \put(4201,196){\makebox(0,0){\strut{}time}}%
      \csname LTb\endcsname%
      \put(5876,1799){\makebox(0,0)[r]{\strut{}output \eqref{eq:OutputAve}}}%
      \csname LTb\endcsname%
      \put(5876,1463){\makebox(0,0)[r]{\strut{}$\ROMOutput_k - \OutErrBnd_k$}}%
      \csname LTb\endcsname%
      \put(5876,1127){\makebox(0,0)[r]{\strut{}$\ROMOutput_k + \OutErrBnd_k$}}%
    }%
    \gplbacktext
    \put(0,0){\includegraphics{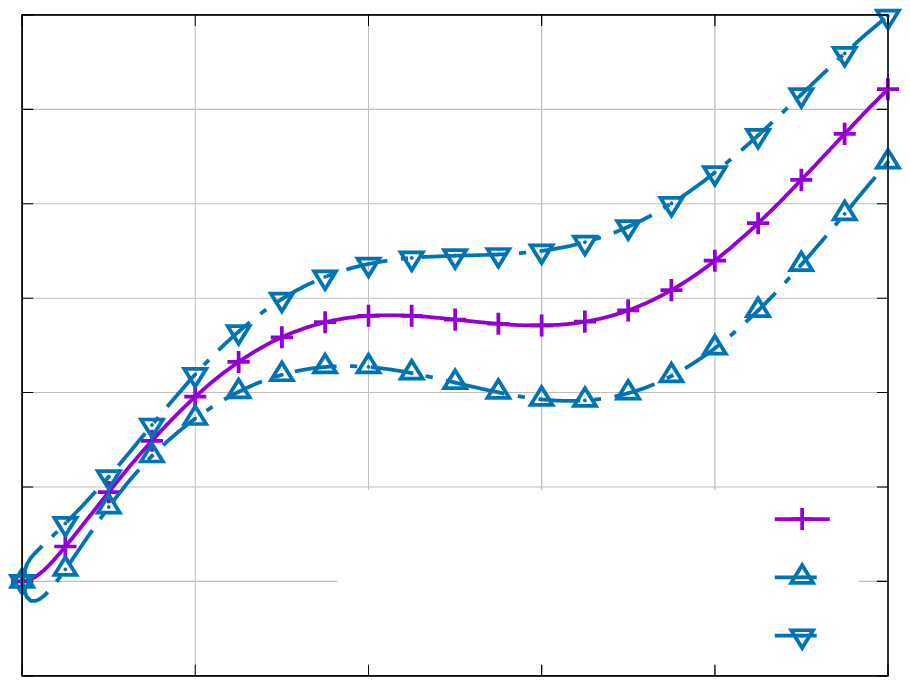}}%
    \gplfronttext
  \end{picture}%
\endgroup

%% file: convDiffAnormHard_Output1Basis12.tex
\begingroup
  \makeatletter
  \providecommand\color[2][]{%
    \GenericError{(gnuplot) \space\space\space\@spaces}{%
      Package color not loaded in conjunction with
      terminal option `colourtext'%
    }{See the gnuplot documentation for explanation.%
    }{Either use 'blacktext' in gnuplot or load the package
      color.sty in LaTeX.}%
    \renewcommand\color[2][]{}%
  }%
  \providecommand\includegraphics[2][]{%
    \GenericError{(gnuplot) \space\space\space\@spaces}{%
      Package graphicx or graphics not loaded%
    }{See the gnuplot documentation for explanation.%
    }{The gnuplot epslatex terminal needs graphicx.sty or graphics.sty.}%
    \renewcommand\includegraphics[2][]{}%
  }%
  \providecommand\rotatebox[2]{#2}%
  \@ifundefined{ifGPcolor}{%
    \newif\ifGPcolor
    \GPcolortrue
  }{}%
  \@ifundefined{ifGPblacktext}{%
    \newif\ifGPblacktext
    \GPblacktexttrue
  }{}%
  \let\gplgaddtomacro\g@addto@macro
  \gdef\gplbacktext{}%
  \gdef\gplfronttext{}%
  \makeatother
  \ifGPblacktext
    \def\colorrgb#1{}%
    \def\colorgray#1{}%
  \else
    \ifGPcolor
      \def\colorrgb#1{\color[rgb]{#1}}%
      \def\colorgray#1{\color[gray]{#1}}%
      \expandafter\def\csname LTw\endcsname{\color{white}}%
      \expandafter\def\csname LTb\endcsname{\color{black}}%
      \expandafter\def\csname LTa\endcsname{\color{black}}%
      \expandafter\def\csname LT0\endcsname{\color[rgb]{1,0,0}}%
      \expandafter\def\csname LT1\endcsname{\color[rgb]{0,1,0}}%
      \expandafter\def\csname LT2\endcsname{\color[rgb]{0,0,1}}%
      \expandafter\def\csname LT3\endcsname{\color[rgb]{1,0,1}}%
      \expandafter\def\csname LT4\endcsname{\color[rgb]{0,1,1}}%
      \expandafter\def\csname LT5\endcsname{\color[rgb]{1,1,0}}%
      \expandafter\def\csname LT6\endcsname{\color[rgb]{0,0,0}}%
      \expandafter\def\csname LT7\endcsname{\color[rgb]{1,0.3,0}}%
      \expandafter\def\csname LT8\endcsname{\color[rgb]{0.5,0.5,0.5}}%
    \else
      \def\colorrgb#1{\color{black}}%
      \def\colorgray#1{\color[gray]{#1}}%
      \expandafter\def\csname LTw\endcsname{\color{white}}%
      \expandafter\def\csname LTb\endcsname{\color{black}}%
      \expandafter\def\csname LTa\endcsname{\color{black}}%
      \expandafter\def\csname LT0\endcsname{\color{black}}%
      \expandafter\def\csname LT1\endcsname{\color{black}}%
      \expandafter\def\csname LT2\endcsname{\color{black}}%
      \expandafter\def\csname LT3\endcsname{\color{black}}%
      \expandafter\def\csname LT4\endcsname{\color{black}}%
      \expandafter\def\csname LT5\endcsname{\color{black}}%
      \expandafter\def\csname LT6\endcsname{\color{black}}%
      \expandafter\def\csname LT7\endcsname{\color{black}}%
      \expandafter\def\csname LT8\endcsname{\color{black}}%
    \fi
  \fi
    \setlength{\unitlength}{0.0500bp}%
    \ifx\gptboxheight\undefined%
      \newlength{\gptboxheight}%
      \newlength{\gptboxwidth}%
      \newsavebox{\gptboxtext}%
    \fi%
    \setlength{\fboxrule}{0.5pt}%
    \setlength{\fboxsep}{1pt}%
\begin{picture}(7200.00,5040.00)%
    \gplgaddtomacro\gplbacktext{%
      \csname LTb\endcsname%
      \put(1260,896){\makebox(0,0)[r]{\strut{}$-0.005$}}%
      \csname LTb\endcsname%
      \put(1260,1372){\makebox(0,0)[r]{\strut{}$0$}}%
      \csname LTb\endcsname%
      \put(1260,1848){\makebox(0,0)[r]{\strut{}$0.005$}}%
      \csname LTb\endcsname%
      \put(1260,2324){\makebox(0,0)[r]{\strut{}$0.01$}}%
      \csname LTb\endcsname%
      \put(1260,2800){\makebox(0,0)[r]{\strut{}$0.015$}}%
      \csname LTb\endcsname%
      \put(1260,3275){\makebox(0,0)[r]{\strut{}$0.02$}}%
      \csname LTb\endcsname%
      \put(1260,3751){\makebox(0,0)[r]{\strut{}$0.025$}}%
      \csname LTb\endcsname%
      \put(1260,4227){\makebox(0,0)[r]{\strut{}$0.03$}}%
      \csname LTb\endcsname%
      \put(1260,4703){\makebox(0,0)[r]{\strut{}$0.035$}}%
      \csname LTb\endcsname%
      \put(1428,616){\makebox(0,0){\strut{}$0$}}%
      \csname LTb\endcsname%
      \put(2481,616){\makebox(0,0){\strut{}$0.1$}}%
      \csname LTb\endcsname%
      \put(3535,616){\makebox(0,0){\strut{}$0.2$}}%
      \csname LTb\endcsname%
      \put(4588,616){\makebox(0,0){\strut{}$0.3$}}%
      \csname LTb\endcsname%
      \put(5642,616){\makebox(0,0){\strut{}$0.4$}}%
      \csname LTb\endcsname%
      \put(6695,616){\makebox(0,0){\strut{}$0.5$}}%
    }%
    \gplgaddtomacro\gplfronttext{%
      \csname LTb\endcsname%
      \put(4061,196){\makebox(0,0){\strut{}time}}%
      \csname LTb\endcsname%
      \put(5876,1799){\makebox(0,0)[r]{\strut{}output \eqref{eq:OutputAve}}}%
      \csname LTb\endcsname%
      \put(5876,1463){\makebox(0,0)[r]{\strut{}$\ROMOutput_k - \OutErrBnd_k$}}%
      \csname LTb\endcsname%
      \put(5876,1127){\makebox(0,0)[r]{\strut{}$\ROMOutput_k + \OutErrBnd_k$}}%
    }%
    \gplbacktext
    \put(0,0){\includegraphics{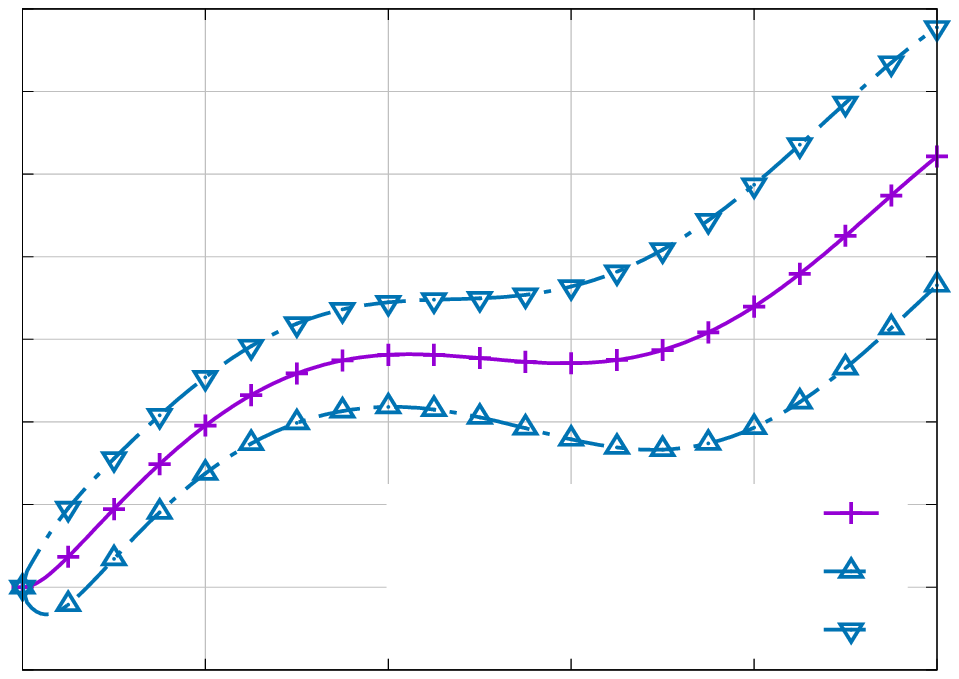}}%
    \gplfronttext
  \end{picture}%
\endgroup

%% file: convDiffAnormHard_Output1Basis7.tex
\begingroup
  \makeatletter
  \providecommand\color[2][]{%
    \GenericError{(gnuplot) \space\space\space\@spaces}{%
      Package color not loaded in conjunction with
      terminal option `colourtext'%
    }{See the gnuplot documentation for explanation.%
    }{Either use 'blacktext' in gnuplot or load the package
      color.sty in LaTeX.}%
    \renewcommand\color[2][]{}%
  }%
  \providecommand\includegraphics[2][]{%
    \GenericError{(gnuplot) \space\space\space\@spaces}{%
      Package graphicx or graphics not loaded%
    }{See the gnuplot documentation for explanation.%
    }{The gnuplot epslatex terminal needs graphicx.sty or graphics.sty.}%
    \renewcommand\includegraphics[2][]{}%
  }%
  \providecommand\rotatebox[2]{#2}%
  \@ifundefined{ifGPcolor}{%
    \newif\ifGPcolor
    \GPcolortrue
  }{}%
  \@ifundefined{ifGPblacktext}{%
    \newif\ifGPblacktext
    \GPblacktexttrue
  }{}%
  \let\gplgaddtomacro\g@addto@macro
  \gdef\gplbacktext{}%
  \gdef\gplfronttext{}%
  \makeatother
  \ifGPblacktext
    \def\colorrgb#1{}%
    \def\colorgray#1{}%
  \else
    \ifGPcolor
      \def\colorrgb#1{\color[rgb]{#1}}%
      \def\colorgray#1{\color[gray]{#1}}%
      \expandafter\def\csname LTw\endcsname{\color{white}}%
      \expandafter\def\csname LTb\endcsname{\color{black}}%
      \expandafter\def\csname LTa\endcsname{\color{black}}%
      \expandafter\def\csname LT0\endcsname{\color[rgb]{1,0,0}}%
      \expandafter\def\csname LT1\endcsname{\color[rgb]{0,1,0}}%
      \expandafter\def\csname LT2\endcsname{\color[rgb]{0,0,1}}%
      \expandafter\def\csname LT3\endcsname{\color[rgb]{1,0,1}}%
      \expandafter\def\csname LT4\endcsname{\color[rgb]{0,1,1}}%
      \expandafter\def\csname LT5\endcsname{\color[rgb]{1,1,0}}%
      \expandafter\def\csname LT6\endcsname{\color[rgb]{0,0,0}}%
      \expandafter\def\csname LT7\endcsname{\color[rgb]{1,0.3,0}}%
      \expandafter\def\csname LT8\endcsname{\color[rgb]{0.5,0.5,0.5}}%
    \else
      \def\colorrgb#1{\color{black}}%
      \def\colorgray#1{\color[gray]{#1}}%
      \expandafter\def\csname LTw\endcsname{\color{white}}%
      \expandafter\def\csname LTb\endcsname{\color{black}}%
      \expandafter\def\csname LTa\endcsname{\color{black}}%
      \expandafter\def\csname LT0\endcsname{\color{black}}%
      \expandafter\def\csname LT1\endcsname{\color{black}}%
      \expandafter\def\csname LT2\endcsname{\color{black}}%
      \expandafter\def\csname LT3\endcsname{\color{black}}%
      \expandafter\def\csname LT4\endcsname{\color{black}}%
      \expandafter\def\csname LT5\endcsname{\color{black}}%
      \expandafter\def\csname LT6\endcsname{\color{black}}%
      \expandafter\def\csname LT7\endcsname{\color{black}}%
      \expandafter\def\csname LT8\endcsname{\color{black}}%
    \fi
  \fi
    \setlength{\unitlength}{0.0500bp}%
    \ifx\gptboxheight\undefined%
      \newlength{\gptboxheight}%
      \newlength{\gptboxwidth}%
      \newsavebox{\gptboxtext}%
    \fi%
    \setlength{\fboxrule}{0.5pt}%
    \setlength{\fboxsep}{1pt}%
\begin{picture}(7200.00,5040.00)%
    \gplgaddtomacro\gplbacktext{%
      \csname LTb\endcsname%
      \put(1092,1189){\makebox(0,0)[r]{\strut{}$-0.02$}}%
      \csname LTb\endcsname%
      \put(1092,1775){\makebox(0,0)[r]{\strut{}$0$}}%
      \csname LTb\endcsname%
      \put(1092,2360){\makebox(0,0)[r]{\strut{}$0.02$}}%
      \csname LTb\endcsname%
      \put(1092,2946){\makebox(0,0)[r]{\strut{}$0.04$}}%
      \csname LTb\endcsname%
      \put(1092,3532){\makebox(0,0)[r]{\strut{}$0.06$}}%
      \csname LTb\endcsname%
      \put(1092,4117){\makebox(0,0)[r]{\strut{}$0.08$}}%
      \csname LTb\endcsname%
      \put(1092,4703){\makebox(0,0)[r]{\strut{}$0.1$}}%
      \csname LTb\endcsname%
      \put(1260,616){\makebox(0,0){\strut{}$0$}}%
      \csname LTb\endcsname%
      \put(2347,616){\makebox(0,0){\strut{}$0.1$}}%
      \csname LTb\endcsname%
      \put(3434,616){\makebox(0,0){\strut{}$0.2$}}%
      \csname LTb\endcsname%
      \put(4521,616){\makebox(0,0){\strut{}$0.3$}}%
      \csname LTb\endcsname%
      \put(5608,616){\makebox(0,0){\strut{}$0.4$}}%
      \csname LTb\endcsname%
      \put(6695,616){\makebox(0,0){\strut{}$0.5$}}%
    }%
    \gplgaddtomacro\gplfronttext{%
      \csname LTb\endcsname%
      \put(3977,196){\makebox(0,0){\strut{}time}}%
      \csname LTb\endcsname%
      \put(2079,4472){\makebox(0,0)[l]{\strut{}output \eqref{eq:OutputAve}}}%
      \csname LTb\endcsname%
      \put(2079,4136){\makebox(0,0)[l]{\strut{}$\ROMOutput_k - \OutErrBnd_k$}}%
      \csname LTb\endcsname%
      \put(2079,3800){\makebox(0,0)[l]{\strut{}$\ROMOutput_k + \OutErrBnd_k$}}%
    }%
    \gplbacktext
    \put(0,0){\includegraphics{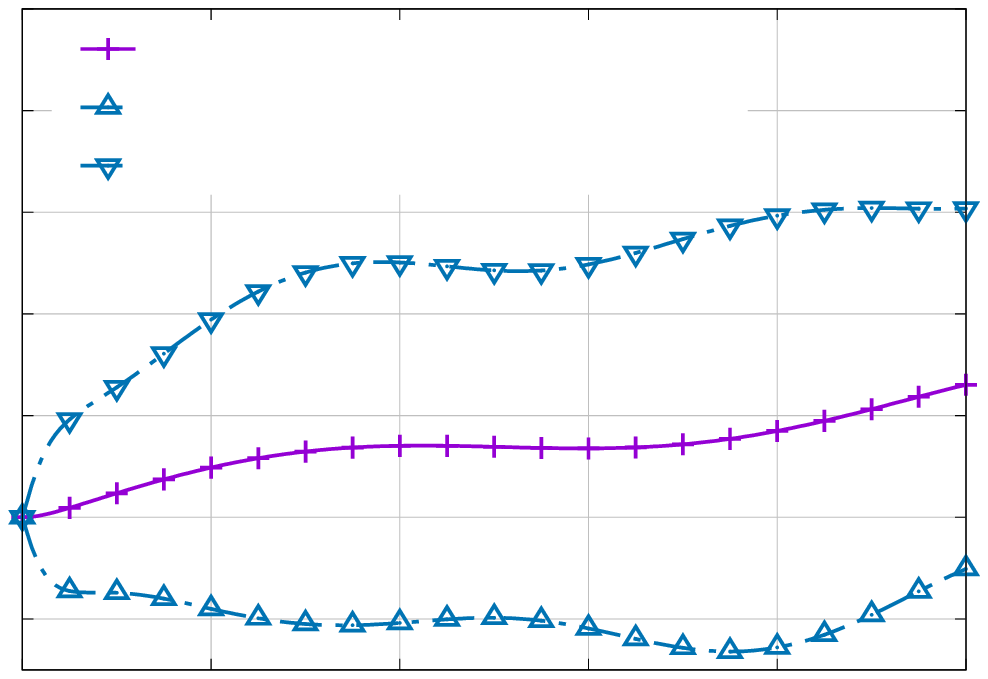}}%
    \gplfronttext
  \end{picture}%
\endgroup

%% file: convDiffAnormHard_Output2Basis15.tex
\begingroup
  \makeatletter
  \providecommand\color[2][]{%
    \GenericError{(gnuplot) \space\space\space\@spaces}{%
      Package color not loaded in conjunction with
      terminal option `colourtext'%
    }{See the gnuplot documentation for explanation.%
    }{Either use 'blacktext' in gnuplot or load the package
      color.sty in LaTeX.}%
    \renewcommand\color[2][]{}%
  }%
  \providecommand\includegraphics[2][]{%
    \GenericError{(gnuplot) \space\space\space\@spaces}{%
      Package graphicx or graphics not loaded%
    }{See the gnuplot documentation for explanation.%
    }{The gnuplot epslatex terminal needs graphicx.sty or graphics.sty.}%
    \renewcommand\includegraphics[2][]{}%
  }%
  \providecommand\rotatebox[2]{#2}%
  \@ifundefined{ifGPcolor}{%
    \newif\ifGPcolor
    \GPcolortrue
  }{}%
  \@ifundefined{ifGPblacktext}{%
    \newif\ifGPblacktext
    \GPblacktexttrue
  }{}%
  \let\gplgaddtomacro\g@addto@macro
  \gdef\gplbacktext{}%
  \gdef\gplfronttext{}%
  \makeatother
  \ifGPblacktext
    \def\colorrgb#1{}%
    \def\colorgray#1{}%
  \else
    \ifGPcolor
      \def\colorrgb#1{\color[rgb]{#1}}%
      \def\colorgray#1{\color[gray]{#1}}%
      \expandafter\def\csname LTw\endcsname{\color{white}}%
      \expandafter\def\csname LTb\endcsname{\color{black}}%
      \expandafter\def\csname LTa\endcsname{\color{black}}%
      \expandafter\def\csname LT0\endcsname{\color[rgb]{1,0,0}}%
      \expandafter\def\csname LT1\endcsname{\color[rgb]{0,1,0}}%
      \expandafter\def\csname LT2\endcsname{\color[rgb]{0,0,1}}%
      \expandafter\def\csname LT3\endcsname{\color[rgb]{1,0,1}}%
      \expandafter\def\csname LT4\endcsname{\color[rgb]{0,1,1}}%
      \expandafter\def\csname LT5\endcsname{\color[rgb]{1,1,0}}%
      \expandafter\def\csname LT6\endcsname{\color[rgb]{0,0,0}}%
      \expandafter\def\csname LT7\endcsname{\color[rgb]{1,0.3,0}}%
      \expandafter\def\csname LT8\endcsname{\color[rgb]{0.5,0.5,0.5}}%
    \else
      \def\colorrgb#1{\color{black}}%
      \def\colorgray#1{\color[gray]{#1}}%
      \expandafter\def\csname LTw\endcsname{\color{white}}%
      \expandafter\def\csname LTb\endcsname{\color{black}}%
      \expandafter\def\csname LTa\endcsname{\color{black}}%
      \expandafter\def\csname LT0\endcsname{\color{black}}%
      \expandafter\def\csname LT1\endcsname{\color{black}}%
      \expandafter\def\csname LT2\endcsname{\color{black}}%
      \expandafter\def\csname LT3\endcsname{\color{black}}%
      \expandafter\def\csname LT4\endcsname{\color{black}}%
      \expandafter\def\csname LT5\endcsname{\color{black}}%
      \expandafter\def\csname LT6\endcsname{\color{black}}%
      \expandafter\def\csname LT7\endcsname{\color{black}}%
      \expandafter\def\csname LT8\endcsname{\color{black}}%
    \fi
  \fi
    \setlength{\unitlength}{0.0500bp}%
    \ifx\gptboxheight\undefined%
      \newlength{\gptboxheight}%
      \newlength{\gptboxwidth}%
      \newsavebox{\gptboxtext}%
    \fi%
    \setlength{\fboxrule}{0.5pt}%
    \setlength{\fboxsep}{1pt}%
\begin{picture}(7200.00,5040.00)%
    \gplgaddtomacro\gplbacktext{%
      \csname LTb\endcsname%
      \put(1540,896){\makebox(0,0)[r]{\strut{}$-0.025$}}%
      \csname LTb\endcsname%
      \put(1540,1319){\makebox(0,0)[r]{\strut{}$-0.02$}}%
      \csname LTb\endcsname%
      \put(1540,1742){\makebox(0,0)[r]{\strut{}$-0.015$}}%
      \csname LTb\endcsname%
      \put(1540,2165){\makebox(0,0)[r]{\strut{}$-0.01$}}%
      \csname LTb\endcsname%
      \put(1540,2588){\makebox(0,0)[r]{\strut{}$-0.005$}}%
      \csname LTb\endcsname%
      \put(1540,3011){\makebox(0,0)[r]{\strut{}$0$}}%
      \csname LTb\endcsname%
      \put(1540,3434){\makebox(0,0)[r]{\strut{}$0.005$}}%
      \csname LTb\endcsname%
      \put(1540,3857){\makebox(0,0)[r]{\strut{}$0.01$}}%
      \csname LTb\endcsname%
      \put(1540,4280){\makebox(0,0)[r]{\strut{}$0.015$}}%
      \csname LTb\endcsname%
      \put(1540,4703){\makebox(0,0)[r]{\strut{}$0.02$}}%
      \csname LTb\endcsname%
      \put(1708,616){\makebox(0,0){\strut{}$0$}}%
      \csname LTb\endcsname%
      \put(2705,616){\makebox(0,0){\strut{}$0.1$}}%
      \csname LTb\endcsname%
      \put(3703,616){\makebox(0,0){\strut{}$0.2$}}%
      \csname LTb\endcsname%
      \put(4700,616){\makebox(0,0){\strut{}$0.3$}}%
      \csname LTb\endcsname%
      \put(5698,616){\makebox(0,0){\strut{}$0.4$}}%
      \csname LTb\endcsname%
      \put(6695,616){\makebox(0,0){\strut{}$0.5$}}%
    }%
    \gplgaddtomacro\gplfronttext{%
      \csname LTb\endcsname%
      \put(224,2799){\rotatebox{-270}{\makebox(0,0){\strut{}output and error bounds \eqref{eq:outputBounds}}}}%
      \put(4201,196){\makebox(0,0){\strut{}time}}%
      \csname LTb\endcsname%
      \put(5876,4458){\makebox(0,0)[r]{\strut{}output \eqref{eq:OutputIntegral}}}%
      \csname LTb\endcsname%
      \put(5876,4094){\makebox(0,0)[r]{\strut{}$\ROMOutput_k - \OutErrBnd_k$}}%
      \csname LTb\endcsname%
      \put(5876,3730){\makebox(0,0)[r]{\strut{}$\ROMOutput_k + \OutErrBnd_k$}}%
    }%
    \gplbacktext
    \put(0,0){\includegraphics{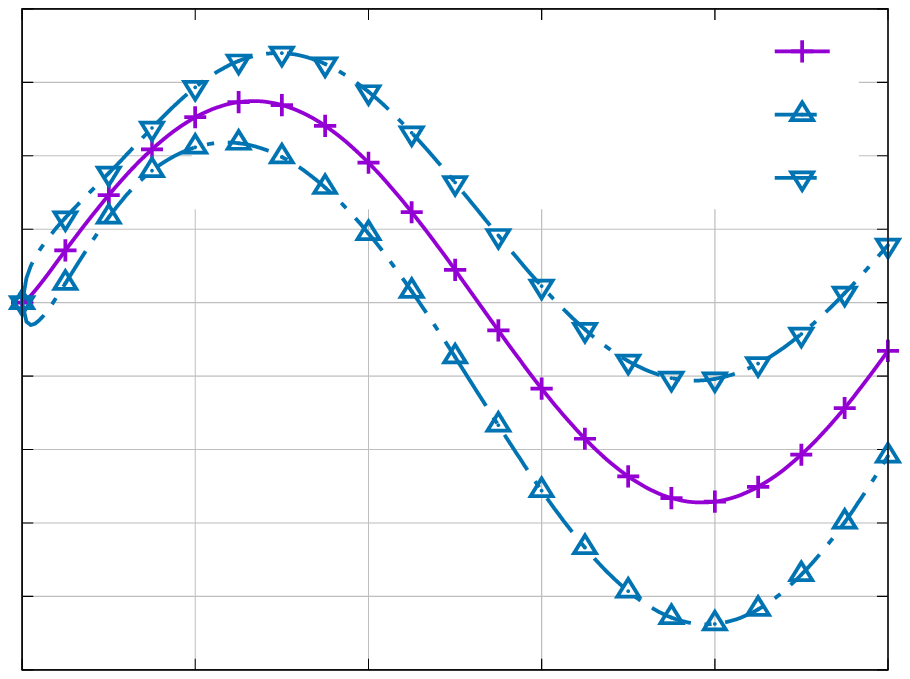}}%
    \gplfronttext
  \end{picture}%
\endgroup

%% file: convDiffAnormHard_Output2Basis10.tex
\begingroup
  \makeatletter
  \providecommand\color[2][]{%
    \GenericError{(gnuplot) \space\space\space\@spaces}{%
      Package color not loaded in conjunction with
      terminal option `colourtext'%
    }{See the gnuplot documentation for explanation.%
    }{Either use 'blacktext' in gnuplot or load the package
      color.sty in LaTeX.}%
    \renewcommand\color[2][]{}%
  }%
  \providecommand\includegraphics[2][]{%
    \GenericError{(gnuplot) \space\space\space\@spaces}{%
      Package graphicx or graphics not loaded%
    }{See the gnuplot documentation for explanation.%
    }{The gnuplot epslatex terminal needs graphicx.sty or graphics.sty.}%
    \renewcommand\includegraphics[2][]{}%
  }%
  \providecommand\rotatebox[2]{#2}%
  \@ifundefined{ifGPcolor}{%
    \newif\ifGPcolor
    \GPcolortrue
  }{}%
  \@ifundefined{ifGPblacktext}{%
    \newif\ifGPblacktext
    \GPblacktexttrue
  }{}%
  \let\gplgaddtomacro\g@addto@macro
  \gdef\gplbacktext{}%
  \gdef\gplfronttext{}%
  \makeatother
  \ifGPblacktext
    \def\colorrgb#1{}%
    \def\colorgray#1{}%
  \else
    \ifGPcolor
      \def\colorrgb#1{\color[rgb]{#1}}%
      \def\colorgray#1{\color[gray]{#1}}%
      \expandafter\def\csname LTw\endcsname{\color{white}}%
      \expandafter\def\csname LTb\endcsname{\color{black}}%
      \expandafter\def\csname LTa\endcsname{\color{black}}%
      \expandafter\def\csname LT0\endcsname{\color[rgb]{1,0,0}}%
      \expandafter\def\csname LT1\endcsname{\color[rgb]{0,1,0}}%
      \expandafter\def\csname LT2\endcsname{\color[rgb]{0,0,1}}%
      \expandafter\def\csname LT3\endcsname{\color[rgb]{1,0,1}}%
      \expandafter\def\csname LT4\endcsname{\color[rgb]{0,1,1}}%
      \expandafter\def\csname LT5\endcsname{\color[rgb]{1,1,0}}%
      \expandafter\def\csname LT6\endcsname{\color[rgb]{0,0,0}}%
      \expandafter\def\csname LT7\endcsname{\color[rgb]{1,0.3,0}}%
      \expandafter\def\csname LT8\endcsname{\color[rgb]{0.5,0.5,0.5}}%
    \else
      \def\colorrgb#1{\color{black}}%
      \def\colorgray#1{\color[gray]{#1}}%
      \expandafter\def\csname LTw\endcsname{\color{white}}%
      \expandafter\def\csname LTb\endcsname{\color{black}}%
      \expandafter\def\csname LTa\endcsname{\color{black}}%
      \expandafter\def\csname LT0\endcsname{\color{black}}%
      \expandafter\def\csname LT1\endcsname{\color{black}}%
      \expandafter\def\csname LT2\endcsname{\color{black}}%
      \expandafter\def\csname LT3\endcsname{\color{black}}%
      \expandafter\def\csname LT4\endcsname{\color{black}}%
      \expandafter\def\csname LT5\endcsname{\color{black}}%
      \expandafter\def\csname LT6\endcsname{\color{black}}%
      \expandafter\def\csname LT7\endcsname{\color{black}}%
      \expandafter\def\csname LT8\endcsname{\color{black}}%
    \fi
  \fi
    \setlength{\unitlength}{0.0500bp}%
    \ifx\gptboxheight\undefined%
      \newlength{\gptboxheight}%
      \newlength{\gptboxwidth}%
      \newsavebox{\gptboxtext}%
    \fi%
    \setlength{\fboxrule}{0.5pt}%
    \setlength{\fboxsep}{1pt}%
\begin{picture}(7200.00,5040.00)%
    \gplgaddtomacro\gplbacktext{%
      \csname LTb\endcsname%
      \put(1092,896){\makebox(0,0)[r]{\strut{}$-0.06$}}%
      \csname LTb\endcsname%
      \put(1092,1440){\makebox(0,0)[r]{\strut{}$-0.04$}}%
      \csname LTb\endcsname%
      \put(1092,1984){\makebox(0,0)[r]{\strut{}$-0.02$}}%
      \csname LTb\endcsname%
      \put(1092,2528){\makebox(0,0)[r]{\strut{}$0$}}%
      \csname LTb\endcsname%
      \put(1092,3071){\makebox(0,0)[r]{\strut{}$0.02$}}%
      \csname LTb\endcsname%
      \put(1092,3615){\makebox(0,0)[r]{\strut{}$0.04$}}%
      \csname LTb\endcsname%
      \put(1092,4159){\makebox(0,0)[r]{\strut{}$0.06$}}%
      \csname LTb\endcsname%
      \put(1092,4703){\makebox(0,0)[r]{\strut{}$0.08$}}%
      \csname LTb\endcsname%
      \put(1260,616){\makebox(0,0){\strut{}$0$}}%
      \csname LTb\endcsname%
      \put(2347,616){\makebox(0,0){\strut{}$0.1$}}%
      \csname LTb\endcsname%
      \put(3434,616){\makebox(0,0){\strut{}$0.2$}}%
      \csname LTb\endcsname%
      \put(4521,616){\makebox(0,0){\strut{}$0.3$}}%
      \csname LTb\endcsname%
      \put(5608,616){\makebox(0,0){\strut{}$0.4$}}%
      \csname LTb\endcsname%
      \put(6695,616){\makebox(0,0){\strut{}$0.5$}}%
    }%
    \gplgaddtomacro\gplfronttext{%
      \csname LTb\endcsname%
      \put(3977,196){\makebox(0,0){\strut{}time}}%
      \csname LTb\endcsname%
      \put(5876,4458){\makebox(0,0)[r]{\strut{}output \eqref{eq:OutputIntegral}}}%
      \csname LTb\endcsname%
      \put(5876,4094){\makebox(0,0)[r]{\strut{}$\ROMOutput_k - \OutErrBnd_k$}}%
      \csname LTb\endcsname%
      \put(5876,3730){\makebox(0,0)[r]{\strut{}$\ROMOutput_k + \OutErrBnd_k$}}%
    }%
    \gplbacktext
    \put(0,0){\includegraphics{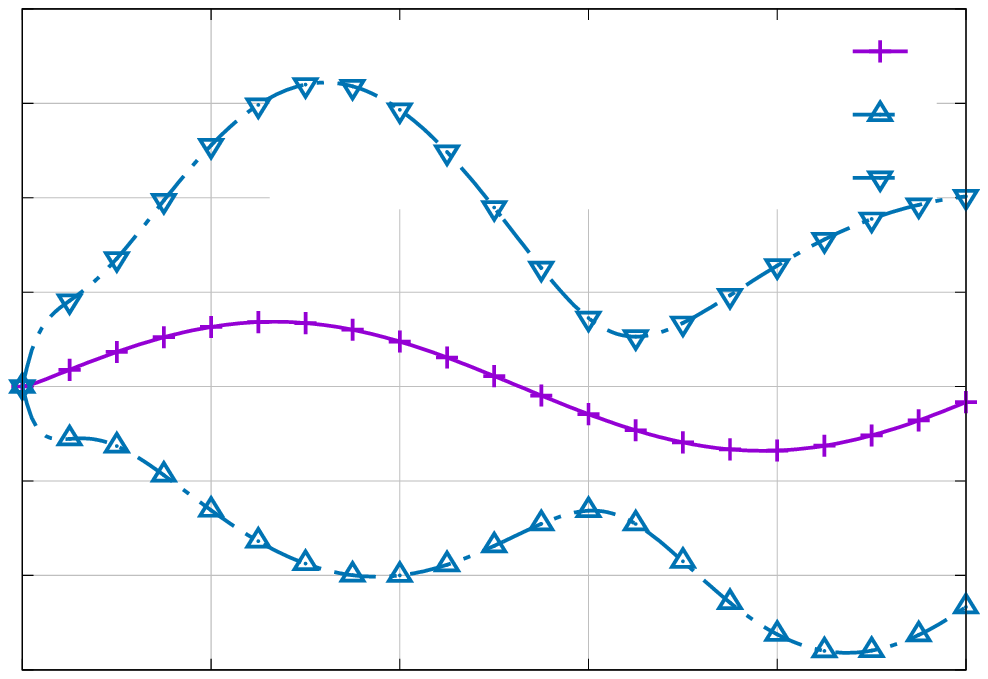}}%
    \gplfronttext
  \end{picture}%
\endgroup

%% file: convDiffAnormHard_Output2Basis5.tex
\begingroup
  \makeatletter
  \providecommand\color[2][]{%
    \GenericError{(gnuplot) \space\space\space\@spaces}{%
      Package color not loaded in conjunction with
      terminal option `colourtext'%
    }{See the gnuplot documentation for explanation.%
    }{Either use 'blacktext' in gnuplot or load the package
      color.sty in LaTeX.}%
    \renewcommand\color[2][]{}%
  }%
  \providecommand\includegraphics[2][]{%
    \GenericError{(gnuplot) \space\space\space\@spaces}{%
      Package graphicx or graphics not loaded%
    }{See the gnuplot documentation for explanation.%
    }{The gnuplot epslatex terminal needs graphicx.sty or graphics.sty.}%
    \renewcommand\includegraphics[2][]{}%
  }%
  \providecommand\rotatebox[2]{#2}%
  \@ifundefined{ifGPcolor}{%
    \newif\ifGPcolor
    \GPcolortrue
  }{}%
  \@ifundefined{ifGPblacktext}{%
    \newif\ifGPblacktext
    \GPblacktexttrue
  }{}%
  \let\gplgaddtomacro\g@addto@macro
  \gdef\gplbacktext{}%
  \gdef\gplfronttext{}%
  \makeatother
  \ifGPblacktext
    \def\colorrgb#1{}%
    \def\colorgray#1{}%
  \else
    \ifGPcolor
      \def\colorrgb#1{\color[rgb]{#1}}%
      \def\colorgray#1{\color[gray]{#1}}%
      \expandafter\def\csname LTw\endcsname{\color{white}}%
      \expandafter\def\csname LTb\endcsname{\color{black}}%
      \expandafter\def\csname LTa\endcsname{\color{black}}%
      \expandafter\def\csname LT0\endcsname{\color[rgb]{1,0,0}}%
      \expandafter\def\csname LT1\endcsname{\color[rgb]{0,1,0}}%
      \expandafter\def\csname LT2\endcsname{\color[rgb]{0,0,1}}%
      \expandafter\def\csname LT3\endcsname{\color[rgb]{1,0,1}}%
      \expandafter\def\csname LT4\endcsname{\color[rgb]{0,1,1}}%
      \expandafter\def\csname LT5\endcsname{\color[rgb]{1,1,0}}%
      \expandafter\def\csname LT6\endcsname{\color[rgb]{0,0,0}}%
      \expandafter\def\csname LT7\endcsname{\color[rgb]{1,0.3,0}}%
      \expandafter\def\csname LT8\endcsname{\color[rgb]{0.5,0.5,0.5}}%
    \else
      \def\colorrgb#1{\color{black}}%
      \def\colorgray#1{\color[gray]{#1}}%
      \expandafter\def\csname LTw\endcsname{\color{white}}%
      \expandafter\def\csname LTb\endcsname{\color{black}}%
      \expandafter\def\csname LTa\endcsname{\color{black}}%
      \expandafter\def\csname LT0\endcsname{\color{black}}%
      \expandafter\def\csname LT1\endcsname{\color{black}}%
      \expandafter\def\csname LT2\endcsname{\color{black}}%
      \expandafter\def\csname LT3\endcsname{\color{black}}%
      \expandafter\def\csname LT4\endcsname{\color{black}}%
      \expandafter\def\csname LT5\endcsname{\color{black}}%
      \expandafter\def\csname LT6\endcsname{\color{black}}%
      \expandafter\def\csname LT7\endcsname{\color{black}}%
      \expandafter\def\csname LT8\endcsname{\color{black}}%
    \fi
  \fi
    \setlength{\unitlength}{0.0500bp}%
    \ifx\gptboxheight\undefined%
      \newlength{\gptboxheight}%
      \newlength{\gptboxwidth}%
      \newsavebox{\gptboxtext}%
    \fi%
    \setlength{\fboxrule}{0.5pt}%
    \setlength{\fboxsep}{1pt}%
\begin{picture}(7200.00,5040.00)%
    \gplgaddtomacro\gplbacktext{%
      \csname LTb\endcsname%
      \put(924,896){\makebox(0,0)[r]{\strut{}$-0.4$}}%
      \csname LTb\endcsname%
      \put(924,1657){\makebox(0,0)[r]{\strut{}$-0.2$}}%
      \csname LTb\endcsname%
      \put(924,2419){\makebox(0,0)[r]{\strut{}$0$}}%
      \csname LTb\endcsname%
      \put(924,3180){\makebox(0,0)[r]{\strut{}$0.2$}}%
      \csname LTb\endcsname%
      \put(924,3942){\makebox(0,0)[r]{\strut{}$0.4$}}%
      \csname LTb\endcsname%
      \put(924,4703){\makebox(0,0)[r]{\strut{}$0.6$}}%
      \csname LTb\endcsname%
      \put(1092,616){\makebox(0,0){\strut{}$0$}}%
      \csname LTb\endcsname%
      \put(2213,616){\makebox(0,0){\strut{}$0.1$}}%
      \csname LTb\endcsname%
      \put(3333,616){\makebox(0,0){\strut{}$0.2$}}%
      \csname LTb\endcsname%
      \put(4454,616){\makebox(0,0){\strut{}$0.3$}}%
      \csname LTb\endcsname%
      \put(5574,616){\makebox(0,0){\strut{}$0.4$}}%
      \csname LTb\endcsname%
      \put(6695,616){\makebox(0,0){\strut{}$0.5$}}%
    }%
    \gplgaddtomacro\gplfronttext{%
      \csname LTb\endcsname%
      \put(3893,196){\makebox(0,0){\strut{}time}}%
      \csname LTb\endcsname%
      \put(5876,4458){\makebox(0,0)[r]{\strut{}output \eqref{eq:OutputIntegral}}}%
      \csname LTb\endcsname%
      \put(5876,4094){\makebox(0,0)[r]{\strut{}$\ROMOutput_k - \OutErrBnd_k$}}%
      \csname LTb\endcsname%
      \put(5876,3730){\makebox(0,0)[r]{\strut{}$\ROMOutput_k + \OutErrBnd_k$}}%
    }%
    \gplbacktext
    \put(0,0){\includegraphics{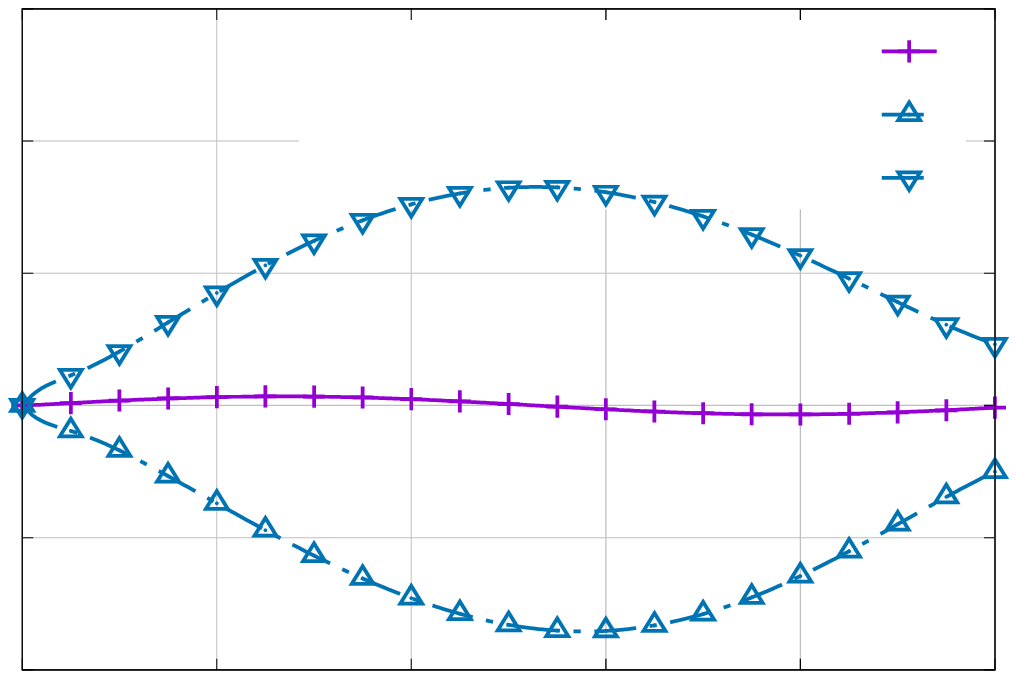}}%
    \gplfronttext
  \end{picture}%
\endgroup